\documentclass[a4paper,10pt,twoside]{article}

\usepackage[utf8]{inputenc}

\usepackage{anysize}
\usepackage{fancyhdr}  

\usepackage{amssymb}
\usepackage{amsthm}
\usepackage{amsmath}

\usepackage{url}

\usepackage{subfigure,graphicx} 
\usepackage{epstopdf}

\usepackage[dvipdfm]{hyperref}

\newcommand{\C}{\mathbb{C}}
\newcommand{\csE}{{\bf E}}
\newcommand{\cF}{\mathfrak{F}}
\newcommand{\N}{\mathbb{N}}
\newcommand{\R}{\mathbb{R}}
\newcommand{\Rplus}{\mathbb{R}^+}
\newcommand{\Rzeroplus}{\mathbb{R}_0^+}

\newcommand{\Z}{\mathbb{Z}}

\newcommand{\lb}{\{}
\newcommand{\rb}{\}}

\def\supp{{\text{\rm supp }}}

\newcommand{\starshaped}{\text{\sc Star}}
\newcommand{\stredge}{\text{\sc Straight}}

\newcommand{\smooth}{C^\beta([-1,1]^2;\nu)}
\newcommand{\cart}{\mathcal{E}^\beta([-1,1]^2;\nu)}
\newcommand{\cartbin}{\mathcal{E}_{bin}^\beta([-1,1]^2;\nu)}
\newcommand{\straight}{{\bf E}^\beta([-1,1]^2;\nu)}

\newcommand{\fqtilde}{f^{\raisebox{-0.45em}[0.4mm][0.4mm]{\textasciitilde}}_Q}
\newcommand{\djtilde}{d^{\raisebox{-0.45em}[0.4mm][0.4mm]{\textasciitilde}}_j}
\newcommand{\fjtilde}{f^{\raisebox{-0.45em}[0.4mm][0.4mm]{\textasciitilde}}_j}

\newcommand{\bessel}{\mathcal{J}_1}
\newcommand{\tile}{\mathcal{W}}
\newcommand{\tileext}{\mathcal{W}^+}
\newcommand{\tileint}{\mathcal{W}^-}

\newcommand{\tilefunc}{W}
\newcommand{\tilefuncint}{W^-}
\newcommand{\curvesys}{\mathfrak{C}_{s,\alpha}}
\newcommand{\curveind}{\mathbb{M}}
\newcommand{\circind}{\mathbb{J}_{\scriptscriptstyle{+}}}
\newcommand{\corona}{\mathcal{C}}

\makeatletter
\newcommand{\raisemath}[1]{\mathpalette{\raisem@th{#1}}}
\newcommand{\raisem@th}[3]{\raisebox{#1}{$#2#3$}}
\makeatother

\newcommand{\phiJ}{\varphi_{\hspace*{-1.5pt}\raisemath{-1.5pt}{J}}}
\newcommand{\phiJplus}{\varphi_{\hspace*{-1.5pt}\raisemath{-1.5pt}{J_+}}}
\newcommand{\phiJdelJ}{\varphi_{\hspace*{-1.5pt}\raisemath{-1.5pt}{J+\Delta J}}}

\newcommand{\ident}{\textsl{Id}}

\newtheorem{theorem}{Theorem}[section]
\newtheorem{prop}[theorem]{Proposition}
\newtheorem{lemma}[theorem]{Lemma}
\newtheorem{cor}[theorem]{Corollary}
\newtheorem{definition}[theorem]{Definition}
\newtheorem{remark}[theorem]{Remark}

\author{Martin Sch\"afer\footnote{\tt
Email:\ schaefer@math.tu-berlin.de}
\\\\
Institute of Mathematics, Technische Universit\"{a}t Berlin\\
Stra{\ss}e des 17.\ Juni 136, 10623 Berlin, Germany
}

\title{The Role of $\alpha$-Scaling for Cartoon Approximation}

\begin{document}

\maketitle

\begin{abstract}
The class of cartoon-like functions, classicly defined as piecewise $C^2$ functions
consisting of smooth regions separated by $C^2$ discontinuity curves,
is a well-established model for image data. The quest for frames
providing optimal approximation for this class has among others led to
the development of curvelets, contourlets, and shearlets.
Due to parabolic scaling, these systems are able to provide $N$-term approximations
converging with a quasi-optimal rate of order $N^{-2}$.
Replacing parabolic scaling by $\alpha$-scaling, one can construct
$\alpha$-curvelet and $\alpha$-shearlet frames which interpolate between
wavelet-type systems for $\alpha=1$, the classic parabolically scaled systems for $\alpha=\frac12$, and
ridgelet-type systems for $\alpha=0$.
Previous research shows that if $\alpha\in[\frac{1}{2},1)$ they
provide quasi-optimal approximation for cartoons of regularity $C^{1/\alpha}$ with a rate of order $N^{-1/\alpha}$.

In this work we continue the exploration of approximation properties of
$\alpha$-scaled representation systems, with the aim to better understand
the role of the parameter $\alpha$. Concerning $\alpha$-curvelets with $\alpha<1$,
we prove that the best possible $N$-term approximation rate achievable for cartoons with curved edges
is limited to at most $N^{-1/(1-\alpha)}$, independent of the smoothness of the cartoons.
The maximal rate that can be obtained by simple thresholding of the frame coefficients is even bounded by
$N^{-1/\max\{\alpha,1-\alpha\}}$.
Systems of $\alpha$-curvelets thus cannot take advantage of regularity higher than $C^{1/\alpha}$ if $\alpha\in[\frac{1}{2},1)$,
the rate of $N^{-1/\alpha}$ cannot be surpassed. For $C^\beta$ cartoons with $\beta\ge2$
the classic $\frac12$-curvelets provide the best performance with a rate of order $N^{-2}$, however below the optimal rate of order $N^{-\beta}$ if $\beta>2$.
In the range $\alpha\in[0,\frac{1}{2}]$ the achievable rate cannot exceed $N^{-1/(1-\alpha)}$ and deteriorates as $\alpha$ approaches $0$.

The approximation performance of $\alpha$-curvelets is different
if the edges of the cartoons are straight.
Assuming $C^\beta$ regularity, we establish an approximation rate of order $N^{-\min\{\alpha^{-1},\beta\}}$,
which improves as $\alpha$ tends to $0$. In the range $\alpha\in [0,\beta^{-1}]$ it is even quasi-optimal,
generalizing optimality results for ridgelets.
By applying the framework of $\alpha$-molecules, we finally extend the obtained results to other $\alpha$-scaled representation systems, including for instance
$\alpha$-shearlet frames.
\end{abstract}

\begin{center}{\it Keywords:} {Cartoon Images, Nonlinear Approximation, 
Wavelets, Curvelets, Shearlets, Ridgelets, Anisotropic Scaling, $\alpha$-Molecules.}
\end{center}

{\it MSC2000 Subject Classification:} {41A25, 41A30, 42C40.}

\section{Introduction}

In the age of `big data', efficient data representation
is an objective of an ever increasing importance. Not only does it simplify the handling of the data due to
the reduction of needed storage space or
the possible speed-up of processing times.
The knowledge of a `good' representation also gives valuable information about the structure of
the data itself, simplifying certain processing tasks or even just enabling them in the first place. As an example we
may think of the restoration of corrupted signals
or the separation of several superimposed signals of distinct types.

Often, the data of interest
can be modeled in a linear space, for instance a Hilbert space as exemplified by the Lebesgue spaces of square-integrable functions.
In this setting
the standard approach for the representation of a signal
is its expansion with respect to a fixed family
of basic elements,
a so-called dictionary for the data.
In practice, one usually needs to contend with approximations and therefore resorts to approximation schemes,
i.e., algorithms that deliver for each signal $f$ a sequence of approximants $(f_N)_{N\in\N}$ converging to the signal.
A standard choice here is to use $N$-term approximations in the respective dictionary, i.e., approximants being
built from just $N$ dictionary elements.

A main goal of approximation theory is the development of approximation schemes with
a best possible
speed of convergence, commonly
quantified by the asymptotic decay
of the approximation error $\|f-f_N\|$ as $N \to \infty$.
With regard to $N$-term approximations, the achievable rate is determined by
the utilized
dictionary in the background and one
aims to find
dictionaries
providing high approximation rates for the data.
Such dictionaries are
said to sparsely approximate the corresponding signals and clearly need to be chosen
depending on the considered signal class.
For efficient data representation it is therefore essential, first, to be able to precisely specify the type of data under consideration, e.g., in the form of an appropriate model,
and, second, to develop dictionaries, well adapted to the
specific data class, providing sparse approximations.

\subsection{Approximation of Image Data}

Subsequently, we are interested in the sparse approximation of image data. In our investigation, we will always stay in the continuum setting, where
images
are as usual represented as functions supported on some compact image domain $\Omega\subset\R^2$ with values containing pixel information at the respective positions, such as e.g.\ color or brightness information.
Being compactly supported and bounded,
the image data can conveniently be
modeled as a subset of the Hilbert space $L^2(\Omega)$, which in turn is considered as a subspace of $L^2(\R^2)$.
Hence, we are in a concrete Hilbert space scenario and can resort to the methodology described above, i.e., we aim for appropriate image models and sparsifying dictionaries.

For the space $L^2(\Omega)$ the classic Fourier systems constitute an orthonormal basis, providing a straight-forward procedure for representation.
However, Fourier systems work well only if the functions under consideration are smooth. For general images such smoothness assumptions are certainly not fulfilled.

As another popular representation system
wavelets~\cite{Dau92,Mal08} come to mind. Nowadays, they are one of the most widely used
systems in applied harmonic analysis, with various applications ranging from signal compression
(e.g.\ JPEG2000~\cite{CSE2000}) and restoration~\cite{CDOS12}
to PDE solvers~\cite{CDD01}.
In particular, they have the ability to
sparsely approximate
functions, which are smooth apart from isolated point singularities.
For general image data, however, such regularity assumptions are still too strict.
A characteristic feature of images are edges, leading to curvilinear discontinuities in the data.
With respect to such
line singularities, wavelet systems do not perform optimally any more.
The isotropy of their scaling prohibits an optimal
resolution of these kind of anisotropic structures.

With the desire
to specifically model
the occurrence of edges, the concept of cartoon-like functions emerged.
These are piecewise smooth functions featuring discontinuities along lower-dimensional manifolds, in our case along the $1$-dimensional edge curves of the image.
Based on such functions, suitable models for natural images have been conceived and
different model classes have been introduced. Typically, these classes are characterized by a specific smoothness
of the regions and by certain conditions on the separating edges. As examples, let us mention the classic cartoons~\cite{CD04} with $C^2$ regularity of the regions and the discontinuity curves,
or the horizon classes considered e.g.\ in \cite{D99,CWBB04c,PM05}.


The achievable approximation rate for a class of cartoon-like functions
essentially depends on the regularity of the cartoons, including
both the smoothness of the edge curves
and the smoothness of the regions in between. It was shown in \cite{PM05,PennecM} that $C^\beta$ regularity of the regions and the separating edges with $\beta>0$ allows
for an asymptotic rate of order $N^{-\beta}$.
By information theoretic arguments, it has further been established that this rate cannot be surpassed \cite{Don01}, at least in a class-wise sense.
Interestingly, the benchmark $N^{-\beta}$
is the same for the class of so-called binary cartoons, i.e., cartoon-like functions with
constant regions, and it also does not change if one restricts to
$C^\beta$ smooth functions without any edges.


With the model of cartoon-like functions at hand,
let us turn again to the question of efficient image representation.
In the past, a great amount of energy has been devoted to the effort of constructing
dictionaries well-suited for cartoon approximation.
Thereby, many different paths have been pursued
and the developed methods can be divided into two general categories: adaptive and nonadaptive methods.

Adaptive methods are by nature more flexible and have the inherent advantage of being able to adjust to the given data.
On the downside, the increased flexibility typically comes at the cost
of higher computational complexity of the employed approximation and reconstruction schemes.
Some prominent examples of adaptive methods
for image data are based on wedgelet dictionaries~\cite{D99} and their higher-order relatives, so-called surflets~\cite{CWBB04b,CWBB04c}.
They have been shown to reach the optimality bound $N^{-\beta}$ for binary cartoons with $C^\beta$ regularity~\cite{CWBB04tech,CWBB04proc}.
Other notable dictionaries used for adaptive approximation
include beamlets~\cite{DH00}, platelets~\cite{WN03}, and derivatives of wedgelets such as multiwedgelets~\cite{L13} or smoothlets~\cite{L11}.
More recently, new adaptive schemes have emerged that use bases, e.g., bandelets~\cite{PM05}, grouplets~\cite{M09}, and tetrolets~\cite{K09}.
Quasi-optimal approximation for $C^\beta$ cartoons with $\beta>0$ has been proved e.g.\ in~\cite{PennecM} for bandelets.

Nonadaptive methods are usually much simpler than adaptive schemes, at least from an algorithmic perspective. Mainly, they are based on frames and
the corresponding
reconstruction formulas. An easy path to approximation is thus provided by
simply thresholding
the frame coefficients.
Surprisingly, despite the simplicity of such schemes, there exist frames with
quasi-optimal approximation performance
for certain cartoon classes.

If the edges of the cartoons are
straight, different variants of so-called
ridgelet frames have been shown to yield quasi-optimal approximation~\cite{C99,GO15,GOtech16}.
Originally, the notion of a ridgelet was
introduced by Cand{\`e}s~\cite{Can98} in 1998, who defined them as bivariate ridge functions obtained by tensoring a univariate wavelet with a constant.
Since these `pure ridgelets' are not square-integrable, the concept was later modified in order to obtain frames or bases for $L^2(\R^2)$.
By giving them a slow decay along the ridge, Donoho constructed an orthonormal basis
whose elements are called `orthonormal ridgelets'~\cite{D98}. Their close relationship to the original concept has been analyzed in~\cite{DonXX}.
Another construction, based on directional scaling,
goes back to Grohs, providing tight frames \cite{GrohsRidLT}.
This kind of construction coincides with the concept of `$0$-curvelets' presented below.

To deal with curved edges, numerous types of frames
have been developed.
An important
milestone was the introduction of the first generation of curvelets~\cite{CD2000} by Cand{\`e}s and Donoho in 1999.
They represent the first frame to reach the optimal approximation order
of $N^{-2}$ for $C^2$ cartoons via simple thresholding.
A modification of this system, the second generation of curvelets~\cite{CD04}, was introduced in 2002 by the same authors.
It is based on a more elegant and simpler construction principle, yet features
the same quasi-optimal approximation properties.
Following this early breakthrough, other constructions better suited for digital implementation
were developed. Let us mention contourlets~\cite{DV05} by Do and Vetterli and shearlets, whose construction goes back mainly to Guo, Kutyniok, Labate, Lim, and Weiss.
The first shearlet construction consisted of band-limited functions and was
presented in~\cite{KuLaLiWe,GKL05}. Later, more sophisticated shearlet systems were developed, such as e.g.\ the well-localized band-limited Parseval frame in \cite{Guo2012a}
or even systems of compactly supported shearlets~\cite{Kittipoom2010}.
Like curvelets, those systems provide quasi-optimal approximation for $C^2$ cartoons.
For the classic band-limited shearlets this was established in~\cite{GL07}, for those with compact support in~\cite{Kutyniok2010}.

A common principle underlying the above constructions is parabolic scaling, a
type of scaling optimally adapted to
$C^2$ singularity curves. It is essential for the quasi-optimal approximation of $C^2$ cartoons
and led to the notion of parabolic molecules~\cite{Grohs2011}. This concept unifies various parabolically scaled systems under one roof, in particular
the classic curvelet and shearlet systems,
and is the predecessor of the more general framework of $\alpha$-molecules~\cite{GKKS15}.

\subsection{Multiscale Systems based on $\alpha$-Scaling}

Comparing the approximation properties of wavelets, curvelets, and ridgelets,
a distinct behavior with respect to their ability to resolve edges is characteristic.
Ridgelets are optimally adapted to straight edges, curvelets are optimal for $C^2$ line singularities, and wavelets for point singularities.
This distinct behavior is due to the different
scaling laws underlying their respective constructions: isotropic scaling for wavelets, parabolic scaling for curvelets, and directional scaling for ridgelets.

Introducing a parameter $\alpha\in\R$ and associated $\alpha$-scaling matrices
\begin{align}\label{eq:alphamat1}
A_{\alpha,s}=\begin{pmatrix}
s & 0 \\ 0 & s^\alpha
\end{pmatrix}, \qquad s>0,
\end{align}
one can interpolate between these different types of scaling
and construct corresponding $\alpha$-scaled representation systems.
Incorporating $\alpha$-scaling in the original construction of curvelets, for instance,
yields so-called $\alpha$-curvelets~\cite{GKKScurve2014}. For $\alpha\in[0,1]$, they constitute a family of systems which encompass
ridgelets (in the sense of \cite{GrohsRidLT}) for $\alpha=0$, the classic curvelets for $\alpha=\frac12$, and wavelets for $\alpha=1$.
In a similar fashion, $\alpha$-shearlet systems~\cite{Kei13,Kutyniok2012correct} can be obtained by modifying the classic shearlet constructions.

A natural question concerning such $\alpha$-scaled systems is how their approximation properties
are affected by a change of the parameter $\alpha$.
With regard to cartoon approximation, this question has been pursued in \cite{GKKScurve2014} for $\alpha$-curvelet frames and in \cite{Kei13,Kutyniok2012correct} for $\alpha$-shearlet frames.
It was shown that, if $\alpha\in[\frac{1}{2},1)$ and if the cartoon $f$ is of regularity $C^\beta$ with $\beta=\alpha^{-1}$, simple thresholding of the coefficients
yields $N$-term approximations $f_N$ with a convergence of
\begin{align}\label{intro-rate}
\|f-f_N\| \lesssim N^{-\beta} \log(N)^{1+\beta} \quad\text{as }N\to\infty,
\end{align}
which apart from the log-factor is optimal.
Later, these results were further extended
utilizing the theory of $\alpha$-molecules~\cite{GKKS15}. This is a framework providing a unified approach to $\alpha$-scaled systems,
based solely on assumptions on the time-frequency localization of the respective functions.
It allows
to transfer approximation results obtained for one system of $\alpha$-molecules to other
systems, under certain consistency conditions. In particular, the rate~\eqref{intro-rate} for $\alpha$-curvelets
was generalized (in a weak form) to other $\alpha$-scaled representation systems~\cite{GKKS15}, which all achieve a rate of $N^{-\beta+\varepsilon}$ with $\varepsilon>0$ arbitrarily small.

Despite these results, many questions concerning $\alpha$-scaled representation systems and their ability to approximate cartoon-like functions
remain open,
e.g., their performance
in the range $\alpha<\frac12$ or their suitability for the approximation of straight edges.
In this research we want to address these open questions, shedding (even) more light
on the role of the parameter $\alpha$.

\subsection{Outline and Contribution}

Our exposition starts with a short review of $\alpha$-scaled systems in Section~\ref{sec:curvelets}, where also a specific construction of an $\alpha$-curvelet frame for $L^2(\R^2)$
is presented. This frame, denoted by $\curvesys$, will serve as a prototypical system whose properties have ramifications for other $\alpha$-scaled systems, such as for example $\alpha$-shearlets,
due to the transference principle of the framework of $\alpha$-molecules.

In the main part of the article, Sections~\ref{sec:cartoon} and \ref{sec:straight}, we analyze the $N$-term approximation properties of
the frame $\curvesys$ with regard to different classes of cartoon images.
In Section~\ref{sec:cartoon} we start with cartoons with curved edges and first introduce corresponding signal classes of $C^\beta$ regularity for $\beta\in[0,\infty)$.
Theorem~\ref{thm:benchmark} recalls $N^{-\beta}$ as the order of the maximal achievable approximation rate for such $C^\beta$ cartoons, which
cannot be surpassed by any polynomial-depth restricted $N$-term approximation scheme, independent of the utilized dictionary.

Then we recall the quasi-optimal approximation~\eqref{intro-rate} of $\alpha$-curvelets, proved in~\cite{GKKScurve2014},
if $\alpha\in[\frac{1}{2},1)$ and $\beta=\alpha^{-1}$.
Our main findings in Section~3, Theorems~\ref{thm:bound1} and~\ref{thm:bound2}, extend and complement this result.
Theorem~\ref{thm:bound1} shows that the best possible $N$-term approximation rate achievable by
$\curvesys$ for cartoons with curved edges is limited to at most $N^{-\frac{1}{1-\alpha}}$, where $\alpha<1$ and the smoothness of the cartoons is arbitrary.
Moreover, according to Theorem~\ref{thm:bound2}, the achievable rate cannot
exceed $N^{-\frac{1}{\max\{\alpha,1-\alpha\}}}$ if a simple thresholding scheme is used.

These bounds
show that $\alpha$-curvelets with $\alpha\in[\frac{1}{2},1)$ cannot take advantage of regularity higher than $C^{1/\alpha}$.
Furthermore, they prohibit optimal approximation of $C^\beta$ cartoons if $\beta>2$, since decreasing $\alpha$ beyond $\frac12$ deteriorates the achievable rates compared to the classic curvelets.
Hence, with a rate of order $N^{-2}$, these provide the best performance among all $\alpha$-curvelet systems, if the regularity of the cartoons is at least $C^2$ and curved singularities are involved.
As a consequence, no curvelet system
can reach the optimality bound $N^{-\beta}$ if $\beta>2$.
In fact, up to now, no frame construction is known where a nonadaptive approximation scheme can break this $N^{-2}$ barrier and
the quest for such frames remains open.

In Section~\ref{sec:straight} we consider
cartoons featuring only straight edges.
For the corresponding classes of regularity $C^\beta$ the same optimality benchmark $N^{-\beta}$ holds true
as for the cartoons with curved edges. Our main result of Section~\ref{sec:straight}, Theorem~\ref{thm:mainappr1}, shows that a simple thresholding scheme for the $\alpha$-curvelet frame $\curvesys$
yields approximation rates of order $N^{-\min\{\alpha^{-1},\beta\}}$.
Hence, here a smaller $\alpha$ is beneficial and even ensures quasi-optimal approximation if $\alpha\in[0,\beta^{-1}]$. This finding generalizes
earlier results for ridgelets.

We finish with a short discussion of our results in Section~\ref{sec:discussion}. In particular, we point out some ramifications for other $\alpha$-scaled representation systems,
utilizing the framework of $\alpha$-molecules. All $\alpha$-scaled systems which are frames and in a certain sense consistent with $\curvesys$
feature similar properties, formulated in Theorem~\ref{thm:mol_app}
and Corollary~\ref{cor:mol_app}.

Some useful properties of Bessel functions needed in Section~3 are collected in the appendix.

\subsection{Notation}

Before we begin, let us fix some
general notation. Writing $\N$ we will refer to the natural numbers without zero, and we let $\N_0:=\N\cup\{0\}$.
As usual, $\Z$, $\R$ and $\C$ denote the integer, real and complex numbers. Further,
we put $\Rzeroplus:=[0,\infty)$ and $\Rplus:=(0,\infty)$.
We also introduce the `floor' and `ceiling' of $t\in\R$,
$\lfloor t \rfloor:= \max\{ n\in\Z : n\le t \}$ and $\lceil t \rceil:= \min\{ n\in\Z : n\ge t \} $.
The symbol $\mathbb{T}$ is used for the torus obtained from the interval $[0,2\pi]$ by
identifying the endpoints.
The unit-circle in $\C\simeq\R^2$ is denoted by $\mathbb{S}^1$.

The vector space $\R^d$, $d\in\N$, is equipped with the Euclidean scalar product $\langle \cdot,\cdot \rangle$ and associated norm $|\cdot|$.
The notation $|\cdot|_p$, $p\in(0,\infty]$, is used for the $p$-(quasi-)norms on $\R^d$.
For a multi-index $m=(m_1,\ldots,m_d)\in\N_0^d$,
$\partial^{m}:=\partial_1^{m_1}\cdots \partial_d^{m_d}$ is a differential operator with $\partial_i$, $i\in \{1,\dots, d\}$, the partial derivative in the $i$-th coordinate direction.
Given a vector $x=(x_1,\ldots,x_d)\in\R^d$,
we further define $x^m:=x_1^{m_1}\cdots x_d^{m_d}$ (with the convention $0^0:=1$).

If $A(\omega)\le C B(\omega)$ holds true for two quantities $A,B\in \R$ depending on a set of parameters $\omega$ with a uniform constant $C>0$,
we write $A\lesssim B$ or equivalently $B\gtrsim A$. If both, $A\lesssim B$ and $B\lesssim A$,
hold true, we denote this by $A \asymp B$.

For measurable subsets $\Omega\subseteq\R^d$ we let
$L^p(\Omega)$, $p\in(0,\infty]$, denote the usual Lebesgue spaces with respect to the Lebesgue measure.
The corresponding (quasi-)norms are denoted by $\|\cdot\|_{L^p(\Omega)}$, in case $\Omega=\R^d$ we abbreviate $\|\cdot\|_p:=\|\cdot\|_{L^p(\R^d)}$.
For the scalar product on $L^2(\Omega)$ the same notation $\langle \cdot,\cdot \rangle$ as for the Euclidean product on $\R^d$ is used.
The Lebesgue sequence spaces, for a discrete index set $\Lambda$, are denoted by $\ell^p(\Lambda)$ with associated (quasi-)norms $\|\cdot\|_{\ell^p}$.
The definition of their weak counterparts 
$w\ell^p(\Lambda)$, equipped with (quasi-)norms $\|\cdot\|_{w\ell^p}$, are recalled in Section~\ref{sec:straight}.

The space $C_{\rm loc}^\beta(\R^d)$, for an integer $\beta\in\N_0\cup\{\infty\}$, shall comprise
all continuous real-valued functions on $\R^d$, whose
classic derivatives up to order $\beta\in\N_0$ exist.
For $\beta\in[0,\infty)$ we then define
\[
	C^{\beta}(\R^d):=
	\Big\{f\in C_{\rm loc}^{\lfloor \beta \rfloor}(\R^d):\, \|f\|_{C^\beta(\R^d)}:= \|f\|_{C^{\lfloor\beta\rfloor}(\R^d)} +
	\sum_{|m|_1=\lfloor\beta\rfloor}
	\text{\sl Höl}(\partial^m f, \beta- \lfloor\beta\rfloor) <\infty
	\Big\}\,,
\]
where $\|f\|_{C^{\lfloor \beta \rfloor}(\R^d)} := \sum_{|m|_1\le \lfloor\beta\rfloor} \displaystyle{\sup_{x\in\R^d} |\partial^mf(x)| }$ and
the H\"older constant of exponent $\alpha\in[0,1]$ is given by
\[
\text{\sl H\"{o}l}(f,\alpha):=\sup_{x,y\in \R^d} \frac{|f(x)-f(y)|}{|x-y|^{\alpha}}.
\]
The notation $C_0^\beta(\overline{\Omega})$, for some open subset $\Omega\subseteq\R^d$, is used for
functions $f\in C^\beta(\R^d)$ whose support $\supp f$ is compact and contained in the closure $\overline{\Omega}$ of $\Omega$.
Frequently, we also need to measure functions $f\in C_{\rm loc}^{\beta}(\R^d)$, $\beta\in\N_0$, with the following Sobolev norms, where $p\in[1,\infty]$,
\[
\|f\|_{\beta,p}:=\| f \|_{W^{\beta,p}(\R^d)} := \sum_{|m|_1\le \beta} \| \partial^mf \|_{L^p(\R^d)}.
\]

\noindent
Finally, we will use the following version of the Fourier transform. For a Schwartz function
$f\in\mathcal{S}(\R^d)$
\[
\mathcal{F}f(\xi):= \int_{\R^d} f(x) \exp(-2\pi i \langle x,\xi \rangle) \,dx\,, \quad\xi\in\R^d\,.
\]
As usual, $\mathcal{F}$ is extended to the tempered distributions $\mathcal{S}^\prime(\R^d)$, and
we often write $\widehat{f}$ for $\mathcal{F}f$.


\section{The Anchor System: $\alpha$-Curvelets}
\label{sec:curvelets}


Directional multi-scale systems
based on $\alpha$-scaling feature a characteristic tiling of the frequency domain.
The multi-scale structure is reflected by a partition of the Fourier plane into dyadic coronae,
further divided into wedge-like tiles, where the energy of the system elements is concentrated.
In case of 
inhomogeneous systems, a ball around the origin corresponds to the low-frequency base scale.

A prototypical instance of such an $\alpha$-scaled system is
the frame $\curvesys$ of $\alpha$-curvelets, thoroughly defined in this section.
It is prototypical in the sense that many of its properties transfer -- via the framework of $\alpha$-molecules~\cite{GKKS15} -- to other $\alpha$-scaled systems.
Among these are other $\alpha$-curvelet constructions~\cite{CD04,GKKScurve2014}, but also band-limited~\cite{KuLaLiWe,GKL05,Guo2012a} as well as compactly supported~\cite{Kittipoom2010,Kei13,Kutyniok2012correct} $\alpha$-shearlet systems.
This fact gives the system $\curvesys$ a special significance for our purpose and motivates its detailed discussion here.

Before defining $\curvesys$, which is similar to the construction of $\alpha$-curvelets in \cite{GKKScurve2014},
let us first elaborate the geometric aspects of the corresponding frequency tiling.
At scales $j\ge1$ we have the coronae
\begin{align}\label{eqdefcorC}
\corona_j:=\Big\{ \xi\in\R^2 ~:~ C2^{s(j-1)}\le|\xi|_2\le C2^{s(j+1)}\Big\},
\end{align}
where $s>0$ is a fixed parameter and $C>0$ is a constant, specified conveniently later.
These coronae are each uniformly divided into
an even number of wedges, whose angular width at scale $j$ is given by the angle
\begin{align}\label{eqdef:fundangle}
\varphi_j:=\pi 2^{-\lfloor js(1-\alpha) \rfloor-1}
\end{align}
and depends on another parameter $\alpha\in(-\infty,1]$.
The approximate size of the resulting wedges correlates with an $\alpha$-scaled rectangle of dimension $2^{js}\times 2^{js\alpha}$.
By combining opposite wedges to wedge pairs, we obtain the tiles for the scales $j\ge1$.
There is only one tile associated with the base scale $j=0$, the low frequency ball $\corona_0:=\{ \xi\in\R^2 ~:~ |\xi|_2 \le C 2^{s} \}$.

For convenience, let us also introduce the angle $\varphi_0:=\pi$.
According to the above construction,
at each scale $j\in\N_0$ the number of tiles $L_j$ is given by
\begin{align}\label{eq:Lj}
L_0:=\pi\varphi_0^{-1}=1 \quad\text{and}\quad L_j:=\pi\varphi_j^{-1}=2^{\lfloor js(1-\alpha) \rfloor + 1}\,, \quad j\ge1.
\end{align}
In the following, the individual tiles will be denoted by $\tile_{j,\ell}$ and indexed by the set
\[
\mathbb{J}:=\big\{ (j,\ell) ~:~ j\in\N_0,\, \ell\in \{-L^-_j, \ldots, L^+_j \} \big\}
\]
with $L_j^-:=\lfloor L_j/2 \rfloor$ and $L_j^+:=\lceil L_j/2 \rceil -1$.
Hereby we let $\mathcal{W}_{0,0}:=\corona_0$, and in each corona $\corona_j$ with $j\ge1$ the wedge-pair $\tile_{j,0}$ shall be aligned horizontally, i.e.,
\begin{align*}
\tile_{j,0}:=\Big\{ \xi=(\xi_1,\xi_2)\in\corona_j ~:~ |\xi_1| \ge \cos(\varphi_j/2) |\xi|_{2} \Big\}.
\end{align*}
The remaining tiles $\tile_{j,\ell}$, $\ell\neq0$, are obtained via rotations of
$\tile_{j,0}$ by integer multiples $\varphi_{j,\ell}:=\ell\varphi_j$ of the angle $\varphi_j$ defined in~\eqref{eqdef:fundangle}. Hence,
$\tile_{j,\ell}:=R^{-1}_{j,\ell}\tile_{j,0}$ with rotation matrix
\begin{align}\label{eq:matrixrot}
R_{j,\ell}:=R_{\varphi_{j,\ell}} \,,\quad \text{where}\quad R_{\varphi}
:=\begin{pmatrix} \cos(\varphi) & -\sin(\varphi) \\
 \sin(\varphi) & \,\cos(\varphi)
\end{pmatrix}
\,,\quad \varphi\in\R.
\end{align}
The resulting tiling of the Fourier domain is schematically depicted in Figure~\ref{fig:freq_domain} (a).

We remark that in contrast to \cite{GKKScurve2014}, where $\alpha\in[0,1]$,
we allow $\alpha\in(-\infty,1]$ in the $\alpha$-curvelet construction.
This range is natural for the considered inhomogeneous systems. If $\alpha>1$, the number of tiles $L_j$ in each corona decreases with rising scale, and eventually
$L_j=1$. Thus, at high scales, those systems would
behave like isotropically scaled systems with $\alpha=1$.

\begin{figure*}[ht]
 \centering
 \subfigure{
 \includegraphics[width = .45\textwidth]{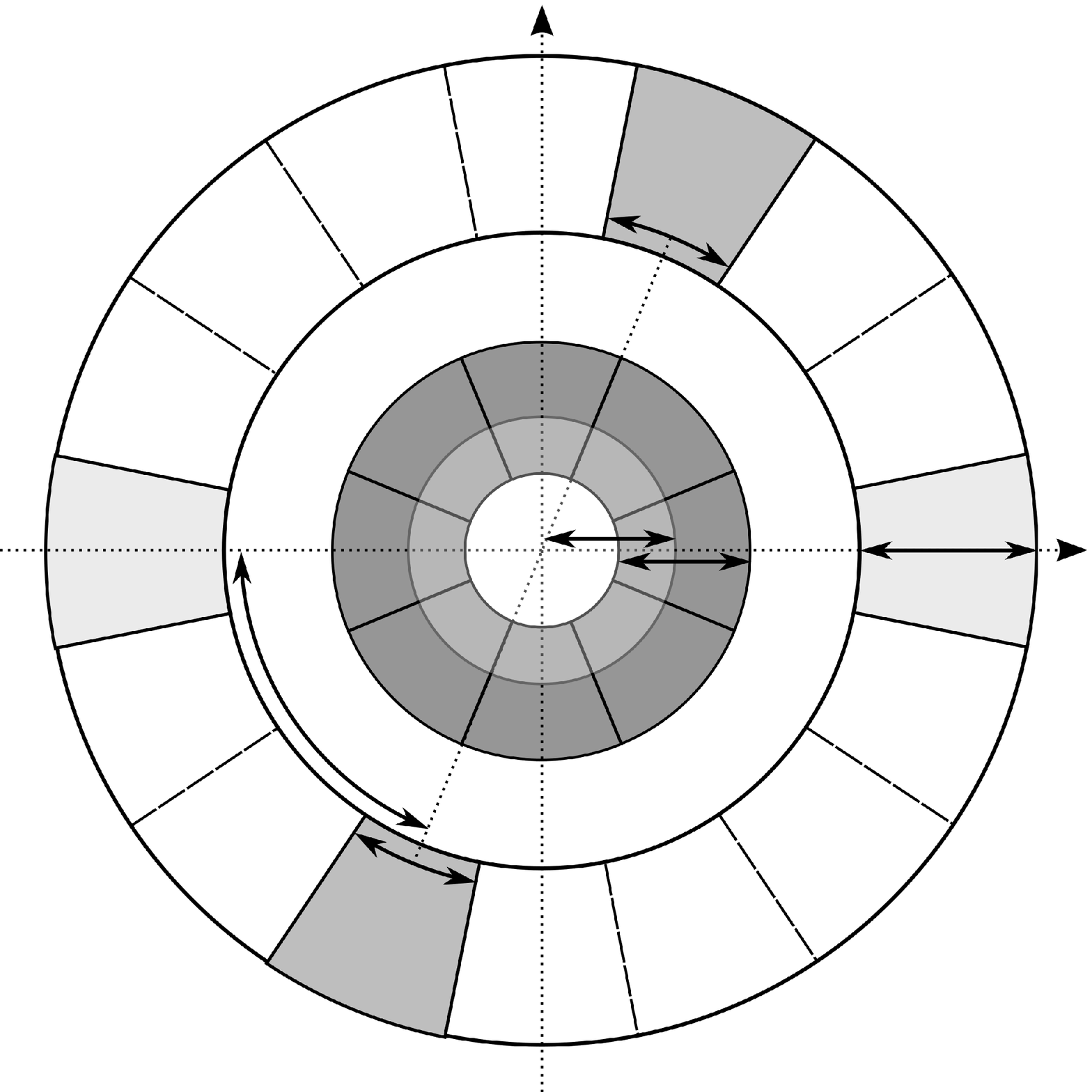}
 \put(-114,105){$\corona_{0}$}
 \put(-101,94.5){$\mathcal{I}_{0}$}
 \put(-114,131){$\corona_1$}
 \put(-77,92){$\mathcal{I}_{1}$}
 \put(-111,147){$\vdots$}
 \put(-114,185){$\corona_j$}
 \put(-32,107){$\tile_{j,0}$}
 \put(-28,93){$\mathcal{I}_{j}$}
 \put(-84,180){$\tile_{j,\ell}$}
 \put(-134,36){$\varphi_j$}
 \put(-149,72){$\varphi_{j,\ell}$}
 \put(-109,-15){(a)}
 }
 \:
 \subfigure{
 \includegraphics[width = .45\textwidth, clip=true]{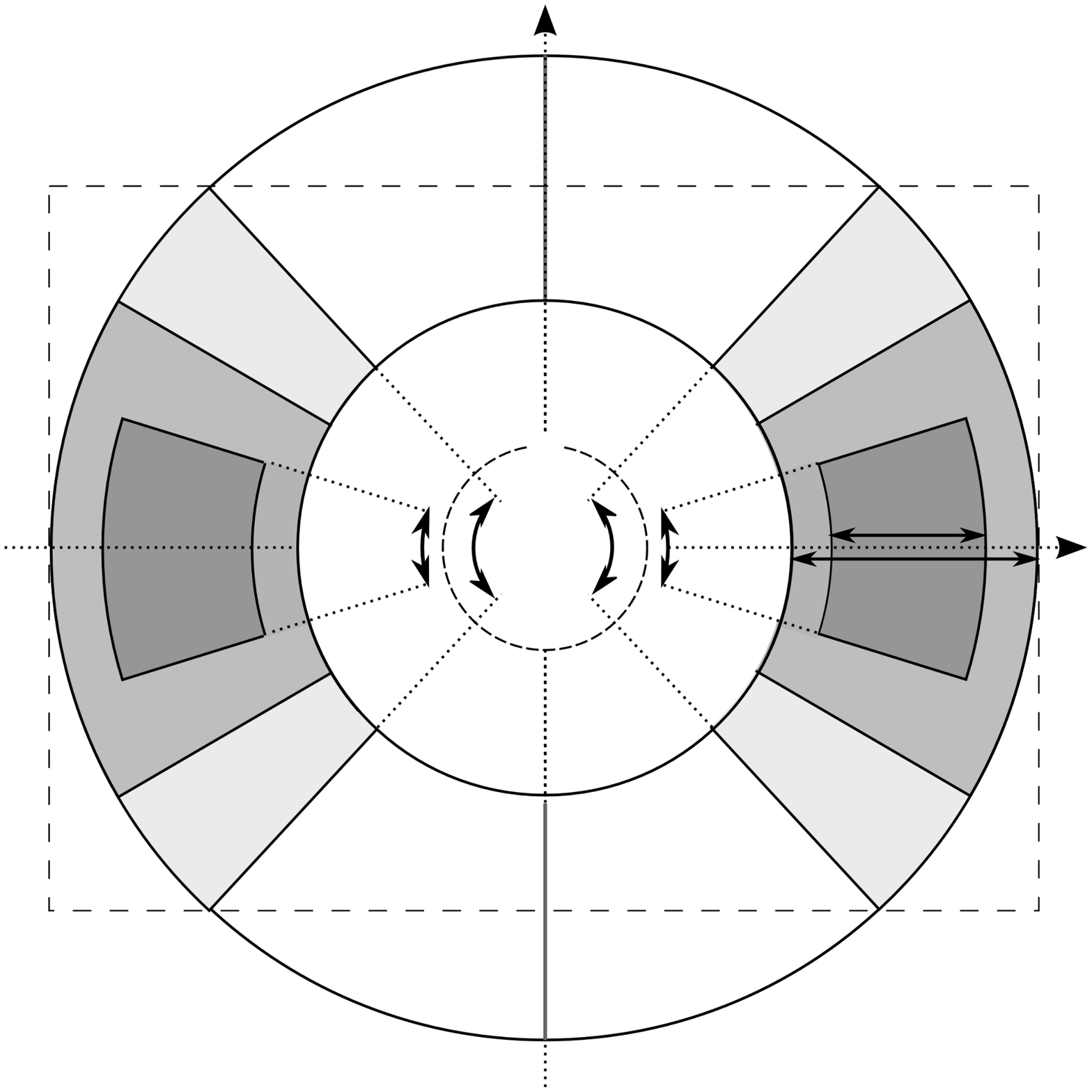}
 \put(-186,133){$\tile_{j,0}$}
 \put(-174,152){$\mathcal{W}_{j,0}^{\boldsymbol{+}}$}
 \put(-181,110){$\mathcal{W}_{j,0}^{\boldsymbol{-}}$}
 \put(-27,160){$\Xi_{j,0}$}
 \put(-99,185){$\corona_j$}
 \put(-106,118){$\boldsymbol{\mathbb{S}}^1$}
 \put(-111,100){$\mathcal{A}_{j,0}$}
 \put(-146,100){$\mathcal{A}^{\boldsymbol{-}}_{j,0}$}
 \put(-36,111){$\mathcal{I}^{\boldsymbol{-}}_{j}$}
 \put(-40,91.2){$\mathcal{I}_{j}$}
 \put(-108,-15){(b)}
 }
 \caption{(a): Tiling of Fourier domain into coronae $\corona_j$ and wedges $\tile_{j,\ell}$.
 (b): Schematic display of the frequency support of a wedge function $W_{j,0}$.}
 \label{fig:freq_domain}
 \end{figure*}

\subsection{The Frame of $\alpha$-Curvelets $\curvesys$}

Let us now turn to the actual
construction of the $\alpha$-curvelet frame $\curvesys$.
To realize the described frequency tiling, smooth functions
$W_J:\R^2\to\C$, $J\in\mathbb{J}$, are used, with compact support approximately given by the tiles $\tile_{J}$.
It is convenient to construct them
as tensor products of a radial and an angular component.
This allows to realize the desired
support separately on the ray $\Rzeroplus=[0,\infty)$ and on the circle $\mathbb{S}^{1}\subset\R^2$.
Projecting the coronae $\corona_j$ onto the ray $\Rzeroplus$ yields the intervals
\begin{align}\label{eqdef:dyintervals}
\mathcal{I}_{0}:=C\cdot[0,2^s] \quad\text{and}\quad\mathcal{I}_{j}:=C\cdot[2^{s(j-1)},2^{s(j+1)}] \,, \quad j\ge1.
\end{align}
For the radial subdivision, we thus utilize nonnegative smooth functions $U_j\in C^\infty(\Rzeroplus)$, $j\in\N_0$, which satisfy
the support condition $\supp\, U_j\subseteq \mathcal{I}_{j}$ and for $r\in\Rzeroplus$
\begin{align}\label{eq:coronalpartition}
A_1\le \sum_{j\ge 0} U^2_j(r) \le B_1 \quad\text{with constants}\quad 0<A_1\le B_1<\infty.
\end{align}
More concretely, we assume that the functions $U_j$, $j\ge1$, are generated by a single function $U\in C^\infty(\Rzeroplus,[0,1])$ via
$U_j(\cdot):=U(2^{-js} \cdot)$ and that there are $1<\tau_1<\tau_2<2^s$ such that
\begin{align}\label{eq:suppcond}
\begin{aligned}
\supp\, U_0\subseteq C\cdot[0,\tau_2],& \quad \sqrt{A_1}\le U_0 \le \sqrt{B_1}  \text{ on } C\cdot[0,\tau_1], \\
\supp\, U\subseteq C\cdot[2^{-s}\tau_1,\tau_2],& \quad \sqrt{A_1}\le U \le\sqrt{B_1} \text{ on } C\cdot[2^{-s}\tau_2,\tau_1].
\end{aligned}
\end{align}
Such functions exist and can even
be constructed with $A_1=B_1=1$ in \eqref{eq:coronalpartition}.

For the angular subdivision, we construct at each scale $j\in\N_0$ a smooth partition
on the unit circle $\mathbb{S}^{1}\subset\R^2$, reflecting the angular support of the tiles $\tile_{j,\ell}$.
We start with a function $\widetilde{V}\in C^\infty(\R,[0,1])$ with the properties
\begin{align*}
\supp\, \widetilde{V}\subseteq [-\textstyle{\frac{3}{4}}\pi,\textstyle{\frac{3}{4}}\pi] ,\quad
\sqrt{A_2}\le \widetilde{V} \le\sqrt{B_2} \text{ on }[-\textstyle{\frac{\pi}{4}},\textstyle{\frac{\pi}{4}}],\quad A_2\le\sum_{k\in\Z} \widetilde{V}^2(\cdot -k\pi)\le B_2,
\end{align*}
where $0<A_2\le B_2<\infty$.
Scaling then gives rise to the functions $\widetilde{V}_j(\cdot):=\widetilde{V}(L_j \cdot)\in C^\infty(\R,[0,1]) $ for $j\in\N_0$.
Via the bijection $t\mapsto e^{it}$ these functions yield functions
$\widetilde{V}_{j,0}\in C^{\infty}(\mathbb{S}^{1},[0,1])$ on the unit circle. We symmetrize
\[
V_{j,0}(\xi):=\widetilde{V}_{j,0}(\xi) + \widetilde{V}_{j,0}(-\xi), \quad\xi\in\mathbb{S}^1,
\]
and note that $\sqrt{A_2}\le V_{0,0}\le\sqrt{B_2}$ on $\mathbb{S}^{1}$.
Applying the rotation~\eqref{eq:matrixrot} then yields functions
$V_{j,\ell}(\cdot):=V_{j,0}(R_{j,\ell} \cdot)$ for every $J=(j,\ell)\in\mathbb{J}$, which satisfy
$A_2\le \sum_{|J|=j} V^2_{J}(\xi)\le B_2$ for all $\xi\in\mathbb{S}^{1}$. Here we use the notation $|J|:=j$ for $J=(j,\ell)\in\mathbb{J}$.

Finally, we are ready to define the wedge functions $W_{j,\ell}\in C^\infty(\R^2)$ as the polar tensor products
\begin{align}\label{eq:suppfunctions}
W_{j,\ell}(\xi):=U_j(|\xi|_{2})V_{j,\ell}(\xi/|\xi|_{2}), \quad \xi\in\R^2.
\end{align}
These functions are non-negative `bumps' approximately supported in the corresponding wedges $\mathcal{W}_{j,\ell}$. They are symmetric, i.e., $W_{j,\ell}(\xi)=W_{j,\ell}(-\xi)$ for $\xi\in\R^2$, and they satisfy
\begin{align}\label{eq:CalderonW}
A:=A_1 A_2\le \sum_{J=(j,\ell)\in\mathbb{J}} W^2_{J}(\xi) \le B_1 B_2 =: B \,,\quad \xi\in\R^2.
\end{align}
Let us analyze the support of $W_J$ in more detail. Recall the angular function $\widetilde{V}_{j,0}$ and note that its support on $\mathbb{S}^1$ covers an angle range of $\varphi_j^+:=\frac{3}{2}\varphi_j$
with $\varphi_j=\pi L_j^{-1}$ as in~\eqref{eqdef:fundangle}. Moreover, $\sqrt{A_2}\le \widetilde{V}_{j,0} \le \sqrt{B_2}$ on a range of size $\varphi_j^-:=\frac{1}{2}\varphi_j$. Hence, $\supp V_{j,\ell} \subseteq \mathcal{A}_{j,\ell}$ and $V_{j,\ell}\asymp1$ on $\mathcal{A}^{-}_{j,\ell}$ for the angular intervals
\begin{align}\label{eq:angularsupp}
\begin{aligned}
\mathcal{A}_{j,\ell}:=R^{-1}_{j,\ell}\mathcal{A}_{j,0} \quad&\text{with}\quad \mathcal{A}_{j,0}:=\Big\{ \xi=(\xi_1,\xi_2)\in\mathbb{S}^1 ~:~  |\xi_1| \ge \cos(\varphi^+_j/2) \Big\}, \\
\mathcal{A}^-_{j,\ell}:=R^{-1}_{j,\ell}\mathcal{A}^-_{j,0} \quad&\text{with}\quad \mathcal{A}^-_{j,0}:=\Big\{ \xi=(\xi_1,\xi_2)\in\mathbb{S}^1 ~:~ |\xi_1|\ge\cos(\varphi^-_j/2) \Big\}.
\end{aligned}
\end{align}
Next, recall the functions $U_j$ on the ray with $\supp U_j\subseteq \mathcal{I}_j$. Due to \eqref{eq:coronalpartition} and \eqref{eq:suppcond} their function values are between $\sqrt{A_1}$ and $\sqrt{B_1}$ on
\begin{align}\label{eq:IIIIII}
\mathcal{I}^{-}_0:=C\cdot[0,\tau_1]  \quad\text{and}\quad \mathcal{I}^{-}_j:=C\cdot[2^{s(j-1)} \tau_2,2^{sj}\tau_1],\quad j\ge1,
\end{align}
respectively.
This leads us to the following definition. For $J=(j,\ell)\in\mathbb{J}$ we introduce the wedge pairs
\begin{align}\label{eq:wedgePJ}
\begin{aligned}
\tileext_{J}:=\Big\{ \xi\in\R^2 ~:~ |\xi|_{2}\in\mathcal{I}_j ,\, \varphi(\xi)\in\mathcal{A}_J \Big\} \quad\text{and} \quad
\tileint_{J}:=\Big\{ \xi\in\R^2 ~:~ |\xi|_{2}\in\mathcal{I}^-_j ,\, \varphi(\xi)\in\mathcal{A}^-_J  \Big\}.
\end{aligned}
\end{align}
The following support properties will be of essential importance later,
\begin{align}\label{eq:suppprop}
\supp W_J \subseteq \tileext_J \qquad\text{and}\qquad \sqrt{A}\le W_J \le \sqrt{B} \text{ on }\tileint_J.
\end{align}
A geometric illustration is displayed in Figure~\ref{fig:freq_domain}~(b).

Now we fix $C=2^{-s}/(3\pi)$ in \eqref{eqdefcorC} such that
each $\tileext_{J}$ is contained in the respective rectangle
\begin{align}\label{eq:supprect}
\Xi_{J}:=R^{-1}_{J}\Xi_{j,0}\,, \quad\text{where}\quad \Xi_{j,0}:= [-2^{js-1},2^{js-1}]\times[-2^{js\alpha-1},2^{js\alpha-1}].
\end{align}
The rectangles $\Xi_{j,0}$ are of size $2^{js}\times2^{js\alpha}$ and hence the Fourier system $\lb u_{j,0,k}\rb_{k\in\Z^2}$ given by
\begin{align*}
u_{j,0,k}(\xi):=2^{-js(1+\alpha)/2}\exp\big(2\pi i (2^{-sj}k_1\xi_1 +  2^{-sj\alpha}k_2\xi_2)\big),\quad \xi\in\R^2,
\end{align*}
constitutes an orthonormal basis for $L^2(\Xi_{j,0})$ . Consequently, the rotated
system $\lb u_{j,\ell,k}\rb_{k\in\Z^2}$ of functions
\begin{align}\label{eq:fouriersys}
u_{j,\ell,k}(\xi):= u_{j,0,k}(R_{j,\ell}\xi),\quad \xi\in\R^2,
\end{align}
is an orthonormal basis for $L^2(\Xi_{J})$.

After this preparation, we are ready to define the $\alpha$-curvelet system $\curvesys$.

\begin{definition}
Let $s>0$, $\alpha\in(-\infty,1]$, and assume that $\lb W_J \rb_{J\in\mathbb{J}}$ is a family of functions of the form~\eqref{eq:suppfunctions} such that~\eqref{eq:CalderonW}
holds for $0<A\le B<\infty$. Further, let $u_{j,\ell,k}$ be the functions defined in \eqref{eq:fouriersys}. The curvelet system $\curvesys(A,B):=\lb\psi_\mu\rb_{\mu\in\curveind}$
with associated index set $\curveind:=\mathbb{J}\times\Z^2$ consists of the functions $\psi_\mu=\psi_{j,\ell,k}$ given by
\begin{align}\label{eqdef:curvelet}
\widehat{\psi}_{j,\ell,k}(\xi):= W_{j,\ell}(\xi) u_{j,\ell,k}(\xi)\,, \quad \xi\in\R^2.
\end{align}
Note that $\curvesys(A,B)$ depends on the utilized family $\lb W_J\rb_{J\in\mathbb{J}}$, which is not accounted for in the notation.
\end{definition}

\noindent
The curvelets $\psi_\mu$ are real-valued due to the symmetry of $W_{j,\ell}$.
Their $L^2$-norms may vary slightly with scale, however there are constants $0<C_1\le C_2<\infty$ such that $C_1 \le  \|\psi_\mu\|_{2} \le C_2 $ holds true for all $\mu\in\curveind$.
Most importantly, the system $\curvesys(A,B)$ is a frame for $L^2(\R^2)$.

\begin{lemma}
The system $\curvesys(A,B)$ given by \eqref{eqdef:curvelet}
is a frame for $L^2(\R^2)$ with frame bounds $A$ and $B$.
\end{lemma}
\begin{proof}
The functions $\tilefunc_J$ satisfy condition \eqref{eq:CalderonW} wherefore
\begin{align*}
A\|f\|^2_2=A\|\widehat{f}\|^2_2\le \sum_{J\in\mathbb{J}} \| \widehat{f}W_J \|_2^2 \le B \|\widehat{f}\|^2_2 = B\|f\|^2_2 \quad\text{for every $f\in L^2(\R^2)$}.
\end{align*}
Since $\supp(\widehat{f}W_J) \subseteq \Xi_J$ and since $\{u_{J,k}\}_{k\in\Z^2}$ is an orthonormal basis of $L^2(\Xi_J)$
we have the orthogonal expansion $\widehat{f}W_J = \sum_k \langle \widehat{f}W_J, u_{J,k} \rangle u_{J,k} \chi_{\Xi_{J}}$. The proof is finished by the following equality,
\[
\| \widehat{f}W_J \|_2^2= \sum_{k\in\Z^2} |\langle \widehat{f}W_J, u_{J,k} \rangle|^2
= \sum_{k\in\Z^2} |\langle \widehat{f}, W_Ju_{J,k} \rangle|^2 = \sum_{k\in\Z^2} |\langle \widehat{f}, \widehat{\psi}_{J,k} \rangle|^2 =  \sum_{k\in\Z^2} |\langle f, \psi_{J,k} \rangle|^2. \qedhere
\]
\end{proof}

\noindent
The Parseval frame $\mathfrak{C}_{s,\alpha}(1,1)$ is of most interest to us and one might wonder why we did not fix the frame bounds $A=B=1$ in the beginning.
The reason is that, in the proof of Lemma~\ref{lem:induction},
we need the additional flexibility provided by variable $A$ and $B$.

\begin{remark}
Subsequently, we will write $\mathfrak{C}_{s,\alpha}$ to refer to the Parseval frame $\mathfrak{C}_{s,\alpha}(1,1)$.
\end{remark}

\noindent
Let us finish this section with a short discussion of the situation in spatial domain.
Here the $\alpha$-curvelets $\lb\psi_{j,\ell,k}\rb_{k\in\Z^2}$ are translates of the functions $\psi_{j,\ell,0}$.
Indeed, since $\widehat{\psi}_{j,\ell,0}= 2^{-js(1+\alpha)/2} W_{j,\ell}$ and
\[
u_{j,\ell,k}(\cdot)= u_{j,0,k}(R_{j,\ell}\cdot)=2^{-js(1+\alpha)/2}\exp\big(2\pi i \langle R^{-1}_{j,\ell} A^{-1}_{j}k,\cdot\rangle\big),
\]
where $R_{j,\ell}$ is the rotation matrix defined in \eqref{eq:matrixrot} and $A_{j}:=A_{\alpha,2^{js}}$ is an $\alpha$-scaling matrix of the form~\eqref{eq:alphamat1}, we have $\widehat{\psi}_{j,\ell,k}=\widehat{\psi}_{j,\ell,0} \exp\big(2\pi i \langle R^{-1}_{j,\ell} A^{-1}_{j}k, \cdot \rangle\big) $ and hence
\begin{align*}
 \psi_{j,\ell,k}= \psi_{j,\ell,0}(\cdot-x_{j,\ell,k})\quad\text{with}\quad x_{j,\ell,k}:=R^{-1}_{j,\ell}A^{-1}_{j}k.
\end{align*}
Since $\psi_{j,\ell,0}$ is the rotation of $\psi_{j,0,0}$ by the angle $\varphi_{j,\ell}=\ell\varphi_j$, we arrive at the representation
\begin{align}\label{eq:spatialrepr}
        \psi_{j,\ell,k} (x) =
        \psi_{j,0,0}
        \left(R_{j,\ell}\left(x - x_{j,\ell,k}\right)\right).
\end{align}
In fact, these systems are instances of $\alpha$-molecules, a concept recalled in the definition below.

\begin{definition}[{\cite[Def.~2.9]{GKKS15}}]
Let $\Lambda$ be a set and $\Phi_\Lambda:\Lambda\to\mathbb{P}$ a map, assigning to each $\lambda\in\Lambda$ a point $(s_\lambda,\theta_\lambda,x_\lambda)\in\mathbb{P}$ in
the so-called phase-space $\mathbb{P}=\Rplus\times\mathbb{T}\times \R^2$.
Futher, assume that $L,M,N_1,N_2\in\N_0$.
A family $\lb m_\lambda\rb_{\lambda \in \Lambda}$ of functions in $L^2(\R^2)$ is called a \emph{family of $\alpha$-molecules
of order $(L,M,N_1,N_2)$ with respect to the parametrization $(\Lambda,\Phi_\Lambda)$},
if there exist generators $a^{(\lambda)}\in L^2(\R^2)$ such that for all $\lambda\in\Lambda$
    \begin{align*}
        m_\lambda (\cdot) =
        s_\lambda^{(1+\alpha)/2}
        a^{(\lambda)}
        \left(A_{\alpha,s_\lambda}R_{\varphi_\lambda}\left(\cdot - x_\lambda\right)\right),
    \end{align*}
    and if for each $\rho\in\N_0^2$, $|\rho|\le L$, there is a constant $C_\rho>0$ such that
    for all $\lambda\in\Lambda$
    \begin{equation}\label{eq:molcond1}
        \big| \partial^{\rho} \hat a^{(\lambda)}(\xi)\big|
        \le C_\rho \min\left\{1,s_\lambda^{-1} + |\xi_1| + s_\lambda^{-(1-\alpha)}|\xi_2|\right\}^M
        \left( 1+ |\xi|^2 \right)^{-N_1/2} ( 1+ |\xi_2|^2)^{-N_2/2}, \quad\xi\in\R^2.
    \end{equation}
\end{definition}
\noindent
We can deduce from \eqref{eq:spatialrepr} that the $\alpha$-curvelets $\psi_{j,\ell,k}$ can be represented in the form
\begin{align}\label{eq:smolrepr}
        \psi_{j,\ell,k} (x) = 2^{js(1+\alpha)/2}
        a_{j}
        \left(A_{j} R_{j,\ell}\left(x - x_{j,\ell,k}\right)\right) = 2^{js(1+\alpha)/2}
        a_{j}
        \left(A_{j} R_{j,\ell} x - k \right)
    \end{align}
with respect to the generators
\begin{align}\label{eq:molgen}
a_{j}:=2^{-js(1+\alpha)/2}\psi_{j,0,0} (A^{-1}_{j} \cdot).
\end{align}
Since these generators fulfill condition~\eqref{eq:molcond1}, as shown in Lemma~\ref{thm:curvmol} below, $\curvesys$
is a system of $\alpha$-molecules of arbitrary order, at least in the range $\alpha\in[0,1]$ for which the concept was formulated.
The associated parametrization, mapping the curvelet index set $\curveind$ into the phase-space $\mathbb{P}=\R^+\times\mathbb{T}\times \R^2$, is given by
\begin{align}\label{eq:curvepara}
\Phi_{\curveind}: \curveind\to \mathbb{P} ,  \,(j,\ell,k)\mapsto (2^{js}, \varphi_{j,\ell}, x_{j,\ell,k}) = (2^{js},\ell \varphi_j, R^{-1}_{j,\ell} A^{-1}_{j}k).
\end{align}

\begin{lemma}\label{thm:curvmol}
Let $M,N_1,N_2\in\N_0$ and $\rho=(\rho_1,\rho_2)\in\N_0^2$ be fixed. There is a constant $C>0$ such that for all $j\in\N_0$ the generators~\eqref{eq:molgen} satisfy the estimate
\begin{align}\label{eq:molcond}
 \big| \partial^\rho \widehat{a}_{j}(\xi)\big|
        \le C \min\big\{1, 2^{-js} + |\xi_1| + 2^{-js(1-\alpha)}|\xi_2|\big\}^M
        (1+ |\xi|^2 )^{-N_1/2}  (1+|\xi_2|^2)^{-N_2/2}.
\end{align}
\end{lemma}
\begin{proof}
On the Fourier side the functions \eqref{eq:molgen} have the form
\[
\widehat{a}_{j}=2^{js(1+\alpha)/2}\widehat{\psi}_{j,0,0} (A_{j} \cdot) =  W_{j,0}( A_{j} \cdot).
\]
Let $j\in\N_0$ be arbitrary. We have $\supp W_{j,0} \subseteq \tileext_{j,0}$ and
\[
\tileext_{j,0} \subseteq [-2^{js-1},2^{js-1}]\times[-2^{js\alpha-1},2^{js\alpha-1}]=\Xi_{j,0},
\]
which implies
\begin{align}\label{eq:supphataj}
\supp \widehat{a}_{j} \subseteq [-2^{-1},2^{-1}]\times[-2^{-1},2^{-1}]=\Xi_{0,0}.
\end{align}
Further, if $j>0$ the function $\widehat{\psi}_{j,0,0}$ vanishes on the square $[-2^{s(j-2)-5},2^{s(j-2)-5}]^2$.
Consequently, $\widehat{a}_{j}$ vanishes on $[-2^{-2s-5},2^{-2s-5}]\times \big( 2^{js(1-\alpha)} \cdot [-2^{-2s-5},2^{-2s-5}] \big)$.

The mixed derivatives $\partial_1^{\rho_1}\partial_2^{\rho_2}W_{j,0}$ obey uniformly in $j\in\N_0$
\begin{align}\label{eq:basic_fact}
\|\partial_1^{\rho_1}\partial_2^{\rho_2}W_{j,0}\|_{\infty} \lesssim 2^{-js\rho_1} 2^{-js\alpha \rho_2}.
\end{align}
With the chain rule we deduce
\[
\| \partial^{\rho}\widehat{a}_{j}\|_\infty  = \| \partial_1^{\rho_1}\partial_2^{\rho_2} W_{j,0}( A_{j} \cdot) \|_\infty =
2^{js \rho_1} 2^{js \alpha \rho_2} \| \big( \partial_1^{\rho_1}\partial_2^{\rho_2} W_{j,0} \big) ( A_{j} \cdot) \|_\infty \lesssim 1.
\]
Due to $\supp \partial^\rho \widehat{a}_{j} \subseteq \supp \widehat{a}_{j}$
this estimate together with the support properties of $\widehat{a}_{j}$ implies~\eqref{eq:molcond}.
\end{proof}

\noindent
With the machinery of $\alpha$-molecules at our disposal, it is possible to
use $\curvesys$ as an anchor system whose properties have consequences for other $\alpha$-scaled systems if they fulfill
certain consistency conditions. In particular, approximation properties of $\curvesys$
are shared by other $\alpha$-scaled systems such as e.g.\ $\alpha$-shearlets. A short discussion of this can be found in Section~\ref{sec:discussion}.
For more details on the topic of $\alpha$-molecules we refer to \cite{GKKS15,FS15}.


\section{Curvelet Approximation of General Cartoons}
\label{sec:cartoon}


In the two central sections of this article, Sections~\ref{sec:cartoon} and \ref{sec:straight}, we study the approximation performance of the $\alpha$-curvelet frame $\curvesys$ 
with respect to different cartoon classes.
We begin in this section with classes of general cartoons, used e.g.\ to model natural images.
In Section~\ref{sec:straight} we then turn our focus on cartoons featuring only straight edges.

\subsection{Cartoon-like Functions}

Many suitable and well-established models for natural images are based on the concept of so-called cartoon-like functions.
In a nutshell, such functions can be thought of as a patchwork of smooth regions separated from one another by
piecewise-smooth discontinuity curves.
Their structure imitates the fact that edges, a typical feature of natural images, are characterized by abrupt changes of color and brightness,
whereas changes in the regions in between occur smoothly.

Mathematically, models based on this idea can be concretised in different ways. A classic model~\cite{CD04} postulates a compact image domain
separated into two $C^2$ regions by a closed $C^2$ discontinuity curve.
This model was generalized in various directions, e.g., to take into account piecewise-smooth edges or to allow
more general $C^\beta$ regularity with $\beta\in[0,\infty)$.
Cartoon classes of this kind
have been studied extensively, especially in the range $\beta\in(1,2]$, e.g., in \cite{Kutyniok2012correct,Kei13,GKKScurve2014}.
Another variant are the closely related horizon classes, where the discontinuity is not a closed curve
in the image domain but a (possibly curved) horizontal or vertical line stretching across. Such classes have been investigated e.g.\ in \cite{D99,CWBB04c,PM05}.
Let us also mention that there exist extensions to multi-dimensions, see e.g.\ \cite{Kutyniok2012correct}.
In particular, the corresponding 3D models have been applied in the investigation of
video data.

Since we are concerned with image approximation,
our attention is restricted to the 2-dimensional setting.
The following definition is a template for different classes of bivariate cartoons, comprising many of those mentioned above.
It provides the flexibility to taylor the model to our particular needs in Sections~\ref{sec:cartoon} and~\ref{sec:straight}.

\begin{definition}\label{def:gencart}
Let $\beta\in[0,\infty)$ and $\nu>0$. Given a domain $\Omega\subseteq\R^2$ and a set $\mathcal{A}$ of admissible subsets of $\R^2$,
the class $\mathcal{E}^{\beta}(\Omega;\mathcal{A},\nu)$ consists of all functions $f\in L^2(\R^2)$ of the form
\begin{align*}
f=f_{1}+f_{2}\chi_{\mathcal{D}},
\end{align*}
where $\mathcal{D}\in\mathcal{A}$ and $f_1,\,f_2\in C^{\beta}(\R^2)$ with $\supp f_{1}, f_2 \subseteq \Omega$ and $\|f_1\|_{C^\beta}, \|f_2\|_{C^\beta}\le\nu$.
The class $\mathcal{E}_{\rm bin}^{\beta}(\Omega;\mathcal{A})$ shall be the collection of all `binary functions' $\chi_{\mathcal{D}}$,
where $\mathcal{D}\in\mathcal{A}$ and $\mathcal{D}\subseteq\Omega$.
\end{definition}

For particular choices of $\mathcal{A}$ many of the classes appearing in the literature can be retrieved,
including classes of horizon-type.
In this section we focus on
the class $\mathcal{E}^{\beta}(\Omega;\mathcal{A},\nu)$
with fixed image domain $\Omega=[-1,1]^2$ and certain $C^\beta$ domains as admissible sets $\mathcal{A}$.
Similar to \cite{Don01,CD04,Kutyniok2010,Kutyniok2012correct}, we restrict our investigation to star-shaped domains, since those
allow a simple parametrization of the boundary curve. The results obtained however also hold true for more general domains.

Let us introduce the collection of admissible sets
$\text{\sc Star}^\beta(\nu)$, $\nu>0$,  as all translates of sets $B\subseteq\R^2$, whose boundary $\partial B$
possesses a parametrization $b:\mathbb{T}\to\R^2$ of the form
\[
b(\varphi)=\rho(\varphi) \begin{pmatrix} \cos(\varphi) \\ \sin(\varphi) \end{pmatrix}, \quad \varphi\in\mathbb{T}=[0,2\pi]\,,
\]
where the radius function $\rho:\mathbb{T} \to \R$ is a $C^\beta$ function with
\begin{align}\label{eq:HolCart}
|\partial^{\lfloor\beta\rfloor}\rho(\varphi)- \partial^{\lfloor\beta\rfloor}\rho(\varphi^\prime)| \le \nu \rho_0 |\varphi - \varphi^\prime|^{\beta-\lfloor\beta\rfloor}
\quad\text{for all }\varphi,\varphi^\prime\in\mathbb{T},
\end{align}
where we set $\rho_0:=\min_{\varphi\in\mathbb{T}} \rho(\varphi) \ge \nu^{-1}$.
The condition~\eqref{eq:HolCart} implies that with $C=C(\beta)=(2\pi)^{\beta}\ge1$ we have
$\|\rho^{(k)}\|_{C^0(\mathbb{T})}\le C\rho_0\nu$ for every $k\in\{1,\ldots,\lfloor\beta\rfloor\}$ if $\beta\ge1$, and $|\rho(\varphi)- \rho(\varphi^\prime)| \le C\rho_0\nu$ for $\varphi,\varphi^\prime\in\mathbb{T}$.
In particular $\rho_0 \le \rho(\varphi) \le \rho_0 (1 + C\nu)$ for all $\varphi\in\mathbb{T}$.

Note, that the set $\starshaped^\beta(\nu)$ differs from the set of star-shaped domains used in \cite{Don01,CD04,Kutyniok2010,Kutyniok2012correct}.
The domains in $\starshaped^\beta(\nu)$ are not restricted to subsets of $[-1,1]^2$. In fact, every star-shaped $C^\beta$ domain with center $0$ and $\rho_0>0$ is contained in $\starshaped^\beta(\nu)$ for suitably large $\nu$.
Moreover, the collection $\starshaped^\beta(\nu)$ is scaling invariant in the sense that for $B\in\starshaped^\beta(\nu)$ and $\lambda>0$ also $\lambda B\in\starshaped^\beta(\nu)$, provided $\lambda\rho_0\ge\nu^{-1}$.
In addition, with $B\in\starshaped^\beta(\nu)$ also the complement $B^c=\R^2\backslash B$ is contained in $\starshaped^\beta(\nu)$.

Building upon Definition~\ref{def:gencart}
we now define the class of functions which we want to study in this section. We put $\Omega=[-1,1]^2$ and $\mathcal{A}=\starshaped^\beta(\nu)$.
Further, we assume $\beta\in[0,\infty)$ and $\nu>0$. For the resulting class $\mathcal{E}^\beta([-1,1]^2;\starshaped^\beta(\nu),\nu)$ we simplify the notation
\begin{align}\label{eqdef:cart}
\mathcal{E}^\beta([-1,1]^2;\nu):= \mathcal{E}^\beta([-1,1]^2;\starshaped^\beta(\nu),\nu).
\end{align}
The associated binary class shall be denoted by
$\mathcal{E}_{bin}^\beta([-1,1]^2;\nu):= \mathcal{E}_{bin}^\beta([-1,1]^2;\starshaped^\beta(\nu))$.

\subsection{Class Bounds}

Before we investigate the approximation performance of the $\alpha$-curvelet frame $\curvesys$ with respect to the class $\mathcal{E}^\beta([-1,1]^2;\nu)$,
let us take a broader stance and aim for best possible $N$-term approximation
in case we can freely choose the utilized dictionary.
Of course, a countable dense subset of $L^2(\R^2)$ would yield arbitrarily good $1$-term approximations.
This shows that, without further restrictions, the question of best possible approximation is not well-posed.

To cast a realistic scenario, when computing $N$-term approximations typically a constraint on the search depth is imposed.
More concretely, given a fixed ordering of the dictionary and some polynomial $\pi$, it is common to allow only $N$-term approximants being built from the first $\pi(N)$ elements of the dictionary.
Under this so-called polynomial depth search constraint, an upper bound on the maximal achievable approximation rate was first derived by Donoho~\cite[Thm.~1]{Don01} for
binary $C^\beta$ cartoons in the range $\beta\in(1,2]$. Later similar results were proved for more general cartoon classes~\cite{Kutyniok2012correct,Kei13,GKKScurve2014}.

Theorem~\ref{thm:benchmark} below establishes a bound for the class $\mathcal{E}^\beta([-1,1]^2;\nu)$ specified in \eqref{eqdef:cart}.

\begin{theorem}\label{thm:benchmark}
Let $\beta,\gamma\in[0,\infty)$ and $\nu>0$. Assume that there is a constant $C>0$ such that
\[
\sup_{f \in \mathcal{E}^{\beta}([-1,1]^2;\nu)} \|f-f_N\|_2^2 \le C N^{-\gamma} \quad\text{ for all }N\in\N,
\]
where $f_N$ denotes the best $N$-term approximation of $f$ obtained by polynomial depth search in a fixed dictionary. Then necessarily $\gamma\le\beta$.
\end{theorem}

\noindent
In principle, this is a known result (see e.g.~\cite{Kutyniok2012correct}). However, for reasons of completeness, we outline a short proof based on the technique used in \cite{Don01}.
It relies on Theorem~\ref{thm:upperbound} below and the fact that the class $\mathcal{E}^{\beta}([-1,1]^2;\nu)$
contains a copy of $\ell_0^p$ for $p=2/(\beta + 1)$. Let us recall this notion introduced in~\cite{Don01}.

\begin{definition}[{\cite[Def.~1\&2]{Don01}}]
A function class $\cF\subseteq L^2(\R^2)$ is said to contain an \emph{embedded orthogonal hypercube} of dimension $m$ and side-length $\delta$ if there exist
$f_0 \in \cF$ and orthogonal functions $\psi_{\ell}\in L^2(\R^2)$, $\ell\in\{1,...,m\}$, with $\|\psi_{\ell}\|_{2}=\delta$ such that the collection of hypercube vertices embeds, i.e.,
\begin{align*}
\Big\{ f_0 + \sum_{\ell=1}^{m} \epsilon_\ell\psi_{\ell} ~:~ \epsilon=(\epsilon_1,\ldots,\epsilon_m) \in \{0,1\}^m \Big\} \subseteq \cF \,.
\end{align*}
It is said to \emph{contain a copy of $\ell_{0}^{p}$, $p>0$,} if it contains a sequence of
embedded orthogonal hypercubes,
whose associated dimensions $m_k$ and side-lengths $\delta_k$ satisfy $\delta_k \rightarrow 0$ for $k\to\infty$ and with a constant $C>0$
\begin{align*}
C\delta_k^{-p} \le m_k \quad\text{for all }k \in\N.
\end{align*}
\end{definition}

\noindent
The significance of this notion is due to the following result, which was first obtained in~\cite[Thm.~2]{Don01}. The reformulated version below can be found in \cite[Thm.~2.2]{GKKScurve2014}.

\begin{theorem}[{\cite[Thm.~2.2]{GKKScurve2014}}]\label{thm:upperbound}
Suppose, that a class of functions $\cF\subseteq L^2(\R^2)$ is uniformly $L^2$-bounded and contains a copy of $\ell^p_0$.
Then, allowing only polynomial depth search in a given dictionary, there is a constant $C>0$ such that for
every $N_0\in\N$ there is a function $f\in\cF$ and an $N\in\N$, $N\ge N_0$ such that
\begin{equation*}
\|f - f_{N} \|^2_2 \ge C \big(N \log_2 (N)\big)^{-(2-p)/p},
\end{equation*}
where $f_N$ denotes the best $N$-term approximation under the polynomial depth search constraint.
\end{theorem}

\noindent
It remains to investigate for which $p>0$ the class $\mathcal{E}^{\beta}([-1,1]^2;\nu)$  contains a copy of $\ell_0^p$.
To this end, let us introduce the following subclass of smooth functions for $\beta\in[0,\infty)$ and $\nu>0$,
\begin{align}\label{eqdef:smoothclass}
C_0^\beta([-1,1]^2;\nu):=\big\{ f\in C_0^\beta([-1,1]^2) ~:~ \|f\|_{C^\beta}\le\nu  \big\}.
\end{align}
Note, that the choice $\Omega=[-1,1]^2$ and $\mathcal{A}=\{\emptyset\}$ in Definition~\ref{def:gencart} yields this class. As a consequence,
\begin{align}\label{eq:smoothemb}
C_0^\beta([-1,1]^2;\nu) \subset \mathcal{E}^{\beta}([-1,1]^2;\nu).
\end{align}
Lemma~\ref{lem:copylp} below is the 2D analogon of the statement of \cite[Thm.~3.2]{Kutyniok2012correct}.
It shows, in particular, that $C_0^\beta([-1,1]^2;\nu)$ contains a copy of $\ell_0^{2/(\beta+1)}$.
Hence, as a consequence of \eqref{eq:smoothemb}, also $\mathcal{E}^{\beta}([-1,1]^2;\nu)$ contains a copy of $\ell_0^{2/(\beta+1)}$. An application of Theorem~\ref{thm:upperbound} thus
yields Theorem~\ref{thm:benchmark}.

\begin{lemma}\label{lem:copylp}
Let $\nu>0$, $\beta\in[0,\infty)$, and $p=2/(\beta+1)$. Then the following holds true.
\begin{enumerate}
\item[(i)] The function class $C_0^\beta([-1,1]^2;\nu)$ contains a copy of $\ell_0^p$.
\item[(ii)] The class of binary cartoons $\mathcal{E}_{\rm bin}^{\beta}([-1,1]^2;\nu)$ contains a copy of $\ell_0^p$ if $\nu\ge1$, otherwise it only contains the zero-function.
\end{enumerate}
\end{lemma}
\begin{proof}
The proof is a $2$D-adaption of the proof of \cite[Thm.~3.2]{Kutyniok2012correct}.
\end{proof}

\noindent
Summarizing, this establishes $N^{-\beta}$ as an upper bound for the possible order of approximation for general $C^\beta$ cartoons.
This rate is the benchmark, against which the performance of $\curvesys$ has to be measured.
We end this paragraph with the following observation.

\begin{remark}\label{rem:benchmark}
According to Lemma~\ref{lem:copylp}(i), the bound of Theorem~\ref{thm:benchmark} actually holds true for the class $C_0^\beta([-1,1]^2;\nu)$.
This is a stronger statement due to the inclusion \eqref{eq:smoothemb}. Further,
due to Lemma~\ref{lem:copylp}(ii), a statement analogous to Theorem~\ref{thm:benchmark} holds true for the binary class $\mathcal{E}_{\rm bin}^{\beta}([-1,1]^2;\nu)$ if $\nu\ge1$.
\end{remark}

\subsection{Approximation Guarantees}
\label{ssec:guarantees}

According to Theorem~\ref{thm:benchmark} and Remark~\ref{rem:benchmark} the order of the $N$-term approximation rate achievable for the classes $\mathcal{E}_{\rm bin}^{\beta}([-1,1]^2;\nu)$, $\nu\ge1$, and $\mathcal{E}^{\beta}([-1,1]^2;\nu)$, $\nu>0$, cannot exceed
$N^{-\beta}$. This bound
is valid for arbitrary dictionaries and independent of the
approximation scheme employed, as long as it respects a polynomial depth search condition. Even adaptive approximation schemes cannot perform better.

Schemes, where these rates are provably achieved, at least up to order, have been developed for binary cartoons based on
wedgelets~\cite{D99} and surflets~\cite{CWBB04b}, for general cartoons utilizing bandelets~\cite{PennecM,PM05}.
These results show that the optimality benchmark $N^{-\beta}$ can indeed be realized
in practice, at least up to order. However, the utilized schemes are mostly adaptive, only for certain cartoon classes nonadaptive methods with quasi-optimal
performance are known.

A breakthrough concerning the nonadaptive approximation of $C^2$ cartoons with curved edges was the introduction of curvelets by Cand{\`e}s and Donoho~\cite{CD2000,CD04}.
By a simple thresholding scheme, curvelet frames achieve an approximation rate matching the class bound $N^{-2}$ up to a log-factor.
The reason for this performance
is due to the parabolic scaling employed.
The following argument shall heuristically explain, why this type of scaling is ideal for the representation
of $C^2$ edges.

In local Cartesian coordinates, a $C^2$ curve can be represented as the graph $(E(x),x)$ of a function $E\in C^2(\R)$ and one can choose
a coordinate system such that $E^\prime(0)=E(0)=0$. A Taylor expansion then yields approximately $E(x)\approx \frac{1}{2} E^{\prime\prime}(0)x^2$, which matches
the essential support $width\approx length^2$ of parabolically scaled functions. Hence, those can provide optimal resolution of
the curve across all scales.

A similar heuristic applies to $C^\beta$ curves if $\beta\in(1,2]$. A Taylor expansion of $E\in C^\beta(\R)$ yields $|E(x)|\lesssim x^\beta$.
The curve is thus contained in a rectangle of size $width\approx length^{1/\beta}$ which suggests $\alpha$-scaling with $\alpha=\beta^{-1}$ for optimal approximation.
And indeed, the classic approximation result by Cand{\`e}s and Donoho could be extended in~\cite[Thm.~4.1]{GKKScurve2014} to the range $\beta\in(1,2]$.

This generalized result is stated below, slightly modified to fit into the setting of this article. The class $\mathcal{E}^{\beta}([-1,1]^2;\nu)$ used here is not fully identical to
the class in \cite{GKKScurve2014}. Moreover, only curvelet frames of the type $\curvesys$ with $s=1$ were considered there. It is not hard to verify though that
the proof carries over to general $s>0$ and that the statement is also
valid in our setting.

\begin{theorem}[{\cite[Thm.~4.1]{GKKScurve2014}}]\label{thm:oldcurveappr}
   Let $\beta\in(1,2]$, $\nu>0$. For the choice $\alpha=\beta^{-1}$, $s>0$ arbitrary, the frame of $\alpha$-curvelets $\curvesys$
   provides almost optimal sparse approximations for the class $\mathcal{E}^{\beta}([-1,1]^2;\nu)$. More precisely,
   there exists a constant $C>0$ such that for every $f\in \mathcal{E}^{\beta}([-1,1]^2;\nu)$ and $N\in\N$
\begin{align*}
 \|f-f_N\|_2^2\le CN^{-\beta} \log_2(1+N)^{1+\beta} \,,
\end{align*}
where $f_N$ denotes the $N$-term approximation of $f$ obtained by choosing the $N$ largest coefficients.
\end{theorem}

\noindent
This theorem naturally raises the question of extendibility beyond the range $\beta\in(1,2]$, a question pursued in the following subsection.
In particular, we investigate if the choice $\alpha=\beta^{-1}$ is still optimal for $\beta>2$.
Obviously, the heuristic consideration from above is not valid any more in this regime. And indeed,
we will see that for $\beta>2$ the optimal choice is not $\alpha=\beta^{-1}$. In fact, it is still $\alpha=\frac{1}{2}$
and choosing $\alpha<\frac{1}{2}$ deteriorates the approximation performance.

\subsection{Approximation Bounds}

The main results of this subsection, Theorems~\ref{thm:bound1} and \ref{thm:bound2}, establish bounds on the achievable $N$-term approximation rate for the class $\cart$, $\beta\in[0,\infty)$, when using the $\alpha$-curvelet frame $\curvesys$ for approximation. Unlike the bounds in Theorem~\ref{thm:benchmark} associated with the signal class the bounds
derived here are tied to the particular approximation system $\curvesys$.
However, via the framework of $\alpha$-molecules they are also effective for other $\alpha$-scaled systems,
such as e.g.\ $\alpha$-shearlets as discussed in Section~\ref{sec:discussion}.

In order to establish these bounds we study the approximability of certain example cartoons.
As a suitable object,
we choose the characteristic function of the ball $B_2(0,\frac{1}{2})\subset\R^2$ of radius $\frac{1}{2}$, for which we subsequently use the symbol
\begin{align}\label{def:Theta}
\Theta(x):=\chi_{B_2(0,\frac{1}{2})}(x_1,x_2)\,, \quad x\in\R^2.
\end{align}
This function embodies an exceptionally regular cartoon with a closed curved $C^\infty$-singularity.
It is radial symmetric and
binary, contained in $\mathcal{E}^\beta_{\rm bin}([-1,1]^2,\nu)$ for arbitrary $\beta\in[0,\infty)$ and $\nu\ge2$.
Furthermore, for every $\beta\in[0,\infty)$ and $\nu\ge2$ there is
$\gamma>0$ such that $\gamma\Theta\in\cart$, wherefore the approximability of $\Theta$ has implications for the approximability of these cartoon classes.

The Fourier transform of $\Theta$ is explicitly known. Let $\bessel$ denote the Bessel function of order 1, then according to~\eqref{eq:Fourier2ball}
\begin{align}\label{eq:ballfou}
\widehat{\Theta}(\xi) = \frac{\bessel(\pi|\xi|)}{2|\xi|} \,, \quad\xi\in\R^2.
\end{align}
Some properties of $\bessel$ and Bessel functions in general are collected in the appendix.

At the center of the following investigation is the lemma below, which estimates the
energy of $\widehat{\Theta}$ contained in the wedges $\mathcal{W}_{J}$, $J\in\mathbb{J}$.
Let $\{W_J\}_{J\in\mathbb{J}}$ be a family of functions of the kind~\eqref{eq:suppfunctions} with property \eqref{eq:CalderonW} for $0<A\le B<\infty$. Further,
let
\begin{align*}
W^-_J:=\chi_{\tileint_J} \quad\text{and}\quad  W^+_J:=\chi_{\tileext_J}
\end{align*}
be the characteristic functions of the sets $\tileint_J$ and $\tileext_J$ defined in \eqref{eq:wedgePJ}.

\begin{lemma}\label{lem:thetaest}
There are constants $0<C_1\le C_2<\infty$, independent of scale $j\ge j_0$, where $j_0\in\N_0$ is a suitable base scale, such that for all $J\in\mathbb{J}$ with $|J|\ge j_0$, where $|J|=j$ for $J=(j,\ell)\in\mathbb{J}$,
\[
AC_1 2^{-js(2-\alpha)} \le A\| \widehat{\Theta} W^{-}_J \|_2^2 \le \| \widehat{\Theta} W_J \|_2^2 \le B \| \widehat{\Theta} W^{+}_J \|_2^2 \le B C_2 2^{-js(2-\alpha)}.
\]
\end{lemma}
\begin{proof}
Let us recall the Bessel function $\bessel$ of order 1 and its
asymptotic behavior.
According to \eqref{eqapp:est} there is a constant $C>0$ and a function $R_1$ on $[1,\infty)$ satisfying $|R_1(r)| \le C r^{-3/2}$  such that
\begin{align*}
\bessel(r)=\sqrt{\frac{2}{\pi r}}\cos(r-\frac{3\pi}{4}) + R_1(r) \quad\text{ for }r\ge1.
\end{align*}
This allows to separate terms of higher order from $\bessel^2$. We decompose
\begin{align*}
\bessel^2(r) = \Big[ \frac{2}{\pi} \cos^2(r-\frac{3\pi}{4})r^{-1} \Big] + \Big[ \sqrt{\frac{8}{\pi}} \cos(r-\frac{3\pi}{4}) r^{-1/2}R_1(r)
+ R_1(r)^2  \Big]=: T_1(r) + T_2(r).
\end{align*}
For the following argumentation we need the square wave function $\sqcap:\R\to \{0,1\}$ defined by
\[
\sqcap(r) := \begin{cases} 1 \quad&,\, r\in\bigcup_{k\in\Z} k\pi + [-\frac{\pi}{2},0], \\
0 &,\, r\in\bigcup_{k\in\Z} k\pi + (0,\frac{\pi}{2}).  \end{cases}
\]
For all $r\in\R$ it has the property $2\cos^2(r-3\pi/4) \ge \sqcap(r)$. Therefore we can deduce for $1\le a\le b$
\begin{align*}
\int_a^{b} T_1(r) r^{-1} \,dr
=\frac{1}{\pi} \int_a^{b}  2\cos^2(r-\frac{3\pi}{4})r^{-2} \,dr
\ge\frac{1}{\pi} \int_a^{b}  \sqcap(r)r^{-2} \,dr
\ge\frac{1}{2} \sum_{k\in I_{a,b}} (k\pi)^{-2}
\end{align*}
with $I_{a,b}:=\{ k\in\Z ~:~ k\pi\in [a+\pi, b] \}$. To proceed, we use the relation
\[
\sum_{k=m}^n (k\pi)^{-2} \ge \frac{1}{\pi} \int_{m\pi}^{(n+1)\pi} k^{-2} \,dk,
\]
which is valid for all $m,n\in\N$ and $m\le n$. We obtain
\begin{align*}
\frac{1}{2} \sum_{k\in I_{a,b}} (k\pi)^{-2} \ge \frac{1}{2\pi} \int_{a+2\pi}^{b} k^{-2} \,dk = \frac{1}{2\pi} \big(\int_{a}^{b} k^{-2} \,dk  -  \int_{a}^{a+2\pi} k^{-2} \,dk \Big)
\ge \frac{1}{2\pi}(a^{-1}- b^{-1})  - a^{-2}.
\end{align*}
Next, we see that with a constant $C>0$ independent of $1\le a\le b$
\begin{align*}
\int_a^{b} |T_2(r)| r^{-1} \,dr \le C \int_a^{b} r^{-3} \,dr \le  C \int_a^{\infty} r^{-3} \,dr \le C a^{-2}.
\end{align*}
Altogether, we conclude that
\begin{align*}
\int_a^{b} \frac{\bessel^2(r)}{r} \,dr \ge \frac{1}{2\pi}(1 - ab^{-1})a^{-1} - (1+ C) a^{-2}.
\end{align*}
If $c=ab^{-1}\le 1$ is fixed, we can deduce for $a\ge 4\pi \frac{1+C}{1-c}$ the estimate
\begin{align}\label{eq:besauxrel}
\int_a^{a/c} \frac{\bessel^2(r)}{r} \,dr \ge \frac{1}{4\pi}(1-c)a^{-1}.
\end{align}
After this preparation, we can now turn to the actual proof of the assertion. The relation
\[
A\| \widehat{\Theta} W^{-}_J \|_2^2 \le \| \widehat{\Theta} W_J \|_2^2 \le B \| \widehat{\Theta} W^{+}_J \|_2^2
\]
is a direct consequence of \eqref{eq:suppprop} and $\|W_J\|_\infty \le \sqrt{B}$.
Let $\mathcal{I}_j$ be the intervals defined in \eqref{eqdef:dyintervals}. Further, recall the intervals $\mathcal{I}^-_j \subset \mathcal{I}_j$ defined in \eqref{eq:IIIIII}.
Using \eqref{eq:ballfou} and the definition \eqref{eq:wedgePJ} of $\mathcal{W}_J^-$ we calculate
\begin{align*}
\| \widehat{\Theta} W^{-}_J \|_2^2
= \int_{\tileint_J} \frac{\bessel^2(\pi|\xi|)}{4|\xi|^2} \,d\xi
=  \int_{\mathcal{I}^-_j} \int_{\mathcal{A}^-_J}  \frac{\bessel^2(\pi r)}{4r} \,d\varphi dr
\asymp  2^{-js(1-\alpha)} \int_{\pi\mathcal{I}^-_j} \frac{\bessel^2(r)}{r} \,dr.
\end{align*}
The intervals $\mathcal{I}^-_j$ scale like $\sim 2^{js}$. Hence, if $j\in\N$ is chosen large enough by \eqref{eq:besauxrel}
\begin{align*}
\| \widehat{\Theta} W^{-}_J \|_2^2 \asymp  2^{-js(1-\alpha)} \int_{\pi\mathcal{I}^-_j} \bessel^2(r) r^{-1} \,dr \gtrsim  2^{-js(1-\alpha)}  2^{-js}  = 2^{-js(2-\alpha)}.
\end{align*}
The estimate from above is much easier to establish. If $j\in\N$ such that $\pi\mathcal{I}_j\subset[1,\infty)$ we have
\begin{gather*}
\| \widehat{\Theta} W^+_{J} \|_2^2
= \int_{\tileext_J} \frac{\bessel^2(\pi|\xi|)}{4|\xi|^2} \,d\xi
=  \int_{\mathcal{I}_j} \int_{\mathcal{A}_J} \frac{\bessel^2(\pi r)}{4r} \,d\varphi dr
\asymp  2^{-js(1-\alpha)} \int_{\pi \mathcal{I}_j} \frac{\bessel^2(r)}{r} \,dr \\
\lesssim  2^{-js(1-\alpha)} \int_{\mathcal{I}_j} r^{-2}\,dr \lesssim  2^{-js(2-\alpha)}.
\tag*{\qedhere}
\end{gather*}
\end{proof}

\noindent
Based on Lemma~\ref{lem:thetaest} we can prove the first main result of this article.

\begin{theorem}\label{thm:bound1}
Let $\curvesys$ be the $\alpha$-curvelet frame constructed in Section~\ref{sec:curvelets} for fixed $\alpha\in(-\infty,1)$ and $s>0$.
There exists a constant $C>0$ such that for any given $N\in\N$ every $N$-term approximation $f_N$ of $\Theta$ with respect to
 $\curvesys$ (not even subject to a polynomial depth search constraint) satisfies
\begin{align*}
\|\Theta-f_N\|_2^2 \ge C N^{-\frac{1}{1-\alpha}}.
\end{align*}
\end{theorem}
\begin{proof}
Let $N\in\N$ be fixed and assume that
\[
f_N=\sum_{r=1}^N \theta_{J_r,k_r} \psi_{J_r,k_r}
\]
is a linear combination of $\alpha$-curvelets $\psi_{J_r,k_r}$ with coefficients $\theta_{J_r,k_r}\in\R$. The curvelets $\psi_{J_r,k_r}\in\curvesys$ satisfy $\supp \widehat{\psi}_{J_r,k_r}\subseteq \tileext_{J_r}$
as recorded in \eqref{eq:suppprop}.
It follows $\supp \widehat{f}_N \subseteq \mathcal{W}_N$ where $\mathcal{W}_N:=\bigcup_{J\in \mathbb{J}_N} \tileext_J$ for $\mathbb{J}_N:=\{ J_1,\ldots, J_N \}\subset \mathbb{J}$.
Using the notation $\mathbb{J}^c_N:=\mathbb{J}\backslash\mathbb{J}_N$ and
$\mathcal{W}^c_N:=\R^2\backslash \mathcal{W}_N$ we get with Lemma~\ref{lem:thetaest}
\begin{align*}
\| \Theta - f_N \|_2^2 &= \| \widehat{\Theta} - \widehat{f}_N \|_2^2
\ge \| \widehat{\Theta} \|_{L^2(\mathcal{W}^c_N)}^2
\ge  \sum_{J\in \mathbb{J}^c_N} \| \widehat{\Theta} W^{-}_J \|^2_2
\gtrsim \sum_{J\in \mathbb{J}^c_N} 2^{-js(2-\alpha)}.
\end{align*}
We want to bound the right-hand side from below.
By~\eqref{eq:Lj}, the number of tiles in each corona $\corona_j$, $j\in\N_0$, is given by $L_j$, where $L_0=1$ and $L_j = 2^{\lfloor js(1-\alpha) \rfloor + 1}$ for $j\ge1$.
Let
$j(N)\in\N$ denote the unique number such that
\begin{align*}
\sum_{j=0}^{j(N)-1} L_j <  N \le \sum_{j=0}^{j(N)} L_j.
\end{align*}
Since $2^{-js(2-\alpha)}$ decreases with rising scale we obtain
\begin{align*}
\sum_{J\in \mathbb{J}^c_N} 2^{-js(2-\alpha)}
\ge \sum_{j=j(N)+1}^{\infty} L_j 2^{-js(2-\alpha)}
\ge \sum_{j=j(N)+1}^{\infty}  2^{-js}  \gtrsim  2^{-j(N)s}.
\end{align*}
Here we used $ L_j \ge 2^{ js(1-\alpha)}$.
Since $N\gtrsim\sum_{j=0}^{j(N)-1} 2^{ js(1-\alpha)}  \gtrsim 2^{ j(N)s(1-\alpha)}$ we can finally deduce
\begin{align*}
\| \Theta - f_N \|_2^2 \gtrsim 2^{-j(N)s} = \Big( 2^{ j(N)s(1-\alpha)} \Big)^{-\frac{1}{1-\alpha}} \gtrsim N^{-\frac{1}{1-\alpha}}.
\tag*{\qedhere}
\end{align*}
\end{proof}

\noindent
This result can be strengthened if we restrict to greedy $N$-term approximations obtained by thresholding the coefficients.
Essential is the following observation, which has also been used in~\cite{GKKScurve2014}. Due to its importance we give a rigorous proof here.

\begin{lemma}\label{lem:apriori}
There is a constant $C>0$ such that all curvelets $\psi_\mu\in\curvesys$, $\mu\in\curveind$, satisfy
\[
\|\psi_\mu\|_1 \le C  2^{-js(1+\alpha)/2}.
\]
\end{lemma}
\begin{proof}
Let $a_j$ be the functions from \eqref{eq:molgen} and recall that according to \eqref{eq:supphataj} the support of $\widehat{a}_j$ is contained
in the unit square $\Xi_{0,0}$ for every $j\in\N_0$. Let $\ident$ denote the identity operator.
We have the estimate
\begin{align*}
\Big\|\mathcal{F}^{-1} \Big( (\ident+\partial_1^2)(\ident+\partial_2^2) \widehat{a}_{j} \Big) \Big\|_\infty
&\le  \| (\ident + \partial_1^2)(\ident + \partial_2^2) \widehat{a}_{j}  \|_1  \le \| (\ident + \partial_1^2)(\ident + \partial_2^2) \widehat{a}_{j}\|_\infty.
\end{align*}
According to Lemma~\ref{thm:curvmol} the right-hand side is bounded uniformly over all scales.
We conclude that there is a constant $C>0$, independent of $j\in\N_0$, such that
\begin{align*}
\sup_{x\in\R^2} |(1+x_1^2)(1+x_2^2) a_{j}(x)|  \le C.
\end{align*}
In other words $|a_{j}(x)| \le C (1+x_1^2)^{-1}(1+x_2^2)^{-1}$.
Using the representation \eqref{eq:smolrepr} we obtain
\begin{align*}
|\psi_{j,0,0}(x)| = 2^{js(1+\alpha)/2} |a_{j} (A_{j}x)| \le C 2^{js(1+\alpha)/2} (1+2^{2js}x_1^2)^{-1} (1+2^{2js\alpha}x_2^2)^{-1}
\end{align*}
and hence
\begin{align*}
\int_{\R^2} |\psi_{j,0,0}(x)| \,dx &\lesssim 2^{js(1+\alpha)/2} \int_{\R^2}(1+2^{2js}x^2_1)^{-1} (1+2^{2js\alpha}x^2_2)^{-1} \,dx \\
&= 2^{-js(1+\alpha)/2} \int_{\R^2}(1+x^2_1)^{-1} (1+x^2_2)^{-1} \,dx \lesssim 2^{-js(1+\alpha)/2}.
\end{align*}
Since $\|\psi_{j,\ell,k}\|_1=\|\psi_{j,0,0}\|_1$ the proof is finished.
\end{proof}

\noindent
Lemma~\ref{lem:apriori} allows to deduce a simple a-priori estimate of the curvelet coefficient size, namely
\begin{align}\label{eq:apriori}
|\theta_\mu|=|\langle f,\psi_\mu \rangle| \le \|f\|_{\infty} \|\psi_\mu\|_{1} \le C \|f\|_{\infty} 2^{-js(1+\alpha)/2} \quad\text{ for }\mu=(j,\ell,k)\in\curveind.
\end{align}
Note, that the constant $C>0$ is fully determined by $\curvesys$.
Using \eqref{eq:apriori} we now prove a stronger statement than Theorem~\ref{thm:bound1} for greedy approximations.

\begin{theorem}\label{thm:bound2}
Let $\alpha\in(-\infty,1]$ and $s>0$ be fixed. Further, let $f_N$ denote the $N$-term approximation of $\Theta$ with respect to the $\alpha$-curvelet frame $\curvesys$
obtained by thresholding the coefficients. There is a constant $C>0$ such that for every $N\in\N$
\begin{align*}
\|\Theta-f_N\|_2^2 \ge C N^{-\frac{1}{\max\{\alpha,1-\alpha\}}}.
\end{align*}
\end{theorem}
\begin{proof}
If $\alpha\le \frac{1}{2}$ the assertion is true by Theorem~\ref{thm:bound1}.
It remains to handle the range $1\ge\alpha>\frac{1}{2}$.
Let $\theta_{J_r,k_r}=\langle \Theta, \psi_{J_r,k_r} \rangle$, $r\in\{1,\ldots,N\}$, be the $N$ largest curvelet coefficients which determine
the approximant $f_N:=\sum_{r=1}^N \theta_{J_r,k_r} \psi_{J_r,k_r}$.
On the Fourier side the curvelet $\psi_{J,k}\in\curvesys$ is the product of
the functions $W_{J}$ and $u_{J,k}$ defined in \eqref{eq:suppfunctions} and \eqref{eq:fouriersys}, respectively.
Using condition~\eqref{eq:CalderonW} we first estimate
\begin{align*}
\|\Theta - f_N \|_2^2 = \|\widehat{\Theta} - \widehat{f}_N \|_2^2 \ge B^{-2} \sum_{J\in\mathbb{J}} \|\widehat{\Theta}W_J- \widehat{f}_NW_J\|_2^2
\ge \frac{A^2}{B^2} \sum_{J\in\mathbb{J}} \|\widehat{\Theta}\tilefuncint_J - \widehat{f}_N \tilefuncint_J \|_2^2,
\end{align*}
where $\tilefuncint_J$ is the characteristic function of the set $\tileint_J$ defined in \eqref{eq:wedgePJ}.
The triangle inequality yields
\begin{align}\label{eqqqqq}
\frac{1}{2} \|\widehat{\Theta} \tilefuncint_J \|_2^2  \le  \|\widehat{\Theta} \tilefuncint_J - \widehat{f}_N  \tilefuncint_J \|_2^2
+ \| \widehat{f}_N \tilefuncint_J \|_2^2 \quad\text{for every }J\in\mathbb{J}.
\end{align}
Observe the relation $\sqrt{A}\tilefuncint_J \le \tilefuncint_J W_{J} \le \sqrt{B} \tilefuncint_J $ and $\tilefuncint_J W_{J^\prime}=0 $ for $J\neq J^\prime$.
Therefore, it holds
\begin{align*}
\widehat{f}_N\tilefuncint_J = \sum_{r=1}^N \theta_{J_r,k_r} \widehat{\psi}_{J_r,k_r}\tilefuncint_J
= \sum_{r=1}^N \theta_{J_r,k_r} u_{J_r,k_r}\tilefunc_{J_r} \tilefuncint_J
\asymp \sum_{k\in K_J} \theta_{J,k} u_{J,k} \tilefuncint_J
\end{align*}
with $K_J=\{ k_r\in \Z^2 ~:~ r\in\{1,\ldots,N \}, \, J_r=J \}$.
Next, we use that $\{u_{J,k}\}_{k\in\Z^2}$  is an orthonormal basis of $L^2(\Xi_J)$, where $\Xi_J\supset\mathcal{W}^-_J$ is the set defined in \eqref{eq:supprect}. We estimate
\begin{align*}
\Big\|\sum_{k\in K_J} \theta_{J,k} u_{J,k} \tilefuncint_J \Big\|_2^2 \le \Big\|\sum_{k\in K_J} \theta_{J,k} u_{J,k} \Big\|_{L^2(\Xi_J)}^2 = \sum_{k\in K_J} |\theta_{J,k}|^2.
\end{align*}
The frame coefficients satisfy the a-priori estimate $|\theta_{J,k}|^2 \lesssim 2^{-js(1+\alpha)}$ according to \eqref{eq:apriori}.
Thus we obtain
\begin{align*}
\|\widehat{f}_N\tilefuncint_J\|_2^2 \asymp \Big\|\sum_{k\in K_J} \theta_{J,k} u_{J,k} \tilefuncint_J \Big\|_2^2 \lesssim (\# K_J)  2^{-js(1+\alpha)}.
\end{align*}
By Lemma~\ref{lem:thetaest} we have $\| \widehat{\Theta} \tilefuncint_J\|_2^2 \gtrsim 2^{-js(2-\alpha)}$. We deduce from \eqref{eqqqqq}
\begin{align*}
\|\widehat{\Theta}\tilefuncint_J - \widehat{f}_N\tilefuncint_J \|_2^2
\ge \frac{1}{2} \|\widehat{\Theta}\tilefuncint_J \|_2^2 - \|\widehat{f}_N\tilefuncint_J \|_2^2
\gtrsim  2^{-js(2-\alpha)} - (\# K_J)  2^{-js(1+\alpha)}.
\end{align*}
Altogether, we conclude
\begin{align*}
\|\Theta - f_N \|_2^2
\ge \sum_{J\in\mathbb{J}} \|\widehat{\Theta}\tilefuncint_J - \widehat{f}_N\tilefuncint_J \|_2^2
\gtrsim \sum_{J\in\mathbb{J}} \max\big\{0, 2^{-js(2-\alpha)} - (\# K_J)  2^{-js(1+\alpha)} \big\}.
\end{align*}
Note that $\sum_{J}(\# K_J) \le N$. To derive a lower bound let us consider the following minimization problem:
\begin{align*}
\underset{\{N_J\}_{J\in\mathbb{J}}}{\text{\sc Minimize}} \quad \sum_{J\in\mathbb{J}} \max\{0,2^{-js(2-\alpha)} - N_J 2^{-js(1+\alpha)} \} \quad \text{ s.t.} \quad \sum_{J\in\mathbb{J}} N_J \le N,\, N_J\in[0,\infty)\,\, (J\in\mathbb{J}).
\end{align*}
The condition $N_J\in[0,\infty)$, which simplifies the subsequent argumentation, is possible since we are only interested in a bound.
For the optimal choice $\{N_J\}_{J}$, it necessarily holds $\sum_{J} N_J = N$ and
\[
N_J \le 2^{-js(2-\alpha)} 2^{js(1+\alpha)} = 2^{js(2\alpha-1)}.
\]
Hence, the minimization problem can be reformulated as minimizing the term
\begin{align*}
\sum_{J\in\mathbb{J}} \big( 2^{-js(2-\alpha)} - N_J 2^{-js(1+\alpha)} \big)
\end{align*}
under the constraints $\sum_J N_J =N$ and $N_J \le 2^{js(2\alpha-1)}$. Assume that the family $\{N_J\}_J$ fulfills these constraints.
Further, let $j(N)\in\N$ denote the number determined by the property
\begin{align}\label{eq:unique}
\sum_{j=0}^{j(N)-1} 2^{js(2\alpha-1)} L_j <  N \le \sum_{j=0}^{j(N)} 2^{js(2\alpha-1)} L_j,
\end{align}
where $L_j$ from \eqref{eq:Lj} counts the wedges in the corona $\corona_j$.
Then the following estimate holds true
\begin{align*}
\sum_{J\in\mathbb{J}} \Big( 2^{-js(2-\alpha)} - N_J 2^{-js(1+\alpha)} \Big) \ge \sum_{j=j(N)+1}^\infty \Big( \sum_{|J|=j} 2^{-js(2-\alpha)} \Big) \ge
\sum_{j=j(N)+1}^\infty  2^{-js} \gtrsim  2^{-j(N)s}.
\end{align*}
To see this, note that $2^{-js(1+\alpha)}$ is decreasing with rising scale and that $L_j\ge 2^{js(1-\alpha)}$.
Since $N \asymp 2^{j(N)s\alpha}$, which follows from \eqref{eq:unique}, we have proven
\begin{align*}
\|\Theta - f_N \|_2^2 \gtrsim \sum_{J\in\mathbb{J}} \max\big\{0, 2^{-js(2-\alpha)} - (\# K_J)  2^{-js(1+\alpha)} \big\} \gtrsim 2^{-j(N)s} \asymp N^{-\frac{1}{\alpha}}\,
\end{align*}
and the proof is finished.
\end{proof}

\noindent
The approximation results for $\Theta$ have direct implications for the class-wise approximation of cartoon-like functions.
If $\nu\ge2$, then $\Theta\in\mathcal{E}_{bin}^{\beta}([-1,1]^2;\nu)$ for arbitrary $\beta\in[0,\infty)$.
Moreover, we can always find $\gamma>0$ such that $\gamma\Theta\in\mathcal{E}^{\beta}([-1,1]^2;\nu)$.
This allows to draw the following conclusion.

\begin{cor}
Let $\beta\in[0,\infty)$ and $\nu\ge2$. The uniform decay of the $N$-term approximation error for $\cartbin$
and $\cart$ provided by $\curvesys$ cannot be faster than $N^{-\frac{1}{1-\alpha}}$. Futhermore, thresholding of coefficients cannot yield rates better
than $N^{-\frac{1}{\max\{\alpha,1-\alpha\}}}$.
\end{cor}

\noindent
If $\beta>2$ it is thus impossible for $\curvesys$ to reach the theoretically possible approximation order of $N^{-\beta}$ for the class $\cart$.
The best performance is achieved for the classic choice $\alpha=\frac{1}{2}$, with a corresponding approximation rate of order $N^{-2}$.
A smaller $\alpha$ leads to a deterioration of the approximation. As is obvious from our investigation,
this behavior applies to cartoons with curved edges
exemplified by the function $\Theta=\chi_{B_2(0,\frac{1}{2})}$ from \eqref{def:Theta}.
For such cartoons the rate inevitably deteriorates as $\alpha$ tends to $0$, since their
energy is spread more or less uniformly across all directions of the Fourier plane.
In the next section, we narrow our focus and consider only cartoons with straight edges. Such cartoons are highly anisotropic and in a certain sense the opposite extreme
of the isotropic function $\Theta$. Since their Fourier energy is concentrated in only one direction, a smaller $\alpha$ will
be an advantage for their approximation.

\section{Images with Straight Edges}
\label{sec:straight}

In the following, we investigate the approximation performance of the curvelet frame $\curvesys$ with respect to cartoons with straight edges.
To specify the associated signal class,
let {\sc Straight} be the collection of all closed half-spaces of $\R^2$. Parameterized by $\varphi\in[0,2\pi)$ and $c\in\R$, these are subsets of the form
\[
H(\varphi,c)= \Big\{ (x_1,x_2)\in\R^2 ~:~ x_1\cos(\varphi)- x_2\sin(\varphi)\ge c\, \Big\} \,.
\]
Using Definition~\ref{def:gencart} we then introduce the following image class with parameters $\beta\in[0,\infty)$ and $\nu>0$
\begin{align*}
\csE^\beta([-1,1]^2;\nu):=\mathcal{E}^\beta([-1,1]^2; \text{\sc Straight}, \nu).
\end{align*}
This is a subclass of the general cartoons~\eqref{eqdef:cart} considered in Section~\ref{sec:cartoon}. Indeed, for $\nu>0$ and $\tilde{\nu}\ge\nu$ chosen large enough
\[
C_0^\beta([-1,1]^2;\nu) \subset \csE^\beta([-1,1]^2;\nu) \subset \mathcal{E}^\beta([-1,1]^2;\tilde{\nu}),
\]
where $C_0^\beta([-1,1]^2;\nu)$ is the class defined in \eqref{eqdef:smoothclass}.
These inclusions allow to transfer the optimality benchmark $N^{-\beta}$, valid for both $\mathcal{E}^\beta([-1,1]^2;\tilde{\nu})$ and $C_0^\beta([-1,1]^2;\nu)$ (see Theorem~\ref{thm:benchmark} and Remark~\ref{rem:benchmark}).
For $\straight$, we thus again aim for an approximation rate of order $N^{-\beta}$.

Ridgelet frames were developed specifically for the optimal representation of
functions with straight line singularities. For both variants, `orthonormal ridgelets'~\cite{D98}
and `$0$-curvelets'~\cite{GrohsRidLT}, it has been shown that they reach the optimality bound $N^{-\beta}$.
More precisely, this rate was proved for `mutilated Sobolev functions' with compact support~\cite{C99,GOtech16}, i.e., compactly supported functions which are in the Sobolev space $H^{\beta}(\R^2)$
apart from straight line singularities.
In line with the result from \cite{GOtech16} for $0$-curvelets,
we can expect that decreasing $\alpha$ improves the approximation ability of $\curvesys$ for $\straight$.

Our main result concerning the $\alpha$-curvelet approximation of $\straight$ is Theorem~\ref{thm:mainappr1} below.
It is formulated and proved for integer $\beta\in\N$ only, although the statement should extend to the whole range $\beta\in\Rplus$.
In this way, we avoid technical difficulties which would arise if we used finite differences instead of integer derivatives (compare~\cite{GKKScurve2014}).

\begin{theorem}\label{thm:mainappr1}
   The parameters $\beta\in\N$, $\nu>0$, $\alpha\in[0,1)$, and $s>0$ shall be fixed.
   Further, let $f_N$ be the $N$-term approximation of a signal $f\in L^2(\R^2)$
   provided by the $N$ largest coefficients with respect to the
   frame $\curvesys=\lb\psi_\mu\rb_{\mu\in\curveind}$.
   There exists a constant $C>0$ such that for every $f\in \csE^{\beta}([-1,1]^2;\nu)$ and $N\in\N$
\begin{align*}
 \|f-f_N\|_2^2\le C \begin{cases}
 N^{-\beta} \log_2(1+N)^{1+\beta} \quad&\text{ if } \alpha\le\beta^{-1}, \\
 N^{-1/\alpha}  \quad&\text{ if } \alpha>\beta^{-1}.
 \end{cases}
\end{align*}
\end{theorem}

\noindent
As expected, decreasing the parameter $\alpha$ improves
the approximation performance.
If $\alpha\in[0,\beta^{-1}]$ the achieved rate is even optimal up to the log-factor.
In this range signals from $\csE^{\beta}([-1,1]^2;\nu)$ are represented with the same efficiency as a
smooth function from $\smooth$.

Theorem~\ref{thm:mainappr1} is deduced by studying the curvelet coefficients, whose decay
is closely related to the achieved $N$-term approximation rate.
Recall that a typical measure for the sparsity of a sequence $\lb c_\lambda\rb_\lambda\subset\C$ is given by the weak $\ell^p$-(quasi)-norms, for $p>0$ defined by
\[
\|\lb c_\lambda\rb_\lambda\|_{w\ell^p}:= \Big( \sup_{\varepsilon>0} \varepsilon^p \cdot \#\{\lambda: |c_\lambda|>\varepsilon \}\Big)^{1/p}.
\]
By definition, the sequence $\lb c_\lambda\rb_\lambda$ belongs to $w\ell^p(\Lambda)$ if and only if the quantity $\|\lb c_\lambda\rb_\lambda\|_{w\ell^p}$ is finite.
This is the case precisely if there exists a constant $C>0$ such that $\#\{\lambda: |c_\lambda|>\varepsilon \}\le C^p \varepsilon^{-p}$
for all $\varepsilon>0$. The smallest possible such constant then coincides with the weak $\ell^p$-(quasi)-norm of the sequence.
Another useful characterization of a sequence $\lb c_\lambda\rb_\lambda\in w\ell^p(\Lambda)$ is given in terms of its non-increasing rearrangement $\lb c^*_n\rb_{n\in\N}$.
It holds $|c^\ast_n|\lesssim n^{-1/p}$ and
$\sup_{n>0} n^{1/p}|c^\ast_n| = \|\lb c_\lambda\rb_\lambda\|_{w\ell^p}$.

As illustrated by the following well-known lemma (see e.g.\ \cite{Devore1998}), the decay of the frame coefficients determines the $N$-term approximation rate achieved by thresholding.
A full proof is given e.g.\ in \cite{GKKS15}.

\begin{lemma}[{\cite[Lem.~5.1]{GKKS15}}]\label{lem:decayapprox}
Let $\lb m_\lambda\rb_{\lambda\in \Lambda}$ be a frame in $L^2(\R^2)$ and $f=\sum c_\lambda m_\lambda$ an expansion of
$f\in L^2(\R^2)$ with respect to this frame. If $\lb c_\lambda\rb_\lambda\in w\ell^{2/(\beta+1)}(\Lambda)$ for some $\beta\ge0$,
then the $N$-term approximations $f_N$ obtained by keeping the $N$ largest coefficients satisfy
\[
 \| f-f_N \|_2^2 \lesssim N^{-\beta}.
\]
\end{lemma}

\noindent
Beginning in Subsection~\ref{ssec:curvesparse}, we study
the sparsity of the coefficients $\theta_\mu=\langle f,\psi_\mu \rangle$
provided by the frame $\curvesys=\lb\psi_\mu\rb_{\mu\in\curveind}$ for a signal $f\in\straight$.
The decay rates proved in Theorem~\ref{thm:mainappr2}
are the foundation of the following proof of Theorem~\ref{thm:mainappr1}.

\begin{proof}[Proof of Theorem~\ref{thm:mainappr1}]
If $\alpha>\beta^{-1}$ the sequence $\lb\theta_\mu\rb_{\mu\in\curveind}$ of curvelet coefficients $\theta_\mu=\langle f,\psi_\mu \rangle$ belongs to $w\ell^p(\curveind)$ with $p=2/(1+1/\alpha)$. This is
proved in Theorem~\ref{thm:mainappr2}. Lemma~\ref{lem:decayapprox} directly translates this into the statement of Theorem~\ref{thm:mainappr1}.
In case $\alpha\le\beta^{-1}$ Theorem~\ref{thm:mainappr2} yields $|\theta_m^*|^2 \le C m^{-(1+\beta)} (\log_2 m)^{1+\beta}$
for the curvelet coefficient $\theta_m^*$ of $m$-th largest modulus. Utilizing the frame property of $\curvesys$ we can estimate
\[
\| f-f_N \|^2 \lesssim \sum_{m>N} |\theta^\ast_m|^2 \lesssim  \sum_{m>N} m^{-(1+\beta)} \cdot \left(\log_2 m\right)^{(1+\beta)}
\le \int_N^\infty t^{-(1+\beta)} \cdot \left(\log_2 (1+t)\right)^{(1+\beta)} \,dt.
\]
Note that $N\ge1$. Partial integration leads to
\begin{align*}
 \int_N^\infty t^{-(1+\beta)} \cdot \left(\log_2(1+t)\right)^{(1+\beta)} \,dt
 \lesssim N^{-\beta}  \left(\log_2 (1+N)\right)^{(1+\beta)}
 + \int_N^\infty t^{-(1+\beta)} \cdot \left(\log_2 (1+t)\right)^{\lceil\beta\rceil} \,dt.
\end{align*}
We repeat this $\lceil\beta\rceil$-times and finally arrive at
\begin{align*}
\int_N^\infty t^{-(1+\beta)} \cdot \left(\log_2(1+t)\right)^{(1+\beta)} \,dt  \lesssim  N^{-\beta}  \left(\log_2 (1+N)\right)^{(1+\beta)}.
\tag*{\qedhere}
\end{align*}
\end{proof}


\subsection{Sparsity of Curvelet Coefficients}
\label{ssec:curvesparse}


Subsequently, we study the decay of the curvelet coefficients $\theta_\mu=\langle f,\psi_\mu \rangle$.
Our main result is Theorem~\ref{thm:mainappr2}.

\begin{theorem}\label{thm:mainappr2}
Let $\alpha\in[0,1)$, $s>0$, $\beta\in\N$, and $\nu>0$ be fixed.
Further, denote by $\theta^*_{N}$ the (in modulus) $N$-th largest coefficient of $f\in\straight$ with respect to
$\curvesys=\lb\psi_\mu\rb_{\mu\in\curveind}$.
There exists a constant $C>0$ independent of $N\ge2$ such that
\begin{align*}
 \sup_{f\in \straight}  |\theta^*_{N}|^2 \le C\cdot \begin{cases} N^{-(1+\beta)} \cdot \left(\log_2 N\right)^{1+\beta}  \quad&\text{ if } \alpha\le\beta^{-1},  \\
 N^{-(1+1/\alpha)} \quad&\text{ if } \alpha>\beta^{-1}.
 \end{cases}
\end{align*}
\end{theorem}
  \begin{proof}
  Let $\curveind_j$ denote the subset of the curvelet index set $\curveind$ corresponding to scale $j$. Further, given $\varepsilon>0$, let us
  define $\curveind_{\varepsilon}:=\Big\{ \mu\in\curveind : |\theta_\mu|>\varepsilon \Big\}$ and $\curveind_{j,\varepsilon}:=\Big\{ \mu\in\curveind_j : |\theta_\mu|>\varepsilon \Big\}$.
  According to \eqref{eq:apriori} there is a constant $\widetilde{C}>0$, independent of scale, such that
  \[
  |\theta_\mu| \le \widetilde{C}\|f\|_{\infty} 2^{-js(1+\alpha)/2} \le  \widetilde{C}\nu 2^{-js(1+\alpha)/2}.
  \]
  At scales $j> j_\varepsilon:=\frac{2 \log_2(\widetilde{C}\nu\varepsilon^{-1})}{s(1+\alpha)}$
  the coefficients thus satisfy $|\theta_\mu|<\varepsilon$ and the sets $\curveind_{j,\varepsilon}$ are empty. In particular $\# \curveind_{\varepsilon}=0$ in case
  $\varepsilon>\widetilde{C}\nu$ since then $j_\varepsilon<0$.
  If $j\le j_\varepsilon$ Proposition~\ref{prop:main_sequence}, which is stated and proved below, gives the estimate
  \begin{align*}
  \# \curveind_{j,\varepsilon}  \lesssim 2^{j\rho}  \varepsilon^{-2/(1+\beta)} \quad\text{ with }\quad  \rho=\frac{s \max\{\alpha\beta-1,0\}}{1+\beta}\ge 0.
  \end{align*}
  If $\alpha>\beta^{-1}$ we have $\rho>0$ and conclude
  \begin{align*}
  \# \curveind_{\varepsilon} = \sum_{j=0}^{\lfloor j_\varepsilon \rfloor} \# \curveind_{j,\varepsilon}
    \lesssim \sum_{j=0}^{\lfloor j_\varepsilon \rfloor} 2^{j\rho}  \varepsilon^{-2/(1+\beta)} \lesssim 2^{j_\varepsilon \rho} \varepsilon^{-2/(1+\beta)}
     = \varepsilon^{-\frac{2(\alpha\beta-1)}{(1+\beta)(1+\alpha)}} \varepsilon^{-2/(1+\beta)}  = \varepsilon^{-2/(1+1/\alpha)}.
  \end{align*}
  From here, a direct argument leads to $|\theta^*_N|^2\lesssim N^{-(1+1/\alpha)}$ for the $N$-th largest coefficient $\theta^\ast_N$.

  If $\alpha\le\beta^{-1}$ we have $\rho=0$ and the estimate
  \begin{align*}
  \# \curveind_{\varepsilon}  \lesssim \sum_{j=0}^{\lfloor j_\varepsilon \rfloor} \varepsilon^{-2/(1+\beta)} \lesssim (\log_2(\widetilde{C}\nu\varepsilon^{-1})+1) \varepsilon^{-2/(1+\beta)} = \log_2(2\widetilde{C}\nu\varepsilon^{-1}) \varepsilon^{-2/(1+\beta)}.
  \end{align*}
  Hence, there is a constant $C_2\ge 1$ such that $\# \curveind_{\varepsilon} \le C_2 \log_2(C_1\varepsilon^{-1}) (C_1\varepsilon^{-1} )^{2/(1+\beta)}$ with $C_1=\max\{1,2\widetilde{C}\nu\}$.
  It follows $|\theta^*_{N}|\le C_1\delta_N$ for the number $\delta_N$ which solves $N=C_2 \log_2(\delta_N^{-1}) \delta_N^{-2/(1+\beta)}$.
  In general $\delta_N$ cannot be calculated explicitly, wherefore we resort to an estimate.

  If $N\ge 2$ we have $\varepsilon_N:=N^{-\frac{1+\beta}{2}}\le \frac{1}{2}$ since $\beta\ge1$.
  Taking into account $C_2\ge1$ we conclude
  \[
  C_2 {\varepsilon_N}^{-2/(1+\beta)} \log_2(\varepsilon_N^{-1}) \ge N =  C_2 \delta_N^{-2/(1+\beta)} \log_2(\delta_N^{-1}),
  \]
  which in turn proves $\delta_N\ge \varepsilon_N=N^{-\frac{1+\beta}{2}}$. Therefore $\widetilde{\delta}_N\ge\delta_N$ for the solution $\widetilde{\delta}_N$ of
  \[
   N= C_2\widetilde{\delta}_N^{-2/(1+\beta)} \log_2(N^{\frac{1+\beta}{2}}).
  \]
  An explicit calculation yields
  $
  \widetilde{\delta}_N =(C_2\frac{1+\beta}{2})^{(1+\beta)/2}  N^{-(1+\beta)/2} (\log_2 N)^{(1+\beta)/2},
  $
  which proves the claim.
  \end{proof}

\noindent
The missing ingredient in the proof of Theorem~\ref{thm:mainappr2} is Proposition~\ref{prop:main_sequence}.

\begin{prop}\label{prop:main_sequence}
Let the parameters $\alpha\in[0,1)$, $s>0$, $\beta\in\N$, and $\nu>0$ be fixed.
Further, let $\curveind_j$ denote the curvelet indices at scale $j$. The sequence $\lb\theta_\mu\rb_{\mu\in\curveind_j}$ of coefficients $\theta_\mu=\langle f,\psi_\mu\rangle$ obeys
\[
\| \lb\theta_\mu\rb_{\mu\in\curveind_j} \|^{2/(1+\beta)}_{w\ell^{2/(1+\beta)}}\lesssim 2^{j\rho} \quad\text{with}\quad \rho=s\max\{\alpha\beta-1,0\}/(1+\beta)
\]
and an implicit constant independent of scale $j\in\N_0$ and $f\in\straight$.
\end{prop}

\noindent
For the proof of Proposition~\ref{prop:main_sequence} we decompose $f$ into fragments, a technique pioneered in \cite{CD04}.
To this end, let $\mathcal{Q}_j$ at every scale $j\in\N_0$ denote the collection of cubes
\[
Q:=Q^{(j)}_{(k_1,k_2)}:=[2^{-js\alpha}(k_1-1),2^{-js\alpha}(k_1+1)]\times[2^{-js\alpha}(k_2-1),2^{-js\alpha}(k_2+1)], \quad (k_1,k_2)\in\Z^2.
\]
Further, let $\omega\in C_0^\infty([-1,1]^2)$ be a nonnegative window vanishing outside the square $[-1,1]^2$, such that
the family $\{\omega_Q\}_{Q\in\mathcal{Q}_j}$ of functions $\omega_Q(x):=\omega(2^{js\alpha}x_1-k_1, 2^{js\alpha}x_2-k_2)$ is a partition of unity, i.e.,
it has the property $\sum_{Q\in\mathcal{Q}_j} \omega_Q=1$.
Following \cite{CD04} we then decompose $f=\sum_{Q} f_Q$ into the fragments
\begin{align}\label{eqdef:frag}
f_Q:=f \omega_Q  \quad, Q\in\mathcal{Q}_j.
\end{align}
Note that $\supp f_Q\subseteq Q$
and that the size of the squares $Q\in \mathcal{Q}_j$ corresponds to the `essential' length of the curvelets at scale $j$.
Therefore $\langle f,\psi_\mu \rangle \approx \langle f_Q,\psi_\mu \rangle$ for a curvelet $\psi_\mu$ at the location of the cube $Q$.

For every $Q\in\mathcal{Q}_j$ we now investigate the sparsity of the sequence
\begin{align}\label{eq:Qsequence}
\theta_{Q}:=\lb\langle f_Q, \psi_\mu\rangle\rb_{\mu\in\curveind_j}\,.
\end{align}
Clearly, due to $\supp f\subseteq [-1,1]^2$ we only need to consider cubes $Q\in \mathcal{Q}_j$ which meet the square $[-1,1]^2$. Of these relevant cubes,
let us collect those which intersect the straight edge in $\mathcal{Q}_j^{1}$, the others in $\mathcal{Q}_j^{0}$. The associated fragments $f_Q$ will be
called \emph{edge fragments} and \emph{smooth fragments}, respectively.
The main result concerning the sparsity of \eqref{eq:Qsequence} is Proposition~\ref{prop:frag}.

\begin{prop}\label{prop:frag}
Let $\alpha\in[0,1)$, $s>0$, $\beta\in\N$, and $\nu>0$ be fixed.
Let $Q\in\mathcal{Q}_j$, $j\in\N_0$, be a square and $\theta_Q$ the curvelet coefficient sequence of the fragment $f_Q=f\omega_Q$ defined in \eqref{eq:Qsequence}.
There is a constant $C>0$ independent of $j\in\N_0$ and $Q\in\mathcal{Q}_j$ such that for all $f\in\straight$ the following estimates hold true.
\begin{enumerate}
\item[(i)]
If $Q\in\mathcal{Q}_j^{0}$ the sequence $\theta_Q$ satifies
$\| \theta_Q \|^{2/(1+\beta)}_{w\ell^{2/(1+\beta)}}\le C \cdot 2^{-2js\alpha}$.
\item[(ii)]
If $Q\in\mathcal{Q}_j^{1}$ the sequence $\theta_Q$ satisfies
$\| \theta_Q \|^{2/(1+\beta)}_{w\ell^{2/(1+\beta)}}\le C \cdot 2^{-js\alpha} 2^{j\rho}$ with
$\rho=s\max\{\alpha\beta-1,0\}/(1+\beta)$.
\end{enumerate}
\end{prop}

\noindent
A direct consequence of Proposition~\ref{prop:frag}, whose proof is given later on, is Proposition~\ref{prop:main_sequence}.

\begin{proof}[Proof of Proposition~\ref{prop:main_sequence}]
We have the decomposition $\left\lb\theta_\mu\right\rb_{\mu\in\curveind_j}=\sum_{Q\in\mathcal{Q}_j} \theta_Q$.
Since $0<2/(1+\beta)\le1$, the $p$-triangle inequality with $p=2/(1+\beta)$ yields
\begin{align*}
\|\left\lb\theta_\mu\right\rb_{\mu\in\curveind_j}\|^{2/(1+\beta)}_{w\ell^{2/(1+\beta)}} \le \sum_{Q\in\mathcal{Q}_j} \|\theta_Q\|^{2/(1+\beta)}_{w\ell^{2/(1+\beta)}}
\le \big(\#\mathcal{Q}_j^{0}\big) \cdot \sup_{Q\in\mathcal{Q}_j^{0}} \|\theta_Q\|^{2/(1+\beta)}_{w\ell^{2/(1+\beta)}}
+ \big(\#\mathcal{Q}_j^{1}\big) \cdot \sup_{Q\in\mathcal{Q}_j^{1}} \|\theta_Q\|^{2/(1+\beta)}_{w\ell^{2/(1+\beta)}}.
\end{align*}
Since $f$ is supported in $[-1,1]^2$, there are constants $C_0,C_1>0$, independent of scale, such that
\begin{align*}
\#\mathcal{Q}_j^{0}\le C_{0} 2^{2js\alpha} \quad\text{and}\quad \#\mathcal{Q}_j^{1}\le C_{1} 2^{js\alpha}.
\end{align*}
Utilizing the estimates of Proposition~\ref{prop:frag}, we thus obtain with $\rho=s\max\{\alpha\beta-1,0\}/(1+\beta)\ge 0$
\begin{align*}
\|\left\lb\theta_\mu\right\rb_{\mu\in\curveind_j}\|^{2/(1+\beta)}_{w\ell^{2/(1+\beta)}}
\lesssim C_0 + C_1 2^{j\rho} \lesssim 2^{j\rho}.
\tag*{\qedhere}
\end{align*}
\end{proof}

\noindent
In the remainder of this section we are concerned with
the proof of Proposition~\ref{prop:frag}.
Hereby, we restrict to functions $f\in\straight$ of the simple form
\begin{align}\label{eqdef:cartsimp}
f=g \chi_{H(\varphi,c)}
\end{align}
with $g\in C_0^\beta([-1,1]^2,\nu)$ and $H(\varphi,c)\in\stredge$ a half-space determined by $\varphi\in[0,2\pi)$ and $c\in\R$.
Note that for a general cartoon $f=f_{1}+f_{2}\chi_{H(\varphi,c)}$ both components $\widetilde{f}_{1}:=f_{1}$ and $\widetilde{f}_{2}:=f_{2}\chi_{H(\varphi,c)}$ have the form~\eqref{eqdef:cartsimp}, due to
the representation $f_{1}=f_{1}\chi_{H(0,-1)}$.

Hence,
if the estimates of Proposition~\ref{prop:frag} are proven for elements of type~\eqref{eqdef:cartsimp},
they are then also true for all $f\in\straight$. This is a consequence of the estimate
$2^{-2js\alpha} \le 2^{-js\alpha} 2^{j\rho}$ and
\[
\|\theta_Q  \|^{2/(1+\beta)}_{w\ell^{2/(1+\beta)}} \le \| \lb\langle \widetilde{f}_{1}\omega_Q, \psi_{\mu} \rangle\rb_{\mu\in\curveind_j}  \|^{2/(1+\beta)}_{w\ell^{2/(1+\beta)}} + \|\lb\langle \widetilde{f}_{2}\omega_Q, \psi_\mu\rangle\rb_{\mu\in\curveind_j}  \|^{2/(1+\beta)}_{w\ell^{2/(1+\beta)}}.
\]
Let $Q\in\mathcal{Q}_j$ be a cube at scale $j\in\N_0$
with center $M_Q:=2^{-js\alpha}(k_1,k_2)\in\R^2$,
which nontrivially intersects the cartoon domain $[-1,1]^2$.
If $Q\in\mathcal{Q}_j^{0}$ we put $P_Q:=M_Q$.
If $Q\in\mathcal{Q}_j^{1}$ let us fix a point $P_Q\in Q$ on the edge curve $\{(x_1,x_2)\in\R^2 : x_1\cos(\varphi) - x_2\sin(\varphi)=c\}$ of the cartoon such that
$\chi_{H(\varphi,c)}=H(R_{\varphi}(x-P_Q))$, with rotation matrix \eqref{eq:matrixrot} and where
\begin{align}\label{eq:bivstep}
H:=\mathfrak{H} \otimes 1 \quad\text{with the Heaviside function}\quad
\mathfrak{H}(t)=\begin{cases} 0\,, \quad& \text{if }t<0, \\ 1\,, & \text{if }t\ge 0.\end{cases}
\end{align}
Putting $\widetilde{g}_Q(x):=g(R^{-1}_\varphi x + P_Q)$ and $\widetilde{\omega}_Q(x):=\omega_Q(R^{-1}_\varphi x + P_Q)$,
the fragment $f_Q$ can then be written as
$f_Q(x)=\fqtilde\big(R_{\varphi} ( x - P_Q )\big)$
with a function $\fqtilde$ of the form
\begin{align}\label{eq:stdedge}
(i) \quad \fqtilde:=\widetilde{g}_Q\widetilde{\omega}_Q \,,\quad
\text{if $Q\in \mathcal{Q}_j^{0} $}, \qquad\quad\text{or} \qquad\quad (ii) \quad  \fqtilde:=\widetilde{g}_Q \widetilde{\omega}_Q H \,,\quad  \text{if $Q\in \mathcal{Q}_j^{1}$}.
\end{align}
On the Fourier side we have
\begin{align*}
\widehat{f}_Q(\xi)=\widehat{\fqtilde}(R_{\varphi} \xi) \exp\big(-2\pi i\langle P_Q , \xi \rangle\big).
\end{align*}
Now, let $\psi_\mu=\psi_{j,\ell,k}\in\curvesys$ be a fixed curvelet and recall
$
\widehat{\psi}_{j,\ell,k}=W_{j,\ell} u_{j,\ell,k}
$
with the real-valued wedge functions $W_{j,\ell}(\cdot)=W_{j,0}(R_{j,\ell}\cdot)$ from \eqref{eq:suppfunctions} and the functions
\[
u_{j,\ell,k}(\cdot)=2^{-js(1+\alpha)/2} \exp(2\pi i \langle R^{-1}_{j,\ell}A^{-1}_jk  , \cdot \rangle ).
\]
There are unique $k_{\bullet}\in\Z^2$ and $\Delta k\in[0,1)^2$ such that $P_Q=R^{-1}_{j,\ell}A^{-1}_{j}(k_{\bullet}+\Delta k)$.
Further, we can express $\varphi$ as a `fractional multiple' of the angle $\varphi_j$ defined in \eqref{eqdef:fundangle}, writing
$\varphi=(\ell_{\bullet} - \Delta\ell) \varphi_j$ with unique $\ell_{\bullet}\in\Z$ and $\Delta\ell\in[0,1)$.
It follows for the curvelet coefficient $\langle f_Q, \psi_{j,\ell,k}\rangle=\langle \widehat{f}_Q, \widehat{\psi}_{j,\ell,k}\rangle$
\begin{align*}
\langle f_Q, \psi_{j,\ell,k}\rangle
&= \int_{\R^2} \widehat{\fqtilde}\big(R_{j,\ell_{\bullet} - \Delta\ell} \xi\big) \exp\big(-2\pi i\langle R^{-1}_{j,\ell}A^{-1}_{j}(k_{\bullet}+\Delta k) , \xi \rangle\big)
W_{j,\ell}(\xi) \overline{u_{j,\ell,k}(\xi)} \,d\xi \\
&= \int_{\R^2} \widehat{\fqtilde}(\xi) W_{j,\ell-\ell_{\bullet}+\Delta\ell}(\xi) \overline{u_{j,\ell-\ell_{\bullet}+\Delta\ell,k+k_{\bullet}+\Delta k_{\bullet}}(\xi)} \,d\xi.
\end{align*}
Relabelling the indices $({\bf l},{\bf k}):=([\ell-\ell_{\bullet}],k+k_{\bullet})$, where $[\ell-\ell_{\bullet}]\in\{-L_j^-,\ldots,L_j^+\}$ is the unique number
obtained by shifting $\ell-\ell_{\bullet}\in\Z$ by integer multiples of $L_j=\pi \varphi^{-1}_j$ (see \eqref{eq:Lj}), we can write
\begin{align}\label{eq:reducedform}
\langle f_Q, \psi_{j,\ell,k}\rangle
&= \int_{\R^2} \widehat{\fqtilde}(\xi) W_{j,{\bf l}+\Delta\ell}(\xi) \overline{u_{j,{\bf l}+\Delta\ell,{\bf k}+\Delta k}(\xi)} \,d\xi.
\end{align}
To estimate the integral~\eqref{eq:reducedform} we need knowledge about the Fourier localization of the functions $\fqtilde$.
This investigation is carried out in the next two subsections.

\subsection{Fourier Analysis of Standard Fragments}
\label{ssec:FourierStd}

The Fourier analysis of the functions $\fqtilde$, $Q\in\mathcal{Q}_j$, from \eqref{eq:stdedge} is conducted in a generic setting, independent of the concrete cube $Q$.
We assume $\alpha\in[0,1)$, $\beta\in\N_0$, and let $\kappa,\nu,\tilde{\nu}>0$ be fixed parameters. Then we consider functions $f_j$, $j\in\N_0$,
called standard fragments, defined by
\begin{align}\label{eq:stdfrags}
(i) \quad f_j:=g \omega_j \;, \qquad\text{or}\qquad (ii) \quad f_j:=g \omega_j H \,,
\end{align}
where $H$ is the step function \eqref{eq:bivstep},
$g\in C_0^\beta(\kappa[-1,1]^2,\nu)$ and $\omega_j:=\omega(2^{js\alpha}\cdot)$ with
$\omega\in C^\infty(\R^2) \cap C_0^\beta(\kappa[-1,1]^2,\tilde{\nu})$.
For every $Q\in\mathcal{Q}_j$ the corresponding fragment $\fqtilde$
is of the form \eqref{eq:stdfrags} with specific functions $g$ and $\omega$, namely $g=\widetilde{g}_Q$ and $\omega=\widetilde{\omega}_Q(2^{-js\alpha}\cdot)$
(compare to \eqref{eq:stdedge}).
Note that the parameters $\kappa,\nu,\tilde{\nu}>0$ can be chosen simultaneously for all $Q\in\mathcal{Q}_j$, e.g.\ $\kappa=2\sqrt{2}$, and $\tilde{\nu},\,\nu>0$ chosen suitably depending
solely on $f\in\straight$ and the partition of unity $\lb \omega_Q \rb_Q$ utilized in \eqref{eqdef:frag}.
Since the results of this subsection
are valid uniformly for all choices of $g$ and $\omega$, as long as
they fulfill the specifications in accordance with $\kappa,\nu,\tilde{\nu}>0$, they hence apply to all fragments $\fqtilde$.

The investigation starts with an elementary lemma, where $\mathcal{I}_j$, $j\in\N_0$, denote the dyadic intervals introduced in \eqref{eqdef:dyintervals}.

\begin{lemma}\label{lem:bas1ic}
Let $s>0$ be fixed and for $j\in\N_0$ let $f_j$ be fragments of the form~\eqref{eq:stdfrags}.
Then there exists a constant $C>0$ independent of $j\in\N_0$ and the
concrete choice of the functions $g$ and $\omega$ in~\eqref{eq:stdfrags} such that
for every $p\in\N_0$ and $\varphi\in [-\pi,\pi)$
\begin{align*}
\int_{\mathcal{I}_{p}} |\widehat{f_j}(r,\varphi)|^2 \,dr \le C \varepsilon_{j,p}^2(\varphi) 2^{-ps} 2^{-2js\alpha} \|g\|^2_{\infty} \|\omega\|_2^2
\end{align*}
with functions $\varepsilon_{j,p}:[-\pi,\pi)\rightarrow \R$
satisfying
$\sum_{p\in\N_0} \int_{-\pi}^{\pi} \varepsilon_{j,p}^2(\varphi) \,d\varphi\le 1$.
\end{lemma}
\begin{proof}
Let us assume $\|g\|_{\infty}\neq0$ and $\|\omega\|_{\infty}\neq0$, otherwise the proof is trivial.
Since for every $p,j\in\N_0$ and $\varphi\in [-\pi,\pi)$
\[
I_{j,p}(\varphi):=\int_{\mathcal{I}_{p}} |\widehat{f_j}(r,\varphi)|^2 \,dr < \infty
\]
we can define functions $\epsilon_{j,p}:[-\pi,\pi)\rightarrow \R$ via
$\epsilon^2_{j,p}(\varphi):= I_{j,p}(\varphi) 2^{ps} 2^{2js\alpha}  \|g\|^{-2}_{\infty} \|\omega\|^{-2}_2$.
Then
\[
I_{j,p}(\varphi)= \epsilon_{j,p}^2(\varphi) 2^{-ps} 2^{-2js\alpha} \|g\|^{2}_{\infty} \|\omega\|^{2}_2.
\]
Let us prove that there is a constant $C>0$, independent of the relevant parameters, such that
\begin{align}\label{eq:sumboundC}
\sum_{p\in\N_0} \int_{-\pi}^{\pi} \epsilon^2_{j,p}(\varphi) \,d\varphi \le C.
\end{align}
We put $\fjtilde=f_j(2^{-js\alpha}\cdot)$. Then
$\widehat{f_j}= 2^{-2js\alpha} \widehat{\fjtilde}(2^{-js\alpha}\cdot)$ and it follows for $p\in\N_0$
\[
\|\widehat{f_j}\|^2_{L^2(\corona_p)} = 2^{-2js\alpha} \|\widehat{\fjtilde}\|^2_{L^2(2^{-js\alpha}\corona_p)},
\]
where $\corona_p$ are the coronae defined in \eqref{eqdefcorC}.
We conclude
\begin{align*}
\|g\|^{2}_{\infty} \|\omega\|^{2}_2 \sum_{p\in\N_0} \int_{-\pi}^{\pi} \epsilon^2_{j,p}(\varphi) \,d\varphi &=  \sum_{p\in\N_0} 2^{2js\alpha} \int_{-\pi}^{\pi} I_{j,p}(\varphi) 2^{ps}\,d\varphi \\
&\asymp \sum_{p\in\N_0} 2^{2js\alpha} \|\widehat{f_j}\|^2_{L^2(\corona_p)}
= \sum_{p\in\N_0}  \|\widehat{\fjtilde}\|^2_{L^2(2^{-js\alpha}\corona_p)} \asymp \|\widehat{\fjtilde}\|^2_{2} = \|\fjtilde\|^2_{2}.
\end{align*}
Using $\|\fjtilde\|_{2}\le \| g(2^{-js\alpha}\cdot)\|_{\infty} \|\omega\|_2 = \|g\|_{\infty} \|\omega\|_2 $ we arrive at \eqref{eq:sumboundC}.
Finally, note that the functions $\varepsilon_{j,p}:=C^{-1/2} \epsilon_{j,p}$ have properties as desired.
\end{proof}

\noindent
An immediate consequence of Lemma~\ref{lem:bas1ic} is the following corollary, with particular choice $j=p$.

\begin{cor}\label{cor:1}
Let $s>0$ be fixed and assume that $f_j$, $j\in\N_0$, are fragments of the form \eqref{eq:stdfrags}.
There exist functions $\varepsilon_{j}:[-\pi,\pi)\rightarrow \R$, each with the property $\int_{-\pi}^{\pi} \varepsilon_{j}^2(\varphi) \,d\varphi\le 1$, and a constant $C>0$ such that
for every $j\in\N_0$ and $\varphi\in [-\pi,\pi)$
\begin{align*}
\int_{\mathcal{I}_{j}} |\widehat{f_j}(r,\varphi)|^2 \,dr \le C\varepsilon_{j}^2(\varphi) 2^{-js} 2^{-2js\alpha} \|g\|^2_{\infty} \|\omega\|_2^2.
\end{align*}
Moreover, the constant $C$ can be chosen independent of the functions $\omega$ and $g$.
\end{cor}
\begin{proof}
The functions $\varepsilon_{j}:=\varepsilon_{j,j}$ obtained from Lemma~\ref{lem:bas1ic}
by choosing $p=j$ have the desired properties. In particular they satisfy $\int_{-\pi}^{\pi} \varepsilon^2_{j}(\varphi) \,d\varphi \le 1$ for every $j\in\N_0$.
\end{proof}

\noindent
Note, that the smoothness of $f_j$ did not enter the proofs of the previous two results.
By incorporating smoothness information we can strengthen Corollary~\ref{cor:1} for a smooth fragment of the form (i) in \eqref{eq:stdfrags}.

\begin{lemma}\label{lem:bas2ic}
Let $s>0$, $\alpha\in[0,1)$, and put $\gamma=\lceil 1/(1-\alpha)\rceil$. For $j\in\N_0$ let $f_j$ be a smooth fragment of the form (i) in~\eqref{eq:stdfrags} with regularity $C^\beta$, $\beta\in\N_0$.
Then there exist functions $\varepsilon_{j}:[-\pi,\pi)\rightarrow \R$ and a constant $C>0$ such that
for every $j\in\N_0$ and $\varphi\in [-\pi,\pi)$
\begin{align*}
\int_{\mathcal{I}_j} |\widehat{f_j}(r,\varphi)|^2 \,dr \le C \varepsilon_{j}^2(\varphi) 2^{-js} 2^{-2js\alpha} 2^{-2js\beta} \|g\|^2_{\beta,\infty} \|\omega\|_{\beta,2}^2
\end{align*}
with $\int_{-\pi}^{\pi} \varepsilon_{j}^2(\varphi) \,d\varphi\le 1$ for every $j\in\N_0$.
The constant $C$ can be chosen independent of $\omega$ and $g$.
\end{lemma}
\begin{proof}
If $\beta=0$ the assertion is given by Corollary~\ref{cor:1}.
For $\beta\ge1$ the statement is proved by induction on $\beta$,
whereby we restrict our considerations to $j\ge 1$ since for $j=0$ the asserted estimate is clearly true, also due to Corollary~\ref{cor:1}.

For fixed angle $\varphi\in [-\pi,\pi)$ let $\partial_r$ denote the radial derivative in the corresponding direction. Put $\widetilde{g}:=\partial_r g$, $\widetilde{\omega}:=\partial_r \omega$, and $\widetilde{\omega}_j:=\widetilde{\omega}(2^{js\alpha}\cdot)$. Then
$\partial_rf_j(\cdot,\varphi)= \widetilde{g} \omega_j + 2^{js\alpha} g \widetilde{\omega}_j$ and we conclude for $j\in\N$
\begin{align*}
2^{2js} \int_{\mathcal{I}_{j}} |\widehat{f_j}(r,\varphi)|^2 \,dr &\asymp \int_{\mathcal{I}_{j}} |r \widehat{f_j}(r,\varphi)|^2 \,dr
\lesssim \int_{\mathcal{I}_{j}} |\widehat{\partial_r f_j}(r,\varphi)|^2 \,dr \\
& \asymp \int_{\mathcal{I}_{j}} |\widehat{\widetilde{g} \omega_j}(r,\varphi)|^2 \,dr + 2^{2js\alpha} \int_{\mathcal{I}_{j}} |\widehat{g \widetilde{\omega}_j}(r,\varphi)|^2 \,dr
 =: \widetilde{I}_j^{(0)}(\varphi) + 2^{2js\alpha} I_j^{(1)}(\varphi).
\end{align*}
Hence, we get
\[
I^{(0)}_{j}(\varphi):=\int_{\mathcal{I}_{j}} |\widehat{f_j}(r,\varphi)|^2 \,dr \lesssim 2^{-2js} \widetilde{I}_j^{(0)}(\varphi) + 2^{-2js(1-\alpha)} I_j^{(1)}(\varphi).
\]
The integral $I_j^{(1)}(\varphi)$ can be estimated in the same way as $I_j^{(0)}(\varphi)$.
After $\gamma=\lceil 1/(1-\alpha)\rceil$ iterations we end up with $\widetilde{I}_j^{(0)}(\varphi)$, $\ldots$, $\widetilde{I}_j^{(\gamma-1)}(\varphi)$, and
$\widetilde{I}_j^{(\gamma)}(\varphi):=I_j^{(\gamma)}(\varphi)$. Since $\gamma\ge 1/(1-\alpha)$
it holds
\begin{align*}
I^{(0)}_j(\varphi) \lesssim 2^{-2js} \sum_{k=0}^{\gamma-1} 2^{-2js(1-\alpha)k} \widetilde{I}_j^{(k)}(\varphi) + 2^{-2js(1-\alpha)\gamma} \widetilde{I}_j^{(\gamma)}(\varphi)
\le 2^{-2js} \sum_{k=0}^{\gamma} \widetilde{I}_j^{(k)}(\varphi).
\end{align*}
Note that $g\in C_0^{\beta-1}([-\kappa,\kappa]^2)$ and
$\widetilde{g}\in C_0^{\beta-1}([-\kappa,\kappa]^2)$, with $\kappa$ the fixed parameter from~\eqref{eq:stdfrags}. Using the induction hypothesis, the expressions $\widetilde{I}_j^{(k)}$ can be estimated with corresponding functions $\varepsilon^{(k)}_j: [-\pi,\pi)\to\R$. Putting
$
\varepsilon_j:=  \sum_{k=0}^{\gamma} \varepsilon^{(k)}_j
$
yields the desired result.
\end{proof}


\noindent
Our next goal is to estimate the energy of $\widehat{f_j}$ contained in wedges $\mathcal{W}^+_J$ of the form~\eqref{eq:wedgePJ}.
However, we allow more general scale-angle pairs $J=(j,\ell)\in\circind$ from the set
\[
\circind:=\big\{ (j,\ell) ~:~ j\in\N_0,\,\ell\in[-L^-_j,L^+_j+1) \big\}.
\]
The associated orientations, given by $\phiJ=\ell \varphi_j$ with $\varphi_j=\pi 2^{-\lfloor js(1-\alpha)\rfloor-1}$ fixed as in \eqref{eqdef:fundangle}, then comprise the whole interval $[-\frac{\pi}{2},\frac{\pi}{2})$.
To formulate the next result we need the quantities
\begin{align}\label{eqdef:AJ}
A_J:=\frac{1}{2} \int_{\mathcal{A}_J}  \varepsilon_{j}^2(\varphi)  \,d\varphi, \quad J\in\circind,
\end{align}
corresponding to angular intervals $\mathcal{A}_J$ given as in \eqref{eq:angularsupp} and the functions $\varepsilon_{j}:[-\pi,\pi)\rightarrow\R$ associated to $f_j$ from Corollary~\ref{cor:1}.

\begin{lemma}\label{lem:wedgesmooth}
Let $(m_1,m_2)\in\N_0^2$ be fixed and assume that $f_j$ is of the form \eqref{eq:stdfrags}.
Further, for $J\in\circind$ let $A_J$ be the value defined in \eqref{eqdef:AJ}. Then
\begin{align*}
\|\partial^{(m_1,m_2)}\widehat{f_j}\|^2_{L^2(\mathcal{W}^+_J)} \lesssim A_J 2^{-2j(m_1+m_2)s\alpha}2^{-2js\alpha} \|g\|^2_{\infty} \|\omega\|^2_2,
\end{align*}
with an implicit constant independent of $J\in\circind$ and the functions $g$ and $\omega$.
\end{lemma}
\begin{proof}
Using Corollary~\ref{cor:1} we calculate (in the nontrivial case when $g\neq 0$ and $\omega\neq0$)
\begin{align*}
 \|g\|^{-2}_{\infty} \|\omega\|^{-2}_2 \int_{\mathcal{W}^+_J}  |\widehat{f}_{j}(\xi)|^2 \,d\xi
=   \int_{\mathcal{I}_j} \int_{\mathcal{A}_J}  |\widehat{f}_{j}(r,\varphi)|^2 r \,d\varphi\,dr 
\lesssim  2^{-2js\alpha} \int_{\mathcal{A}_J} \varepsilon_j^2(\varphi) \,d\varphi
\asymp A_J 2^{-2js\alpha} .
\end{align*}
This proves the assertion for $(m_1,m_2)=(0,0)$.
If $m=(m_1,m_2)\neq(0,0)$ we define a new window $\tilde{\omega}(x):= x^m \omega(x)$ and put $\tilde{\omega}_j(x):=\tilde{\omega}(2^{js\alpha}x)$ for $x\in\R^2$. Then
\[
x^m \omega_j(x)= 2^{-js\alpha(m_1+m_2)} \tilde{\omega}(2^{js\alpha}x) = 2^{-js\alpha(m_1+m_2)} \tilde{\omega}_j(x)\,,\quad x\in\R^2.
\]
Introducing the function $f^{\raisebox{-0.45em}[0.4mm][0.4mm]{\textasciitilde}}_j:=g \tilde{\omega}_j H$ (or in case of a smooth fragment
$f^{\raisebox{-0.45em}[0.4mm][0.4mm]{\textasciitilde}}_j:=g \tilde{\omega}_j$) we can write
\begin{align*}
\int_{\mathcal{W}^+_J} |\partial^{(m_1,m_2)}\widehat{f_j}(\xi)|^2 \,d\xi \asymp \int_{\mathcal{W}^+_J} |\widehat{x^m f_j}(\xi)|^2 \,d\xi
=  2^{-2js\alpha(m_1+m_2)} \int_{\mathcal{W}^+_J} |\widehat{f^{\raisebox{-0.45em}[0.4mm][0.4mm]{\textasciitilde}}_j}(\xi)|^2 \,d\xi .
\end{align*}
Since $f^{\raisebox{-0.45em}[0.4mm][0.4mm]{\textasciitilde}}_j$ is of the form \eqref{eq:stdfrags},
the integral on the right-hand side can be estimated as above with Corollary~\ref{cor:1}. The proof is finished since $\|\tilde{\omega}\|_2\lesssim \|\omega\|_2 $.
\end{proof}

\noindent
For the smooth fragments we can improve this result, taking into account smoothness information.

\begin{lemma}\label{lem:wedgenonsmooth}
Let $s>0$, $\alpha\in[0,1)$, and $\gamma=\lceil 1/(1-\alpha)\rceil$. For $j\in\N_0$ let $f_j$ be a smooth fragment of the form (i) in~\eqref{eq:stdfrags} with regularity $C^\beta$, $\beta\in\N_0$.
Let $J=(j,\ell)\in\circind$ be a scale-angle pair, $A_J$ be given as in \eqref{eqdef:AJ}. For $(m_1,m_2)\in\N_0^2$
\begin{align*}
\|\partial^{(m_1,m_2)}\widehat{f_j}\|^2_{L^2(\mathcal{W}^+_J)} \lesssim A_J 2^{-2j(m_1+m_2)s\alpha}2^{-2js\alpha} 2^{-2js\beta} \|g\|^2_{\beta,\infty} \|\omega\|^2_{\beta,2}.
\end{align*}
\end{lemma}
\begin{proof}
The proof is analogous to Lemma~\ref{lem:wedgesmooth}, using Lemma~\ref{lem:bas2ic} instead of Corollary~\ref{cor:1}.
\end{proof}

\noindent
To formulate the main result of this subsection we need the differential operator
\begin{align}\label{eq:diffop1}
\mathcal{L}_{J,1}&:=(\ident - 2^{2js\alpha}\mathcal{D}_{J,1}^2)(\ident - 2^{2js\alpha}\mathcal{D}_{J,2}^2),
\end{align}
where $\ident$ is the identity and the partial derivatives $\mathcal{D}_{J,1}$ and $\mathcal{D}_{J,2}$, dependent on $J\in\circind$, are given by
\begin{align}\label{eq:diffops}
\mathcal{D}_{J,1}:=\cos(\phiJ)\partial_1 + \sin(\phiJ)\partial_2 \quad\text{and}\quad \mathcal{D}_{J,2}:=-\sin(\phiJ)\partial_1 + \cos(\phiJ)\partial_2.
\end{align}
Recall that $\phiJ=\ell\varphi_j$ with $\varphi_j$ as in~\eqref{eqdef:fundangle}. Further, recall the functions $W_J$ from~\eqref{eq:suppfunctions}
with $\supp W_J\subseteq \mathcal{W}^+_J$.

\begin{prop}\label{prop:fundament1}
Let $\mathcal{L}_{J,1}$ be the differential operator~\eqref{eq:diffop1} and let $d\in\N_0$ be arbitrary but fixed.
\begin{enumerate}
\item[(i)]
An edge fragment $f_j$ of the form (ii) in \eqref{eq:stdfrags} satisfies the estimate
\begin{align*}
\int_{\R^2} |\mathcal{L}^d_{J,1}(\widehat{f}_{j}W_J)(\xi)|^2 \,d\xi \lesssim A_J 2^{-2js\alpha}.
\end{align*}
\item[(ii)]
A smooth fragment $f_j$ of the form (i) in \eqref{eq:stdfrags} satisfies the improved estimate
\begin{align*}
\int_{\R^2} |\mathcal{L}^d_{J,1}(\widehat{f}_{j}W_J)(\xi)|^2 \,d\xi \lesssim A_J 2^{-2js\alpha} 2^{-2js\beta}.
\end{align*}
\end{enumerate}
Here $A_J$ are the quantities defined in \eqref{eqdef:AJ}. The implicit constants are independent of $J\in\circind$, $\omega$ and $g$.
\end{prop}
\begin{proof}
Using the definition~\eqref{eq:diffops} of the operators $\mathcal{D}_{J,1}$ and $\mathcal{D}_{J,2}$ we obtain for $(m_1,m_2)\in\N^2_0$
\begin{align}\label{eq:formulaD1D2}
\mathcal{D}_{J,1}^{m_1}\mathcal{D}_{J,2}^{m_2} = \sum_{\substack{a_1+b_1=m_1 \\ a_2+b_2=m_2 }}  c_{a_1,a_2,b_1,b_2} (\sin\phiJ)^{a_2+b_1}(\cos\phiJ)^{a_1+b_2} \partial^{(a_1+a_2,b_1+b_2)}
\end{align}
with purely combinatorial coefficients $c_{a_1,a_2,b_1,b_2}\in\Z$. This leads to
\begin{align*}
\|\mathcal{D}_{J,1}^{m_1}\mathcal{D}_{J,2}^{m_2} \widehat{f}_{j}\|^2_{L^2(\mathcal{W}^+_J)} \le C(m_1,m_2)
\sum_{\substack{a_1+b_1=m_1 \\ a_2+b_2=m_2 }}    \| \partial^{(a_1+a_2,b_1+b_2)}\widehat{f}_{j} \|^2_{L^2(\mathcal{W}^+_J)}
\end{align*}
with a constant $C(m_1,m_2)>0$.
If $f_j$ is an edge fragment, we proceed with Lemma~\ref{lem:wedgesmooth} and deduce
\begin{align*}
\|\mathcal{D}_{J,1}^{m_1}\mathcal{D}_{J,2}^{m_2} \widehat{f}_{j}\|^2_{L^2(\mathcal{W}^+_J)} &\lesssim \sum_{\substack{a_1+b_1=m_1 \\ a_2+b_2=m_2 }}   A_J 2^{-2j(m_1+m_2)s\alpha} 2^{-2js\alpha} \lesssim  A_J 2^{-2j(m_1+m_2)s\alpha} 2^{-2js\alpha}.
\end{align*}
Let $d_1,d_2\in\N_0$.
The function
$\mathcal{D}_{J,1}^{d_1} \mathcal{D}_{J,2}^{d_2}( \widehat{f}_{j} W_J ) $
is a linear combination of terms $(\mathcal{D}_{J,1}^{m_1} \mathcal{D}_{J,2}^{m_2} \widehat{f}_{j})( \mathcal{D}_{J,1}^{n_1} \mathcal{D}_{J,2}^{n_2} W_J ) $ with
$m_1+n_1=d_1$ and $m_2+n_2=d_2$. In view of~\eqref{eq:basic_fact} and the estimate above, it holds
\begin{align*}
\|\mathcal{D}_{J,1}^{m_1} \mathcal{D}_{J,2}^{m_2} \widehat{f}_{j} \|^2_{L^2(\mathcal{W}^+_J)}\cdot \|\mathcal{D}_{J,1}^{n_1} \mathcal{D}_{J,2}^{n_2} W_J\|^2_\infty
&\lesssim  A_J 2^{-2j(m_1+m_2)s\alpha} 2^{-2js\alpha} \cdot 2^{-2sjn_1} 2^{-2sj\alpha n_2}  \\
&\le A_J 2^{-2js\alpha d_1}  2^{-2js\alpha d_2}  2^{-2js\alpha}.
\end{align*}
Using H\"{o}lder's inequality we thus obtain for $d_1,d_2\in\N_0$
\begin{align*}
\|\mathcal{D}_{J,1}^{d_1} \mathcal{D}_{J,2}^{d_2} ( \widehat{f}_{j} W_J ) \|^2_{2}
&\lesssim  A_J 2^{-2js\alpha (d_1+d_2) } 2^{-2js\alpha}.
\end{align*}
Since $\mathcal{L}^d_{J,1}(\widehat{f}_{j}W_J)$ consists of terms of the form
\[
2^{2js\alpha(d_1+d_2)} \mathcal{D}^{2d_1}_1 \mathcal{D}^{2d_2}_2
\]
with $d_1,d_2\le d$, not taking into account combinatorial coefficients, the desired estimate for each term of $\mathcal{L}^d_{J,1}(\widehat{f}_{j}W_J)$ follows.

If $f_j$ is a smooth fragment of regularity $C^\beta$, we use Lemma~\ref{lem:wedgenonsmooth} instead of
Lemma~\ref{lem:wedgesmooth}. The rest of the proof is completely analogous.
\end{proof}

\subsection{Further Preparation}

As in the previous Subsection~\ref{ssec:FourierStd}, let $\alpha\in[0,1)$, $\beta\in\N_0$, and $\kappa,\nu,\tilde{\nu}>0$ be fixed, and assume
$g\in C^\beta(\kappa[-1,1]^2,\nu)$, $\omega_j=\omega(2^{js\alpha}\cdot)$ and $\omega\in C^\infty(\R^2)\cap C^\beta(\kappa[-1,1]^2,\tilde{\nu})$.
Further, let $\delta$ denote the univariate Dirac distribution and define $\delta_{\{x_1=0\}}:=\delta\otimes 1$.
We are interested in the Fourier localization of the distributions
\begin{align}\label{eq:stddistr}
d_j:=g \omega_j \delta_{\{x_1=0\}}\,,\quad j\in\N_0.
\end{align}
The exposition is analogous to the investigation of the functions~\eqref{eq:stdfrags} in Subsection~\ref{ssec:FourierStd}.
A valuable tool is given by the following lemma, where
$\mathcal{I}_j$ are the intervals defined in~\eqref{eqdef:dyintervals}.

\begin{lemma}\label{Lemma1D}
Let $\widetilde{A}\neq0$ and $\kappa,s>0$ be fixed. Further assume that $h\in C^\beta(\R)$, $\beta\in\N_0$, is a function with $\supp h\subseteq[-\kappa,\kappa]$.
Then there are a constant $C>0$ and numbers $\eta_j\in[0,1]$, $j\in\N_0$, with $\sum_{j\in\N_0} \eta_j\le 1$ such that for every $j\in\N_0$
\[
\int_{\widetilde{A}\mathcal{I}_j} |\widehat{h}(r)|^2\,dr = C \eta_j |\widetilde{A} 2^{js}|^{-2\beta} \| h^{(\beta)}\|_{2}^2.
\]
Moreover, the constant $C$ can be chosen independent of $h$ and $\widetilde{A}$.
\end{lemma}
\begin{proof}
Define
\[
\tilde{\eta}_j:= |\widetilde{A} 2^{js}|^{2\beta} \int_{\widetilde{A}\mathcal{I}_j} |\widehat{h}(r)|^2\,dr.
\]
Then $\sum_{j\in\N_0} \tilde{\eta}_j \le C \| h^{(\beta)}\|_{2}^2 $ with a constant $C>0$ as claimed, since we can estimate
\begin{align*}
\sum_{j\in\N_0} \tilde{\eta}_j
 \asymp \sum_{j\in\N_0}  \int_{\widetilde{A}\mathcal{I}_j} |r|^{2\beta} |\widehat{h}(r)|^2\,dr
 \asymp \sum_{j\in\N_0}  \int_{\widetilde{A}\mathcal{I}_j} |\widehat{h^{(\beta)}}(r)|^2\,dr
 \lesssim \int_{\R} |\widehat{h^{(\beta)}}(r)|^2\,dr = \|h^{(\beta)}\|^2_2.
\end{align*}
In case $\| h^{(\beta)}\|_{2}\neq 0$, rescaling yields functions $\eta_j:=C^{-1}\| h^{(\beta)}\|_{2}^{-2} \tilde{\eta}_j$ as desired. The case $\| h^{(\beta)}\|_{2}= 0$
is trivial, since then $h\equiv 0$ due to $\supp h\subseteq[-\kappa,\kappa]$.
\end{proof}

\noindent
With Lemma~\ref{Lemma1D} we can prove the following result.

\begin{lemma}\label{lem:essdistr}
Let $s>0$ be fixed and $\varphi\in[-\pi,\pi)$. We have for $j\in\N_0$
\begin{align*}
\int_{\mathcal{I}_j} |\widehat{d_j}(r,\varphi)|^2 \,dr \lesssim  2^{-js\alpha} 2^{js(1-\alpha)}(1+ 2^{js(1-\alpha)}|\sin(\varphi)| )^{-2\beta-1} \|g\|^2_{\beta,\infty} \|\omega\|^2_{\beta,2}.
\end{align*}
\end{lemma}
\begin{proof}
The distribution $d_j=g\omega_j \delta_{\{x_1=0\}}$ can be written as the tensor product
$d_j= \delta \otimes h_j$ of the Dirac distribution $\delta$ with the function $h_j:=(g\omega_j)|_{\{x_1=0\}}$. Therefore, we have
\[
\widehat{d}_j= \widehat{\delta \otimes h_j} = 1 \otimes \widehat{h}_j = \widehat{h}_j \circ\pi_2,
\]
where $\pi_2:\R^2\to\R$ is the orthogonal projection onto the second variable.

Let $\varphi\in[-\pi,\pi)$ and assume first that $|\sin(\varphi)|\ge 2^{-js(1-\alpha)}$. Then $\varphi\notin\{-\pi,0\}$ and it holds
\begin{align*}
\int_{\mathcal{I}_j} |\widehat{d_j}(r,\varphi)|^2 \,dr = \int_{\mathcal{I}_j} |\widehat{h}_j(r\sin(\varphi))|^2 \,dr
= |\sin(\varphi)|^{-1}  \int_{\sin(\varphi)\mathcal{I}_j} |\widehat{h}_j(r)|^2 \,dr.
\end{align*}

\noindent
Applying Lemma~\ref{Lemma1D} with $\widetilde{A}=\sin(\varphi)$ yields
\begin{align*}
\int_{\mathcal{I}_j} |\widehat{d_j}(r,\varphi)|^2 \,dr \lesssim \eta_{j} 2^{-2js\beta} |\sin(\varphi)|^{-2\beta-1} \| h^{(\beta)}_j \|^2_2
= \eta_j 2^{-2js\alpha\beta}  2^{-2js(1-\alpha)\beta} |\sin(\varphi)|^{-2\beta-1} \| h^{(\beta)}_j \|^2_{L^2(\R)} ,
\end{align*}
where $\eta_{j}\le 1$ for every $j\in\N_0$.
Note that Lemma~\ref{Lemma1D} is applied with a different integrand $|\widehat{h}_j|^2$ at each scale.
However, the implicit constants are uniform over all $j\in\N_0$.

Applying Leibniz's rule $h^{(\beta)}_j= \sum_{\gamma\le \beta} \binom{\beta}{\gamma} \partial_2^{\gamma} g(0,\cdot) \partial_2^{\beta-\gamma} \omega_j(0,\cdot)$ we further deduce
\begin{align*}
\| h^{(\beta)}_j \|^2_2 \lesssim  \sum_{\gamma\le \beta} \| \partial_2^{\gamma}g(0,\cdot)\|^2_\infty \| \partial_2^{\beta-\gamma}\omega_j(0,\cdot) \|^2_{2}.
\lesssim 2^{-js\alpha} 2^{2js\alpha\beta} \|\omega\|^2_{\beta,2}  \|g\|^2_{\beta,\infty}.
\end{align*}
This settles the case $|\sin(\varphi)|\ge 2^{-js(1-\alpha)}$.
If $|\sin(\varphi)|< 2^{-js(1-\alpha)}$ we argue differently based on
$
\|\widehat{h}_j\|^2_\infty\le \|h_j\|^2_1 \le 2\cdot 2^{-js\alpha} \|h_j\|^2_2.
$
We deduce
\begin{align*}
\int_{\mathcal{I}_j} |\widehat{d_j}(r,\varphi)|^2 \,dr
= \int_{\mathcal{I}_j} |\widehat{h}_j(r\sin(\varphi))|^2 \,dr  \lesssim  2^{js} \|\widehat{h}_j\|^2_\infty \lesssim 2^{js(1-\alpha)} \|h_j\|^2_2 .
\end{align*}
The proof is finished since $\|h_j\|_2^2 \le \|\omega_j(0,\cdot)\|^2_{2}  \|g(0,\cdot)\|^2_{\infty} \le 2^{-js\alpha} \|\omega\|^2_{2}  \|g\|^2_{\infty} $.
\end{proof}

\noindent
Lemma~\ref{lem:essdistr} shows that the Fourier decay of $d_j$ is highly dependent on the direction $\varphi\in[-\pi,\pi)$.
It motivates the introduction of the quantity
\begin{align}\label{eqdef:lJ}
\ell_J:= 1+ 2^{js(1-\alpha)} |\sin(\phiJ)| \,,\quad J=(j,\ell)\in\circind\,,
\end{align}
where $\phiJ=\ell \varphi_j$ and $\varphi_j=\pi 2^{-\lfloor js(1-\alpha)\rfloor-1}$ is the angle in \eqref{eqdef:fundangle}. Note that $1\le\ell_J\le 1+ 2^{js(1-\alpha)}$.

Similar to the analysis of the fragments~\eqref{eq:stdfrags}, we now proceed to estimate the Fourier energy of $\widehat{d_j}$ concentrated in a wedge $\mathcal{W}^+_J$.
The following result corresponds to Lemmas~\ref{lem:wedgesmooth} and \ref{lem:wedgenonsmooth}.

\begin{lemma}\label{lem:refine2}
Let $J\in\circind$ be a scale-angle pair, $\ell_J$ the associated quantity~\eqref{eqdef:lJ}. For $(m_1,m_2)\in\N_0^2$
\begin{align*}
\|\partial^{(m_1,m_2)}\widehat{d_j}\|^2_{L^2(\mathcal{W}^+_J)}  \lesssim \|g\|^2_{\beta,\infty} \|\omega\|^2_{\beta,2} \begin{cases} 0 \quad&,\, m_1\neq0, \\ 2^{-2jm_2s\alpha} 2^{js(1-\alpha)}  \ell_J^{-2\beta-1} &,\, m_1=0. \end{cases}
\end{align*}
The implicit constant is independent of $J\in\circind$ and $g$ and $\omega$.
\end{lemma}
\begin{proof}
If $m_1\neq0$ the assertion follows from $\partial^{m_1}_1 \widehat{d}_j=\partial^{m_1}_1 \big( \widehat{h}_j \circ\pi_2 \big) = 0$.
To handle the case $m_1=0$, let us introduce the modified window $\tilde{\omega}(x)=x_2^{m_2} \omega(x)$ and its rescaled versions $\tilde{\omega}_j=\tilde{\omega}(2^{js\alpha}\cdot)$.
Then $\tilde{\omega}_j(x)= 2^{js\alpha m_2}x_2^{m_2}\omega_j(x)$, and as a consequence
\[
\partial^{m_2}_2 \widehat{d}_j = (-2\pi i)^{m_2}\widehat{x_2^{m_2} d_j}
= (2\pi i)^{m_2} 2^{-js\alpha m_2 } \widehat{\djtilde}
\]
with $\djtilde:=g \tilde{\omega}_j \delta_{\{x_1=0\}}$ of the form \eqref{eq:stddistr}.
Hence, we can apply Lemma~\ref{lem:essdistr}, which yields
\begin{align*}
\int_{\mathcal{W}^+_J}  |\partial_2^{m_2}\widehat{d_{j}}(\xi)|^2 \,d\xi
&\asymp 2^{-2js\alpha m_2} \int_{\mathcal{I}_j}\int_{\mathcal{A}_J}  |\widehat{d^{\raisebox{-0.45em}[0.4mm][0.4mm]{\textasciitilde}}_j}(r,\varphi)|^2 r \,d\varphi\,dr \\
&\lesssim 2^{-2js\alpha m_2} \|g\|^2_{\beta,\infty} \|\tilde{\omega}\|^2_{\beta,2} \int_{\mathcal{A}_J}  2^{2js(1-\alpha)} (1+2^{js(1-\alpha)}|\sin(\varphi)|)^{-2\beta-1} \,d\varphi \\
&\lesssim 2^{-2js\alpha m_2}  2^{js(1-\alpha)} \ell_J^{-2\beta-1}  \|g\|^2_{\beta,\infty} \|\omega\|^2_{\beta,2}.
\tag*{\qedhere}
\end{align*}
\end{proof}

\noindent
Next we utilize the differential operator
\begin{align}\label{eq:diffop2}
\mathcal{L}_{J,2}&:=(\ident - 2^{2js} \ell^{-2}_J\mathcal{D}_{J,1}^2)(\ident - 2^{2js\alpha}\mathcal{D}_{J,2}^2),
\end{align}
where we use the same notation as in the definition of the operator \eqref{eq:diffop1}.
Similar to Proposition~\ref{prop:fundament1} we obtain the following result.

\begin{prop}\label{prop:stredge}
Let $\mathcal{L}_{J,2}$ be the differential operator \eqref{eq:diffop2}, $J\in\circind$, and $d\in\N_0$. We have
\begin{align*}
\int_{\R^2} |\mathcal{L}^d_{J,2}(\widehat{d}_{j}W_J)(\xi)|^2 \,d\xi \lesssim  2^{js(1-\alpha)} \ell_J^{-2\beta-1}.
\end{align*}
The implicit constant is independent of $J\in\circind$, $\omega$ and $g$.
\end{prop}
\begin{proof}
Let $(m_1,m_2)\in\N_0^2$. In view of~\eqref{eq:formulaD1D2} and Lemma~\ref{lem:refine2} we obtain
\begin{align*}
\|\mathcal{D}_{J,1}^{m_1}\mathcal{D}_{J,2}^{m_2} \widehat{d}_{j}\|^2_{L^2(\mathcal{W}^+_J)} &\lesssim
\sum_{\substack{a_1+b_1=m_1 \\ a_2+b_2=m_2 }}   |\sin(\phiJ)|^{2(a_2+b_1)} \| \partial^{(a_1+a_2,b_1+b_2)}\widehat{d}_{j} \|^2_{L^2(\mathcal{W}^+_J)} \\
&=|\sin(\phiJ)|^{2m_1} \| \partial^{(0,m_1+m_2)}\widehat{d}_{j} \|^2_{L^2(\mathcal{W}^+_J)} \\
&\lesssim  |\sin(\phiJ)|^{2m_1} 2^{-2j(m_1+m_2)s\alpha} 2^{js(1-\alpha)} \ell_J^{-2\beta-1}.
\end{align*}
Using $|\sin(\phiJ)|\le 2^{-js(1-\alpha)} \ell_J$,
we can further deduce
\begin{align*}
\|\mathcal{D}_{J,1}^{m_1}\mathcal{D}_{J,2}^{m_2} \widehat{d}_{j}\|^2_{L^2(\mathcal{W}^+_J)}
\lesssim 2^{-2jm_1s} 2^{-2jm_2s\alpha} 2^{js(1-\alpha)} \ell_J^{2m_1-2\beta-1}.
\end{align*}
The function
$\mathcal{D}_{J,1}^{d_1} \mathcal{D}_{J,2}^{d_2}( \widehat{d}_{j} W_J ) $
is a linear combination of terms $(\mathcal{D}_{J,1}^{m_1} \mathcal{D}_{J,2}^{m_2} \widehat{d}_{j})(\mathcal{D}_{J,1}^{n_1} \mathcal{D}_{J,2}^{n_2} W_J ) $ with
$m_1+n_1=d_1$ and $m_2+n_2=d_2$. They satisfy
\begin{align*}
\|\mathcal{D}_{J,1}^{m_1} \mathcal{D}_{J,2}^{m_2} \widehat{d}_{j} \|^2_{L^2(\mathcal{W}^+_J)}\cdot \|\mathcal{D}_{J,1}^{n_1} \mathcal{D}_{J,2}^{n_2} W_J\|^2_\infty
&\lesssim 2^{-2jm_1s} 2^{-2jm_2s\alpha} 2^{js(1-\alpha)} \ell_J^{2m_1-2\beta-1} \cdot 2^{-2jsn_1} 2^{-2js\alpha n_2} \\
&= 2^{-2jsd_1} 2^{-2sj\alpha d_2} 2^{js(1-\alpha)} \ell_J^{2m_1-2\beta-1}.
\end{align*}
Using H\"{o}lder's inequality,
it follows for $d_1,d_2\in\N_0$
\begin{align*}
\|\mathcal{D}_{J,1}^{d_1} \mathcal{D}_{J,2}^{d_2} ( \widehat{d}_{j} W_J ) \|^2_{2}
\lesssim   2^{-2jsd_1} 2^{-2js\alpha d_2} 2^{js(1-\alpha)} \ell_J^{2d_1-2\beta-1}.
\end{align*}
This proves the desired estimate for each term of $\mathcal{L}^d_{J,2}(\widehat{d}_{j}W_J)$, since these are of the form
\[
2^{2jsd_1}2^{2js\alpha d_2} \ell_J^{-2d_1} \mathcal{D}^{2d_1}_1 \mathcal{D}^{2d_2}_2(\widehat{d}_{j}W_J) \quad\text{with $d_1,d_2\le d$.}
\tag*{\qedhere}
\]
\end{proof}

\subsection{Proof of Proposition~\ref{prop:frag}}

\noindent
After the preparation of the preceding two subsections we now turn back to the proof of Proposition~\ref{prop:frag}.
Due to the assumptions, $\alpha\in[0,1)$, $s>0$, $\beta\in\N$, $\nu>0$ are fixed and
$f\in\straight$ is of the simplified form \eqref{eqdef:cartsimp}.
Further recall that for a cube $Q\in\mathcal{Q}_j$, $j\in\N_0$, the notation $f_{Q}$ is used for the associated fragment~\eqref{eqdef:frag}.

Instead of the sequence $\theta_Q=\{\theta_\mu\}_{\mu\in\curveind_j}$, we will analyze the relabelled sequence $\tilde{\theta}_Q:=\{\tilde{\theta}_\mu\}_{\mu\in\curveind_j}$
with elements $\tilde{\theta}_{j,\ell,k}:=\theta_{j,[\ell+\ell_{\bullet}],k-k_{\bullet}}$, where we use the notation introduced at the end of
Subsection~\ref{ssec:curvesparse}. Recall that the quantities $\ell_{\bullet}\in \Z$, $k_{\bullet}\in\Z^2$ are determined by $Q\in\mathcal{Q}_j$. In view of \eqref{eq:reducedform}, we then have
\begin{align}\label{relabelledcoeff}
\tilde{\theta}_{j,\ell,k}= \int_{\R^2} \widehat{\fqtilde}(\xi) W_{j,\ell+\Delta\ell}(\xi) \overline{u_{j,\ell+\Delta\ell,k+\Delta k}(\xi)} \,d\xi
\end{align}
with fixed $\Delta k\in[0,1)^2$, $\Delta\ell\in[0,1)$ depending on $Q\in\mathcal{Q}_j$. We define $\Delta J:=(0,\Delta\ell)$
and $J_+:=J+\Delta J$ for scale-angle pairs $J=(j,\ell)\in\mathbb{J}$. Further, we define for $J=(j,\ell)\in\mathbb{J}$ and $K=(K_1,K_2)\in\Z^2$ the sets
\begin{align}\label{eq:sumindex}
\begin{aligned}
\mathfrak{Z}^{Q}_{J,K}&:=\Big\{ (k_1,k_2) \in \Z^2 ~:~ \ell_{J+\Delta J}^{-1}(k_1+\Delta k_1)\in[K_1,K_1+1),\, k_2+\Delta k_2\in [K_2,K_2+1) \Big\}, \\
\widetilde{\mathfrak{Z}}^{Q}_{J,K}&:=\Big\{ (k_1,k_2) \in \Z^2 ~:~ 2^{-js(1-\alpha)}(k_1+ \Delta k_1)\in[K_1,K_1+1),\, k_2+ \Delta k_2\in [K_2,K_2+1) \Big\}.
\end{aligned}
\end{align}
In the definition of $\mathfrak{Z}^{Q}_{J,K}$ the quantity $\ell_{J+\Delta J}=1+2^{-js(1-\alpha)|}\sin(\phiJdelJ)|$ is used, with angle $\phiJdelJ=(\ell+\Delta\ell)\varphi_j$ and $\varphi_j$
as in \eqref{eqdef:fundangle}.
To shorten notation, it is further useful to henceforth abbreviate
\begin{align}\label{eq:quantLK}
L_K:=(1+ K_1^2)(1+ K^2_2).
\end{align}

\noindent
Essential for the proof of Proposition~\ref{prop:frag}, especially part~(ii), is the following lemma which disentangles
the smooth contribution from the singular part.

\begin{lemma}\label{lem:induction}
Let $j\in\N_0$ and $Q\in\mathcal{Q}_j$ be fixed. Under the assumptions of Proposition~\ref{prop:frag},
the relabelled coefficients $\tilde{\theta}_Q=\{ \tilde{\theta}_{\mu}\}_{\mu\in\curveind_j}$
given by \eqref{relabelledcoeff} can be decomposed in the form
\begin{align*}
\tilde{\theta}_{\mu}= a_{\mu} + b_{\mu},\quad \mu\in\curveind_j,
\end{align*}
such that for every $J\in\mathbb{J}$ with $|J|=j$
and every $K\in\Z^2$, with a uniform constant and $d\in\N_0$ fixed,
\begin{align*}
\sum_{k\in\mathfrak{Z}^Q_{J,K}} |a_{j,\ell,k}|^2   \lesssim L_K^{-2d}   2^{-js(1+\alpha)} \ell_J^{-2\beta-1} \quad\text{and}\quad
\sum_{k\in\widetilde{\mathfrak{Z}}^Q_{J,K}} |b_{j,\ell,k}|^2 \lesssim L_K^{-2d}  \widetilde{A}_J 2^{-2js\alpha} 2^{-2js\beta}.
\end{align*}
Here $L_K$ is the quantity defined in \eqref{eq:quantLK}, $\mathfrak{Z}^{Q}_{J,K}$ and $\widetilde{\mathfrak{Z}}^{Q}_{J,K}$ are given by \eqref{eq:sumindex}, and $\widetilde{A}_J\in[0,1]$ are numbers with $\sum_{|J|=j} \widetilde{A}_J\le 1$.
If $f_Q$ is a smooth fragment, a possible decomposition is given by $a_{\mu}:=0$ and $b_{\mu}:=\tilde{\theta}_{\mu}$ for $\mu\in\curveind_j$.
\end{lemma}
\noindent
It is important to note that the implicit constants in Lemma~\ref{lem:induction} can be chosen uniformly for all $j\in\N_0$ and $Q\in\mathcal{Q}_j$.
\begin{proof}
Recall, that the functions $u_{J,k}$, $J\in\circind$, are obtained by rotation of the function
\[
u_{j,0,k}(\xi)=2^{-js(1+\alpha)/2} \exp\big(\langle 2\pi i (2^{-js}k_1,2^{-js\alpha}k_2), \xi \rangle\big), \quad\xi\in\R^2.
\]
Hence $\mathcal{D}_{J,1} u_{J,k} = (2\pi i) 2^{-js} k_1 u_{J,k}$
and $\mathcal{D}_{J,2} u_{J,k} = (2\pi i) 2^{-js\alpha } k_2 u_{J,k}$ for each $J\in\circind$.
We thus establish
\begin{align*}
\mathcal{L}_{J,1} u_{J,k} = \big(1+ (2\pi)^2 2^{-2js(1-\alpha)} k_1^2\big) \big(1+ (2\pi)^2 k^2_2 \big) u_{J,k}
\end{align*}
for the differential operator $\mathcal{L}_{J,1}$ defined in~\eqref{eq:diffop1}.
Applying partial integration, we obtain from \eqref{relabelledcoeff}
\begin{align*}
\tilde{\theta}_{J,k}
= \big(\big(1+ 4\pi^2 2^{-2js(1-\alpha)}(k_1+ \Delta k_1)^2\big)\big(1+ (2\pi)^2(k_2 + \Delta k_2)^2\big)\big)^{-d} \int\limits_{\R^2} \mathcal{L}^d_{J_{+},1}(\widehat{\fqtilde}W_{J_{+}})(\xi)
\overline{u_{J_{+},k+\Delta k}}(\xi) \,d\xi.
\end{align*}
Further, since
\[
u_{J+\Delta J,k+\Delta k}(\xi) = u_{J+\Delta J,k}(\xi) \cdot \exp\big(\langle 2\pi i (2^{-js}\Delta k_1,2^{-js\alpha}\Delta k_2), R_{J+\Delta J}\xi \rangle\big)
\]
and $\{u_{J_{+},k}\}_{k\in\Z^2}$ is an orthonormal basis for $L^2(\Xi_{J_{+}})$, we obtain for $J\in\mathbb{J}$, $|J|=j$, and $K=(K_1,K_2)\in\Z^2$
\begin{align}\label{1234}
\sum_{k\in\widetilde{\mathfrak{Z}}^{Q}_{J,K}} |\tilde{\theta}_{j,\ell,k}|^2
\le (1+ K_1^2)^{-2d}(1+ K^2_2 )^{-2d} \int_{\R^2} |\mathcal{L}^d_{J_{+},1}(\widehat{\fqtilde}W_{J_{+}})(\xi)|^2  \,d\xi.
\end{align}
In case that $\fqtilde$ is a smooth fragment, Proposition~\ref{prop:fundament1}~(ii) yields
\begin{align*}
\sum_{k\in\widetilde{\mathfrak{Z}}^Q_{J,K}} |\tilde{\theta}_{j,\ell,k}|^2
\lesssim L_K^{-2d} A_{J+\Delta J} 2^{-2js\alpha} 2^{-2js\beta}.
\end{align*}
By relabelling $\widetilde{A}_{J}:=A_{J+\Delta J}$ we get the desired result.

If $\fqtilde$ is an edge fragment, we prove the assertion by induction on $\beta$. In case $\beta=0$,
we choose $b_{\mu}:=\tilde{\theta}_{\mu}$ and $a_{\mu}:=0$. Then the assertion is fulfilled, since by \eqref{1234} and Proposition~\ref{prop:fundament1}~(i)
\begin{align*}
\sum_{k\in\widetilde{\mathfrak{Z}}^Q_{J,K}} |\tilde{\theta}_{j,\ell,k}|^2
\lesssim L_K^{-2d} A_J 2^{-2js\alpha}.
\end{align*}
For the following, let $\beta\ge1$ and note that the assertion is always fulfilled for $j=0$, also due to Proposition~\ref{prop:fundament1}~(i).

It thus remains to prove the assertion for $j,\,\beta\in\N$.
If $j\in\N$, by definition,
$W_J(\xi)=U_j(|\xi|) V_{J}(\xi/|\xi|)= U( 2^{-js}|\xi|) V_{J}(\xi/|\xi|)$. To use induction
we rewrite \eqref{relabelledcoeff} in the form
\begin{align*}
\tilde{\theta}_{J,k} = 2^{-js} \int_{\R^2} |\xi| \widehat{\fqtilde}(\xi) \frac{U(2^{-js}|\xi|)V_{J+\Delta J}(\xi/|\xi|)}{2^{-js}|\xi|} \overline{u_{J+\Delta J,k+\Delta k}(\xi)} \,d\xi.
\end{align*}
We introduce the function $\widetilde{U}(r)=\frac{U(r)}{r}$, $r\in\Rzeroplus$, and put $\widetilde{U}_j:=\widetilde{U}( 2^{-js} \cdot)$ for $j\ge 1$.
In addition, we put $\widetilde{U}_0(r)=U_0(r)$, $r\in\Rzeroplus$.
Further, we define $\widetilde{V}_{J}(\xi):=V_J(\xi) \cos(|\varphi(\xi)-\phiJ|)^{-1}$ for  $\xi\in\mathbb{S}^1$ and $J\in\circind$, $|J|\ge 1$.
For $J=(0,0)$ we define $\widetilde{V}_{J}:=V_{J}$. Note that
for $\xi\in\mathcal{A}_{J}$, $|J|\ge 1$, we have $|\varphi(\xi)-\phiJ| \le \varphi_j^+/2 \le 3\pi/8$ and thus $1\le \cos(|\varphi(\xi)-\phiJ|)^{-1} \le 3$.
For $J\in\circind$ we then define
\[
\widetilde{W}_J(\xi):= \widetilde{U}_j(|\xi|) \widetilde{V}_{J}(\xi/|\xi|), \quad \xi\in\R^2.
\]
The functions $\lb\widetilde{W}_J\rb_{J\in\mathbb{J}}$ are again wedge functions of the form \eqref{eq:suppfunctions} which satisfy condition~\eqref{eq:CalderonW}
with some (possibly different) constants $0<A\le B<\infty$.
Using these functions the coefficients take the form
\begin{align}\label{eq:curvecoeff2}
\tilde{\theta}_{J,k} = 2^{-js} \int_{\R^2} |\xi| \cos(|\varphi(\xi)-\phiJdelJ|) \widehat{\fqtilde}(\xi) \widetilde{W}_{J+\Delta J}(\xi) \overline{u_{J+\Delta J,k+\Delta k}(\xi)} \,d\xi. \quad (J,k)\in\curveind_j.
\end{align}
Now recall the directional derivative $\mathcal{D}_{J,1}=\cos(\phiJ)\partial_1 + \sin(\phiJ)\partial_2$ depending on $J\in\circind$.
For $\xi=(\xi_1,\xi_2)=(|\xi|\cos\varphi, |\xi|\sin\varphi)\in\R^2$ we have
\begin{align*}
\xi_1 \cos(\phiJ) + \xi_2 \sin(\phiJ) = |\xi| \big(\cos(\varphi)\cos(\phiJ)+ \sin(\varphi)\sin(\phiJ) \big) = |\xi| \cos(|\varphi-\phiJ|).
\end{align*}
Hence, \eqref{eq:curvecoeff2} becomes
\begin{align*}
\tilde{\theta}_{J,k}= (2\pi i)^{-1} 2^{-js} \int_{\R^2} \big(\mathcal{D}_{J_{+},1}\fqtilde\big)^{\wedge}(\xi) \widetilde{W}_{J_{+}}(\xi) \overline{u_{J_{+},k+\Delta k}(\xi)} \,d\xi.
\end{align*}
The edge fragment $\fqtilde$ is of the form $f_j=g\omega(2^{js\alpha}\cdot)H$ with $g\in C_0^\beta(\R^2)$, $\omega\in C_0^\infty(\R^2)$, and the bivariate step function $H=\mathfrak{h}\otimes 1$
(see \eqref{eq:stdedge}).
Let us define $\widetilde{g}= \mathcal{D}_{J_{+},1} g$, $\widetilde{\omega}=\mathcal{D}_{J_{+},1}\omega$, and $\widetilde{\omega}_j=\widetilde{\omega}(2^{js\alpha} \cdot)$.
Further, recall $\partial_1 H=\delta_{\{x_1=0\}}$ and note that
\[
\mathcal{D}_{J_{+},1}H = \cos(\phiJplus)\partial_1 H + \sin(\phiJplus)\partial_2 H =  \cos(\phiJplus) \delta_{\{x_1=0\}}.
\]
The product rule yields
\begin{align*}
\mathcal{D}_{J_{+},1} f_j = \widetilde{g}\omega_j H  + \cos(\phiJplus) \delta_{\{x_1=0\}}\omega_j g + 2^{js\alpha}g\widetilde{\omega}_j H = T_1 + \cos(\phiJplus) T_2 +  2^{js\alpha} T_3
\end{align*}
with terms $T_1:=\widetilde{g}\omega_j H$, $T_2:=\delta_{\{x_1=0\}}\omega_j g$, and $T_3:=g\widetilde{\omega}_j H$. This
leads to the decomposition
\begin{align}\label{eq:seqdecomp}
  \tilde{\theta}_{j,\ell,k} \asymp  2^{-js} c^{(0)}_{j,\ell,k} +  2^{-js} \cos(\phiJplus) d^{(0)}_{j,\ell,k}
+ 2^{-js(1-\alpha)} \tilde{\theta}^{(1)}_{j,\ell,k}
\end{align}
with
\begin{align*}
c^{(0)}_{j,\ell,k} &:= \int_{\R^2} \widehat{T}_1 \widetilde{W}_{J_{+}}(\xi) \overline{u_{J_{+},k+\Delta k}(\xi)} \,d\xi,  \\
d^{(0)}_{j,\ell,k} &:= \int_{\R^2} \widehat{T}_2 \widetilde{W}_{J_{+}}(\xi) \overline{u_{J_{+},k+\Delta k}(\xi)} \,d\xi, \\
\tilde{\theta}^{(1)}_{j,\ell,k} &:=  \int_{\R^2} \widehat{T}_3 \widetilde{W}_{J_{+}}(\xi) \overline{u_{J_{+},k+\Delta k}(\xi)} \,d\xi.
\end{align*}
Note that $\widetilde{g}\in C_0^{\beta-1}(\R^2)$ and $\widetilde{\omega}\in C^\infty(\R^2)$ with $\supp\widetilde{\omega}\subseteq\supp\omega$.
By induction we can decompose
\[
c^{(0)}_{\mu} = a^{(0)}_{\mu} + b^{(0)}_{\mu}, \quad \mu\in\curveind_j,
\]
where the sequences $\lb a^{(0)}_{\mu}\rb_{\mu\in\curveind_j}$ and $\lb b^{(0)}_{\mu} \rb_{\mu\in\curveind_j}$ satisfy the assertion for $\beta-1$.
The coefficients $\{ d^{(0)}_{j,\ell,k} \}_{\mu\in\curveind_j}$
can be handled with the help of Proposition~\ref{prop:stredge}. We have for the differential operator $\mathcal{L}_{J,2}$ from \eqref{eq:diffop2}
\begin{align*}
\mathcal{L}_{J,2} u_{J,k} = \big(1+ (2\pi)^2 \ell^{-2d}_J  k_1^2\big) \big(1+ (2\pi)^2 k^2_2 \big) u_{J,k}.
\end{align*}
Partial integration leads to
\[
d^{(0)}_{J,k}= \big(1+ (2\pi)^2 \ell^{-2d}_{J_{+}}  (k_1+\Delta k_1)^2\big)^{-d} \big(1+ (2\pi)^2 (k_2 + \Delta k_2)^2 \big)^{-d} \int_{\R^2}
\mathcal{L}^d_{J_{+},2}(\widehat{T}_2 \widetilde{W}_{J_{+}})(\xi) \overline{u_{J_{+},k+\Delta k}(\xi)} \,d\xi.
\]
We deduce that for every $J\in\mathbb{J}$ with $|J|=j$ and every $K=(K_1,K_2)\in\Z^2$
\begin{align*}
\sum_{k\in\mathfrak{Z}^Q_{J,K}} |d^{(0)}_{J,k}|^2 \le (L_K)^{-2d}
 \int_{\R^2} |\mathcal{L}^d_{J_{+},2}(\widehat{T}_{2} \widetilde{W}_{J_{+}})(\xi)|^2 \,d\xi
 \lesssim (L_K)^{-2d}  2^{js(1-\alpha)} \ell_{J_{+}}^{-2\beta-1}.
\end{align*}
Here we applied the fact that $\{u_{J_{+},k}\}_{k\in\Z^2}$ is an orthonormal basis for $L^2(\Xi_{J_{+}})$ and Proposition~\ref{prop:stredge}.
Finally, note that $|\sin(\phiJ)|\asymp|\phiJ|\asymp |\ell| 2^{-js(1-\alpha)}$ uniformly for $J\in\circind$. Hence, due to $\Delta\ell\in[0,1)$, $\ell_J\asymp 1+|\ell|\asymp 1+|\ell+\Delta\ell|
\asymp \ell_{J+\Delta J}$.

It remains to handle the sequence $\lb \tilde{\theta}^{(1)}_{\mu}\rb_{\mu\in\curveind_j}$
which resembles the original sequence $\lb \tilde{\theta}_{\mu} \rb_{\mu\in\curveind_j}$ and can be handled accordingly.
After $\gamma$ iterations of the decomposition process~\eqref{eq:seqdecomp} we end up with sequences
$\lb c^{(0)}_{\mu}\rb_{\mu\in\curveind_j}, \ldots, \lb c^{(\gamma-1)}_{\mu}\rb_{\mu\in\curveind_j}$, $\lb d^{(0)}_{\mu}\rb_{\mu\in\curveind_j},\ldots, \lb d^{(\gamma-1)}_{\mu}\rb_{\mu\in\curveind_j}$, and $\lb \tilde{\theta}^{(\gamma)}_{\mu}\rb_{\mu\in\curveind_j}$.
We choose $\gamma=\lceil \frac{1}{1-\alpha} \rceil$ so that
\[
2^{-js(1-\alpha)\gamma} \le 2^{-js}.
\]
We can apply the induction hypothesis on $\lb c^{(\tau)}_{\mu}\rb_{\mu\in\curveind_j}$ for every $\tau\in\{0,\ldots,\gamma-1\}$,
which leads to sequences $\lb a^{(\tau)}_{\mu}\rb_{\mu\in\curveind_j}$ and $\lb b^{(\tau)}_{\mu}\rb_{\mu\in\curveind_j}$.
Since $g\in C^\beta(\R^2) \subset C^{\beta-1}(\R^2)$
also $\lb \tilde{\theta}^{(\gamma)}_{\mu}\rb_{\mu\in\curveind_j}$ can be decomposed into two sequences
$\lb a^{(\gamma)}_{\mu}\rb_{\mu\in\curveind_j}$ and $\lb b^{(\gamma)}_{\mu}\rb_{\mu\in\curveind_j}$.

Finally, we obtain the desired decomposition $\tilde{\theta}_{\mu}= a_{\mu} + b_{\mu}$, $\mu\in\curveind_j$, with
\begin{align*}
a_{\mu}&:= 2^{-js} \sum_{\tau=0}^{\gamma-1}  2^{-js(1-\alpha)\tau} a^{(\tau)}_{\mu} + 2^{-js(1-\alpha)\gamma} a^{(\gamma)}_{\mu},  \\
b_{\mu}&:= 2^{-js} \sum_{\tau=0}^{\gamma-1}  2^{-js(1-\alpha)\tau} b^{(\tau)}_{\mu} + 2^{-js(1-\alpha)\gamma} b^{(\gamma)}_{\mu} + 2^{-js} \cos(\phiJplus) \sum_{\tau=0}^{\gamma-1}  2^{-js(1-\alpha)\tau} d^{(\tau)}_{\mu}.
\tag*{\qedhere}
\end{align*}
\end{proof}

\noindent
With Lemma~\ref{lem:induction} in our toolbox, it is not difficult any more to prove Proposition~\ref{prop:frag}.
The remaining considerations are merely interpolation arguments.

\begin{proof}[{\bf Proof of Proposition~\ref{prop:frag}}]
We first handle part (i) of the proposition, when $f_j$ is a smooth fragment.
Let $\curveind_j$ denote the curvelet indices at scale $j\in\N_0$ and define $\curveind^Q_{j,K}:=\{ (j,\ell,k)\in\curveind_j : k\in \widetilde{\mathfrak{Z}}^{Q}_{J,K} \}$ for $K\in\Z^2$.
Since $\sum_{|J|=j} A_J \lesssim 1$, Lemma~\ref{lem:induction} yields for $K\in\Z^2$
\[
\sum_{\mu\in\curveind^Q_{j,K}}  |\tilde{\theta}_{J,k}|^2 = \sum_{|J|=j} \sum_{k\in \widetilde{\mathfrak{Z}}^{Q}_{J,K}} |\tilde{\theta}_{J,k}|^2 \lesssim L_K^{-2d} 2^{-2js\alpha} 2^{-2js\beta}.
\]
Let us fix $d\in\N_0$ as the smallest integer satisfying $d>(1+\beta)/4$, i.e.,
$ d:=  \lfloor (1+\beta)/4 \rfloor +1$.
This ensures
\begin{align}\label{LKfinite}
\sum_{K\in\Z^2} L^{-2d/(1+\beta)}_K = \sum_{K\in\Z^2} \big((1+ K_1^2)(1+ K^2_2 )\big)^{-2d/(1+\beta)} \lesssim 1,
\end{align}
which will be important below. Further, note that we have the estimate
\[
\sum_{|J|=j} \# \widetilde{\mathfrak{Z}}^{Q}_{J,K} \le \sum_{|J|=j} 2^{js(1-\alpha)} \lesssim 2^{2js(1-\alpha)}.
\]
Recall the interpolation inequality $\|\lb c_\lambda\rb_{\lambda\in\Lambda}\|_{\ell^p} \le (\#\Lambda)^{1/p-1/2} \|\lb c_\lambda\rb_{\lambda\in \Lambda}\|_{\ell^2}$
valid for $0<p\le2$ and finite sequences $\lb c_\lambda\rb_{\lambda\in \Lambda}$.
Interpolation with $p=2/(1+\beta)$ yields
\[
\| \lb\tilde{\theta}_{\mu}\rb_{\mu\in\curveind^{Q}_{j,K}} \|_{2/(1+\beta)} \lesssim 2^{js\beta(1-\alpha)} (L_K)^{-d} 2^{-js\alpha} 2^{-js\beta}= (L_K)^{-d} 2^{-js\alpha(1+\beta)} .
\]
The proof of part~(i) is finished by applying the $p$-triangle inequality with $p=2/(1+\beta)\le1$.
In view of~\eqref{LKfinite} we arrive at
\[
\| \lb\tilde{\theta}_{\mu}\rb_{\mu\in\curveind_j} \|^{2/(1+\beta)}_{2/(1+\beta)} \le \sum_{K\in\Z^2} \|\lb\tilde{\theta}_{\mu}\rb_{\mu\in\curveind^Q_{j,K}} \|^{2/(1+\beta)}_{2/(1+\beta)} \lesssim   2^{-2js\alpha}.
\]
We finally turn to the proof of part~(ii) and assume that $f_j$ is an edge fragment.
We denote by $\lb a_{\mu}\rb_{\mu\in\curveind_j}$ and $\lb b_{\mu}\rb_{\mu\in\curveind_j}$ the decomposition of the sequence $\lb\tilde{\theta}_{\mu}\rb_{\mu\in\curveind_j}$
according to Lemma~\ref{lem:induction}.
Analogous to the treatment of the smooth case, one can deduce
\begin{align}\label{eq:prfsmooth}
\| \lb b_{\mu}\rb_{\mu\in\curveind_j} \|^{2/(1+\beta)}_{2/(1+\beta)} \lesssim  2^{-2js\alpha}.
\end{align}
It remains to handle $\lb a_{\mu}\rb_{\mu\in\curveind_j}$.
Due to Lemma~\ref{lem:induction} we have with $d\in\N_0$ chosen as above
\begin{align}\label{zzz}
\sum_{k\in \mathfrak{Z}^{Q}_{J,K}}|a_{j,\ell,k}|^2 \lesssim  L_K^{-2d}   2^{-js(1+\alpha)}  \ell_J^{-2\beta-1}.
\end{align}
Recall that $\ell_J= 1+ 2^{js(1-\alpha)}|\sin(\phiJ)|\ge 1$ and note that we can estimate
\begin{align}\label{yyy}
\#\mathfrak{Z}^{Q}_{J,K}\le  \ell_{J+\Delta J} \asymp \ell_{J}.
\end{align}
In view of \eqref{zzz} and \eqref{yyy} we conclude for $\varepsilon>0$
\begin{align*}
N^Q_{J,K}(\varepsilon):= \#\Big\{ k\in\mathfrak{Z}^{Q}_{J,K} ~:~ |a_{j,\ell,k}|>\varepsilon  \Big\} \lesssim  \min\Big\{ \ell_J ,  \varepsilon^{-2} L_K^{-2d}  2^{-js(1+\alpha)} \ell_J^{-2\beta-1}  \Big\}.
\end{align*}
The next step is to show
\begin{align}\label{intermediate}
\sum_{|J|=j} N^Q_{J,K}(\varepsilon) \lesssim \varepsilon^{-2/(\beta+1)} L_K^{-2d/(\beta+1)} 2^{-js(1+\alpha)/(1+\beta)}.
\end{align}
Since $\ell_J\asymp 1+|\ell|$ we can estimate, where we use the quantities $\ell^-_*:=\lceil\ell_*\rceil-1$ and $\ell^+_*:=\lceil\ell_*\rceil$ with $\ell_*:= \varepsilon^{-1/(1+\beta)} L_K^{-d/(1+\beta)} 2^{-js\frac{1+\alpha}{2(1+\beta)}}$,
\begin{align*}
\sum_{\ell=0}^{L_j^+} N_{j,\ell,K}(\varepsilon) &\lesssim \sum_{\ell=1}^{L_j^++1} \min\Big\{ \ell , \varepsilon^{-2} L_K^{-2d}  2^{-js(1+\alpha)} \ell^{-2\beta-1} \Big\}
\le  \sum_{\ell=1}^{\ell^-_*} \ell  + \sum_{\ell=\ell^+_*}^{L^+_j+1} \varepsilon^{-2} L^{-2d}_K 2^{-js(1+\alpha)} \ell^{-2\beta-1}.
\end{align*}
Note that $\ell^-_*\in\N_0$. Therefore, it holds
\[
\sum_{\ell=1}^{\ell^-_*} \ell = \frac{1}{2} \ell^-_* (\ell^-_*+1) \le \ell_*^2 = {\rm rhs}(\ref{intermediate}).
\]
Further, taking into account $\ell_*\le \ell^+_*$, we obtain
\[
\sum_{\ell=\ell^+_*}^{L^+_j+1} \varepsilon^{-2} L^{-2d}_K 2^{-js(1+\alpha)} \ell^{-2\beta-1} \lesssim \varepsilon^{-2} L^{-2d}_K 2^{-js(1+\alpha)} \ell_*^{-2\beta} = {\rm rhs}(\ref{intermediate}).
\]
Altogether, this proves \eqref{intermediate} since the sum $\sum_{\ell=-L_j^-}^{0} N^Q_{j,\ell,K}(\varepsilon)$ can be estimated analogously.

Recall that $\curveind_j$ denotes the curvelet indices at scale $j$. Using \eqref{LKfinite} we deduce from \eqref{intermediate}
\[
\# \Big\{ \mu\in\curveind_j ~:~ |a_\mu|>\varepsilon \Big\} = \sum_{K\in\Z^2} \sum_{|J|=j} N^Q_{J,K}(\varepsilon) \lesssim  2^{-js(1+\alpha)/(1+\beta)} \varepsilon^{-2/(1+\beta)}.
\]
This implies the following estimate, where we let $\rho=\max\big\{0,s(\alpha\beta-1)/(1+\beta)\big\}$,
\begin{align}\label{eq:prfedge}
\| \lb a_{\mu}\rb_{\mu\in\curveind_j}\|^{2/(1+\beta)}_{w\ell^{2/(1+\beta)}} \lesssim 2^{-js(1+\alpha)/(1+\beta)} =  2^{-js\alpha}  2^{js(\alpha\beta-1)/(1+\beta)}\le 2^{-js\alpha}  2^{j\rho}.
\end{align}

\noindent
In a last step, we combine \eqref{eq:prfsmooth} and \eqref{eq:prfedge}. Using the $p$-triangle inequality with $p=\frac{2}{1+\beta}\le1$ gives
\[
\| \lb\tilde{\theta}_{\mu}\rb_{\mu\in\curveind_j} \|^{2/(1+\beta)}_{w\ell^{2/(1+\beta)}}
\le \|\lb a_{\mu}\rb_{\mu\in\curveind_j} \|^{2/(1+\beta)}_{w\ell^{2/(1+\beta)}} + \| \lb b_{\mu}\rb_{\mu\in\curveind_j} \|^{2/(1+\beta)}_{w\ell^{2/(1+\beta)}}
\lesssim  2^{-js\alpha} 2^{j\rho} + 2^{-2js\alpha} \lesssim 2^{-js\alpha} 2^{j\rho},
\]
which finishes the proof.
\end{proof}


\section{Discussion and Extension}
\label{sec:discussion}


In this final section we interpret and discuss the results of our previous investigations.
First we note that Theorem~\ref{thm:bound2} complements the result of Theorem~\ref{thm:oldcurveappr}.
The latter guarantees at least an approximation rate of order $N^{-1/\alpha}$
for $\cart$ if $\beta\ge\alpha^{-1}$ and $\alpha\in[\frac12,1)$.
In view of Theorem~\ref{thm:benchmark} the optimal approximation order is thus realized in case $\beta=\alpha^{-1}$.
Theorem~\ref{thm:bound2}
now tells us that
this rate does not improve for $C^{\beta}$ cartoons with $\beta>\alpha^{-1}$, at least if we restrict to greedy approximations obtained by simple thresholding.
Hence, $\alpha$-curvelets in the range $\alpha\in[\frac12,1)$ cannot take advantage of cartoon regularity higher than $\alpha^{-1}$.

Turning to the range $\alpha\in[0,\frac12)$,
according to both, Theorem~\ref{thm:bound1} and Theorem~\ref{thm:bound2},
the approximation deteriorates as $\alpha$ tends to $0$.
In Theorem~\ref{thm:bound2} the achievable rate peaks for $\alpha=\frac12$,
a confirmation of the outstanding role of parabolic scaling for cartoon approximation.
Among all $\alpha$-curvelet frames, the classic parabolically scaled systems provide the best performance for $\cart$
if $\beta\ge2$.
However, if $\beta>2$ the achieved rate of order $N^{-2}$ is suboptimal.

To better understand this behavior, recall the heuristic considerations in Subsection~\ref{ssec:guarantees}.
A Taylor expansion showed that $C^\beta$ curves with $\beta\in(1,2]$ are locally contained in (properly aligned) rectangles of size $width \approx length^{1/\beta}$.
This explains why $\alpha$-scaling with $\alpha=\beta^{-1}$ is optimally suited to resolve such curves.
It also indicates that it is not the
smoothness of the curves that determines the best type of scaling, but
their local scaling behavior.
If the second-order Taylor term at some point of a $C^\beta$ curve, where $\beta\ge2$,  does not vanish the scaling locally obeys $width \approx length^{1/2}$.
Consequently, the choice $\alpha=\frac12$ is still the best for $C^\beta$ curves with $\beta\ge2$ and nonvanishing curvature.

The situation is different if the curvature vanishes. For cartoons with curved edges, however, this typically happens only at certain isolated points which
are negligible in the overall approximation.
Otherwise, in case of a straight line segment, directionally scaled $0$-curvelets provide the best approximation.
A deviation of $\alpha$ from $0$ deteriorates the approximability of the edge,
but according to Theorem~\ref{thm:mainappr1} for signals from $\straight$
this deterioration is masked by the overall approximation performance of order $N^{-\beta}$ if $\alpha\in[0,\beta^{-1}]$.

It is remarkable that
up to now no frame is known where a nonadaptive thresholding scheme yields approximation rates better than $N^{-2}$ for the class $\cart$, $\beta>2$.
As we have seen, $\alpha$-scaling is not able to take advantage of smoothness beyond $C^2$, wherefore
new ideas need to be considered.
One approach might be based on the bendlet transform~\cite{LPS2016}, which incorporates bending in addition to $\alpha$-scaling for improved adaptability to the edges.
While the bendlet dictionary seems to be useful for certain image analysis tasks, the question of how to extract bendlet frames for approximation is not clear however and requires further research.

Finally, let us derive some implications of the obtained results for other $\alpha$-scaled representation systems.
The framework of $\alpha$-molecules allows to transfer properties of $\curvesys$
to other systems of $\alpha$-molecules if their parametrization is consistent with the
parametrization $(\curveind,\Phi_{\curveind})$ of $\curvesys$ from \eqref{eq:curvepara}.
For the required notion of consistency, let us first recall the phase-space metric $\omega_\alpha$ introduced in \cite{GKKS15} for
the phase space $\mathbb{P}=\R^+\times\mathbb{T}\times\R^2$.

\begin{definition}[{\cite[Def.~4.1]{GKKS15}}]
Let $\alpha\in[0,1]$. The $\alpha$-scaled index distance $\omega_\alpha:\mathbb{P} \times\mathbb{P}\to[1,\infty)$
is defined by
\begin{align*}
\omega_\alpha( \textit{\textbf{p}}_\lambda, \textit{\textbf{p}}_\mu)=\max\Big\{\frac{s_\lambda}{s_\mu},\frac{s_\mu}{s_\lambda}\Big\}
(1+d_\alpha(\textit{\textbf{p}}_\lambda,\textit{\textbf{p}}_\mu))\,,
\end{align*}
where $\textit{\textbf{p}}_\lambda=(s_\lambda,\theta_\lambda,x_\lambda)\in\mathbb{P}$, $\textit{\textbf{p}}_\mu=(s_\mu,\theta_\mu,x_\mu)\in\mathbb{P}$, and with $s_0=\min\{s_\lambda,s_\mu\} $, $e_\lambda=(\cos(\theta_\lambda),-\sin(\theta_\lambda))$,
\[
d_\alpha(\textit{\textbf{p}}_\lambda,\textit{\textbf{p}}_\mu)=s_0^{2(1-\alpha)} |\theta_\lambda-\theta_\mu|^2+
s_0^{2\alpha} |x_\lambda-x_\mu|^2 + \frac{s^2_0}{1+s_0^{2(1-\alpha)}|\theta_\lambda-\theta_\mu|^2} |\langle e_\lambda, x_\lambda-x_\mu \rangle|^2.
\]
\end{definition}
\noindent
The consistency of two parametrizations is then defined as follows.

\begin{definition}[{\cite[Def.~5.5]{GKKS15}}]
Let $\alpha\in[0,1]$ and $k>0$. Two parametrizations $(\Lambda,\Phi_\Lambda)$ and $(\Delta,\Phi_\Delta)$, for index sets $\Lambda$ and $\Delta$ respectively,
are called $(\alpha,k)$-consistent if
\[
\sup_{\lambda\in\Lambda} \sum_{\mu\in\Delta} \omega_\alpha\big(\Phi_\Lambda(\lambda),\Phi_\Delta(\mu)\big)^{-k}<\infty
\quad\text{and}\quad
\sup_{\mu\in\Delta} \sum_{\lambda\in\Lambda} \omega_\alpha\big(\Phi_\Lambda(\lambda),\Phi_\Delta(\mu)\big)^{-k}<\infty.
\]
\end{definition}

\noindent
Since $\curvesys$ is a tight frame of $\alpha$-molecules
of arbitrary order, as shown by Lemma~\ref{thm:curvmol},
the theory of $\alpha$-molecules allows to deduce the following result practically for free.

\begin{theorem}\label{thm:mol_app}
Let $\alpha\in[0,1]$ and let $\mathfrak{M}:=\lb m_\lambda\rb_{\lambda\in\Lambda}$ be a frame of
$\alpha$-molecules whose parametrization, for some $k>0$, is $(\alpha,k)$-consistent with the $\alpha$-curvelet parametrization $(\curveind,\Phi_\curveind)$ of $\curvesys$. Further, assume that
for some $\gamma\in\R_0^+$ the order $(L,M,N_1,N_2)$ of $\mathfrak{M}$
satisfies
\begin{align}\label{disc:ordercond}
L\geq k(1+\gamma) ,\quad  M \geq \frac{3k}{2} (1+\gamma) + \frac{\alpha-3}{2} , \quad N_1 \geq \frac{k}{2} (1+\gamma) +\frac{1+\alpha}{2} , \quad N_2\geq k(1+\gamma).
\end{align}
Then the following holds true:
\begin{enumerate}
\item[(i)]
Let $\tilde{c}_\lambda:=\langle f,m_\lambda\rangle$, $\lambda\in\Lambda$, denote the analysis coefficients of $f\in\csE^{\beta}([-1,1]^2,\nu)$ with respect to $\mathfrak{M}$, and assume $\beta\in\N$.
If \eqref{disc:ordercond} is fulfilled for $\gamma=\min\{\beta,\alpha^{-1}\}$,
then $\lb\tilde{c}_\lambda\rb_{\lambda\in\Lambda}\in \ell^p(\Lambda)$ for all $p>\frac{2}{1+\gamma}$.
\item[(ii)]
Let $\Theta=\sum_{\lambda\in\Lambda} c_\lambda m_\lambda$ be a representation of the function
$\Theta$ from~\eqref{def:Theta} with respect to $\mathfrak{M}$.
If \eqref{disc:ordercond} is fulfilled for some $\gamma>\tilde{\gamma}:=\max\{\alpha,1-\alpha\}^{-1}$,
then $\lb c_\lambda\rb_{\lambda\in\Lambda}\notin\ell^p(\Lambda)$ for $p<\frac{2}{1+\tilde{\gamma}}$.
\end{enumerate}
\end{theorem}
\begin{proof}
According to \cite[Thm.~5.6]{GKKS15} condition~\eqref{disc:ordercond} ensures that the systems $\mathfrak{M}$ and $\curvesys$ are sparsity equivalent in $\ell^p$
for $p:=\frac{2}{1+\gamma}$, which means $\| (\langle m_\lambda,\psi_\mu \rangle)_{\lambda,\mu} \|_{\ell^p\to\ell^p} < \infty$  (see \cite[Def.~5.3]{GKKS15}).
Since $f=\sum_\mu \langle f,\psi_\mu \rangle \psi_\mu $ and $\lb\langle f,\psi_\mu \rangle\rb_\mu\in\ell^{p+\varepsilon}(\curveind)$, $\varepsilon>0$, by Theorem~\ref{thm:mainappr2},
assertion $(i)$ follows. For $(ii)$ assume that $\lb c_\lambda\rb_{\lambda}\in\ell^p(\Lambda)$, which implies by sparsity equivalence $\lb\langle \Theta,\psi_\mu \rangle\rb_{\mu}\in\ell^p(\curveind)$.
Using $\Theta=\sum_\mu \langle \Theta,\psi_\mu \rangle \psi_\mu $ and Lemma~\ref{lem:decayapprox}, this then implies an $N$-term approximation rate of order $N^{-\gamma}$, in contradiction to Theorem~\ref{thm:bound2}.
\end{proof}

\noindent
A direct corollary is obtained via Lemma~\ref{lem:decayapprox}.

\begin{cor}\label{cor:mol_app}
Under the assumptions of Theorem~\ref{thm:mol_app}~(i),
every dual frame $\lb\tilde{m}_\lambda\rb_{\lambda\in \Lambda}$ of $\mathfrak{M}$ yields -- via simple thresholding -- $N$-term approximations $f_N$ to $f\in\csE^{\beta}([-1,1]^2)$ satisfying
\[
\|f-f_N\|_2^2 \lesssim N^{-\min\{\beta,\alpha^{-1}\}+\varepsilon} \,, \quad \varepsilon >0 \text{ arbitrary} \,, \quad\text{as }N\to\infty.
\]
\end{cor}

\noindent
To see the reach of these results, let us mention that the $\alpha$-shearlet parametrization is $(\alpha,k)$-consistent with the $\alpha$-curvelet parametrization for $k>2$ (see~\cite[Thm.~5.7]{GKKS15}).
The results thus comprise in particular $\alpha$-shearlet frames, including both band-limited and compactly supported constructions (see~\cite[Prop.~3.11]{GKKS15}).


\begin{appendix}

\section[Appendix]{Bessel Functions}


In this appendix
we collect some useful facts about Bessel functions mainly taken from
\cite{Teubner1996} and \cite{Grafakos2008}.
We
are only interested in Bessel functions $J_\nu$ of integer and half-integer order in the range $\nu\in\{-\frac12,0,\frac12,1,\ldots\}$.
Bessel functions of this kind occur naturally in the Fourier analysis of radial functions. For $t\in\Rplus$ the value $J_\nu(t)$ is
conveniently
defined by either of the two series (see \cite{Teubner1996} and \cite[Appendix B.3]{Grafakos2008})
\begin{align}\label{app:defbessel}
J_{\nu}(t)=\Big(\frac{t}{2}\Big)^{\nu} \sum_{k=0}^\infty \frac{(-1)^k}{\Gamma(k+1)\Gamma(k+\nu+1)}\Big(\frac{t}{2}\Big)^{2k}
=\frac{1}{\sqrt{\pi}}
\Big(\frac{t}{2}\Big)^{\nu} \sum_{k=0}^\infty \frac{(-1)^k\Gamma(k+\frac{1}{2})}{\Gamma(k+\nu+1)}\frac{t^{2k}}{(2k)!}\,,
\end{align}
where the Gamma function $\Gamma$ extends the factorial $z!$ to the complex numbers with $\Gamma(z)=(z-1)!$.
To verify the equivalence of both representations, it is useful to note that $\Gamma(k+\frac12)=\frac{(2k)!}{k!4^k}\sqrt{\pi}$ for $k\in\N_0$.
We explicitly remark, that definition~\eqref{app:defbessel} is also valid for $\nu=-\frac12$, although this case is not included in the exposition of \cite{Grafakos2008}.
As is obvious from the second representation, the functions $J_\nu$ of half-integer order can be expressed in closed form in terms of trigonometric functions. For integer orders such closed form
representations do not exist.

If $f(x)=f_0(|x|)$ is a radial function on $\R^d$, $d\in\N$, with a suitable
function $f_0$ defined on $\R_0^+$, the Fourier transform of $f$ is given by the formula
\[
\widehat{f}(\xi)=\frac{2\pi}{|\xi|^{(d-2)/2}} \int_0^\infty f_0(r) J_{d/2-1}(2\pi r|\xi|)r^{d/2}\,dr \,,\quad \xi\in\R^d.
\]
Applying this formula to the characteristic function $\chi_{B_d(0,1)}$ of the $d$-dimensional unit ball $B_d(0,1)$ centered at the origin of $\R^d$ yields
\begin{align}\label{eq:Fourier1ball}
(\chi_{B_d(0,1)})^{\wedge}(\xi)
= \frac{2\pi}{|\xi|^{(d-2)/2}} \int_0^1 J_{d/2-1}(2\pi|\xi|r)r^{d/2}\,dr=
\frac{J_{d/2}(2\pi|\xi|)}{|\xi|^{d/2}} \,,\quad \xi\in\R^d.
\end{align}
Here, for the integration, we used the second of the following recurrence relations~\cite[Appendix B.2]{Grafakos2008}, which are valid for $\nu\in\frac{1}{2}\N$ and all $t\in\Rplus$,
\[
t^{-\nu+1} J_{\nu}(t)=-\frac{d}{dt} \big( t^{-\nu+1} J_{\nu-1}(t) \big) \quad\text{and}\quad t^\nu J_{\nu-1}(t)=\frac{d}{dt} \big( t^\nu J_\nu(t) \big)\,.
\]
The case $\nu=\frac12$ is not treated in~\cite{Grafakos2008}, yet it can be easily confirmed by a direct calculation.

By scaling, we can further deduce from \eqref{eq:Fourier1ball} the following Fourier representation of the bivariate function $\Theta(x)=\chi_{B_2(0,1)}(2x)$, $x\in\R^2$, from \eqref{def:Theta},
\begin{align}\label{eq:Fourier2ball}
\widehat{\Theta}(\xi)= \frac{1}{4} (\chi_{B_2(0,1)})^{\wedge}(\xi/2) =
\frac{J_{1}(\pi|\xi|)}{2|\xi|}, \quad\xi\in\R^2.
\end{align}
Important for our investigation in Section~3 is the asymptotic behavior of $J_{\nu}(r)$ as $r\rightarrow\infty$.
We cite the following result from \cite[Appendix B.8]{Grafakos2008}, which states for $\nu\in\frac{1}{2}\N_0$ the identity
\begin{align}\label{app:import1}
J_{\nu}(r)=\sqrt{\frac{2}{\pi r}}\cos(r-\frac{\pi\nu}{2}-\frac{\pi}{4}) + R_{\nu}(r) \,,\quad r\in\Rplus,
\end{align}
with a function $R_{\nu}$ given on $\Rplus$ by
\begin{align*}
R_{\nu}(r)&=\frac{(2\pi)^{-1/2}r^{\nu}}{\Gamma(\nu+1/2)} e^{i(r-\pi\nu/2-\pi/4)}
\int_0^\infty e^{-rt} t^{\nu+1/2} [(1+it/2)^{\nu-1/2}-1] \, \frac{dt}{t} \\
&\quad + \frac{(2\pi)^{-1/2}r^{\nu}}{\Gamma(\nu+1/2)} e^{-i(r-\pi\nu/2-\pi/4)}
\int_0^\infty e^{-rt} t^{\nu+1/2} [(1-it/2)^{\nu-1/2}-1] \, \frac{dt}{t}.
\end{align*}
Further, for each $\nu\in\frac{1}{2}\N_0$ there is a constant $C_\nu>0$ such that $R_{\nu}$ satisfies the estimate
\begin{align}\label{eqapp:est}
|R_\nu(r)|\le C_\nu r^{-3/2} \quad\text{whenever $r\ge1$}.
\end{align}
The representation~\eqref{app:import1} and the estimate~\eqref{eqapp:est} play an important role in the proof of Lemma~\ref{lem:thetaest}.
For completeness, let us finally note that the identity~\eqref{app:import1} especially holds true in case $\nu=-\frac12$, with vanishing $R_{-\frac12}\equiv0$. This is a direct
consequence of the definition~\eqref{app:defbessel} and the Taylor series of the cosine.

\end{appendix}

\section*{Acknowledgements}
The author acknowledges support by the BMS (Berlin Mathematical School) and thanks Prof.~Dr.~Gitta Kutyniok and Anton Kolleck
for proofreading the manuscript, as well as many helpful comments.


\begin{thebibliography}{10}

\bibitem{CDOS12}
J.~Cai, B.~Dong, S.~Osher, and Z. Shen.
\newblock
Image restoration: total variation, wavelet frames, and beyond.
\newblock
{\em J. Amer. Math. Soc.}, 25(4):1033--1089, 2012.

\bibitem{Can98}
E.~J.~Cand\`{e}s.
\newblock
{\em Ridgelets: theory and applications}.
\newblock
Ph.D.~thesis, Stanford University, CA, 1998.
\newblock Online available: {\url{http://statweb.stanford.edu/~candes/publications.html}}.

\bibitem{C99}
E.~J.~Cand\`{e}s.
\newblock {R}idgelets and the representation of mutilated {S}obolev functions.
\newblock {\em {SIAM} {J}. {M}ath. {A}nal.}, 33(2):347--368, 2001.

\bibitem{CD2000}
E.~J.~Cand\`es and D.~L.~Donoho.
\newblock
Curvelets -- a surprisingly effective nonadaptive representation for objects with edges.
\newblock
In C.~Rabut, A.~Cohen, and L.~Schumaker, editors, {\em Curves and Surfaces}, pages 105--120.
\newblock
Vanderbilt University Press, 2000.

\bibitem{CD04}
E.~J.~Cand\`es and D.~L.~Donoho.
\newblock
New tight frames of curvelets and optimal representations of objects with {$C^2$} singularities.
\newblock
{\em Comm. Pure Appl. Math.}, 57(2):219--266, 2004.

\bibitem{CWBB04tech}
V.~Chandrasekaran, M.~B.~Wakin, D.~Baron, and R.~G.~Baraniuk.
\newblock
Compressing piecewise smooth multidimensional functions using surflets: rate-distortion analysis.
\newblock
Technical report, Department of Electrical and Computer Engineering, Rice University, Mar.~2004.
\newblock Online available: \url{http://dsp.rice.edu/sites/dsp.rice.edu/files/publications/report/2004/compressin-riceece-2004.pdf}.

\bibitem{CWBB04proc}
V.~Chandrasekaran, M.~B.~Wakin, D.~Baron, and R.~G.~Baraniuk.
\newblock
Compression of higher dimensional functions containing smooth discontinuities.
\newblock
In {\em Conference on Information Sciences and Systems}, Princeton, Mar.~2004.

\bibitem{CWBB04c}
V.~Chandrasekaran, M.~B.~Wakin, D.~Baron, and R.~G.~Baraniuk.
\newblock
Surflets: a sparse representation for multidimensional functions containing smooth discontinuities.
\newblock
In {\em IEEE Symposium on Information Theory}, Chicago, Jul.~2004.

\bibitem{CWBB04b}
V.~Chandrasekaran, M.~B.~Wakin, D.~Baron, and R.~G.~Baraniuk.
\newblock
Representation and compression of multidimensional piecewise functions using surflets.
\newblock
{\em IEEE Trans. Inform. Theory}, 55(1):374--400, 2009.

\bibitem{CSE2000}
C.~Christopoulos, A.~Skodras, and T.~Ebrahimi.
\newblock
The {JPEG2000} still image coding system: an overview.
\newblock
{\em IEEE Trans. Consum. Electron.}, 46(4):1103--1127, 2000.

\bibitem{CDD01}
A.~Cohen, W.~Dahmen, and R.~DeVore.
\newblock
Adaptive wavelet methods for elliptic operator equations: convergence rates.
\newblock
{\em Math. Comp.}, 70(233):27--75, 2001.

\bibitem{Dau92}
I.~Daubechies.
\newblock
{\em Ten Lectures on Wavelets}.
\newblock
SIAM, Philadelphia, 1992.

\bibitem{Devore1998}
R.~A.~DeVore.
\newblock
Nonlinear approximation.
\newblock
{\em Acta Numerica}, 7:51--150, 1998.

\bibitem{DV05}
M.~N.~Do and M.~Vetterli.
\newblock
The contourlet transform: an efficient directional multiresolution
image representation.
\newblock
{\em IEEE Trans. Image Process.}, 14(12):2091--2106, 2005

\bibitem{D99}
D.~L.~Donoho.
\newblock
Wedgelets: nearly-minimax estimation of edges.
\newblock
{\em Ann. Statist.}, 27:859--897, 1999.

\bibitem{D98}
D.~L.~Donoho.
\newblock
Orthonormal ridgelets and linear singularities.
\newblock
{\em {SIAM} J. Math. Anal.}, 31(5):1062--1099, 2000.

\bibitem{DonXX}
D.~L.~Donoho.
\newblock
Ridge functions and orthonormal ridgelets.
\newblock
{\em J. Approx. Theory}, 111(2):143--179, 2001.

\bibitem{Don01}
D.~L.~Donoho.
\newblock
Sparse components of images and optimal atomic decompositions.
\newblock
{\em Constr. Approx.}, 17(3):353--382, 2001.

\bibitem{DH00}
D.~L.~Donoho and X.~Huo.
\newblock
Beamlet pyramids: a new form of multiresolution analysis suited for extracting lines, curves, and objects from very noisy image data.
\newblock
In {\em Wavelet Applications in Signal and Image Processing VIII (San Diego, CA, 2000), Proc. SPIE}, volume 4119, pages 434--444. SPIE, 2000.
\newblock

\bibitem{FS15}
A.~Flinth and M.~Sch\"{a}fer.
\newblock
Multivariate $\alpha$-molecules.
\newblock
{\em J. Approx. Theory}, 202:64--108, 2016.

\bibitem{Grafakos2008}
L.~Grafakos.
\newblock
{\em Classical Fourier Analysis}.
\newblock
Springer, 2nd edition, 2008.

\bibitem{GrohsRidLT}
P.~Grohs.
\newblock
Ridgelet-type frame decompositions for {S}obolev spaces related to linear transport.
\newblock
{\em J. Fourier Anal. Appl.}, 18(2):309--325, 2012.

\bibitem{GKKScurve2014}
P.~Grohs, S.~Keiper, G.~Kutyniok, and M.~Sch\"{a}fer.
\newblock
Cartoon approximation with $\alpha$-curvelets.
\newblock
{\em J. Fourier Anal. Appl.}, 22(6):1235–1293, 2016.

\bibitem{GKKS15}
P.~Grohs, S.~Keiper, G.~Kutyniok, and M.~Sch\"{a}fer.
\newblock
$\alpha$-{M}olecules.
\newblock
{\em Appl. Comput. Harmon. Anal.}, 41(1):297--336, 2016.

\bibitem{Grohs2011}
P.~Grohs and G.~Kutyniok.
\newblock
Parabolic molecules.
\newblock
{\em Found. Comput. Math.}, 14(2):299--337, 2014.

\bibitem{GOtech16}
P.~Grohs and A.~Obermeier.
\newblock
On the approximation of functions with line singularities by ridgelets.
\newblock
Technical Report 2016-4, Seminar for Applied Mathematics, ETH Z{\"u}rich, Switzerland, 2016.
\newblock Online available: {\url{http://www.sam.math.ethz.ch/sam_reports/reports_final/reports2016/2016-04_fp.pdf}}.

\bibitem{GO15}
P.~Grohs and A.~Obermeier.
\newblock
Optimal adaptive ridgelet schemes for linear advection equations.
\newblock
{\em Appl. Comput. Harmon. Anal.}, 41(3):768--814, 2016.

\bibitem{GKL05}
K.~Guo, G.~Kutyniok, and D.~Labate.
\newblock
Sparse multidimensional representations using anisotropic dilation and shear operators.
\newblock
In {\em Wavelets and Splines (Athens, GA, 2005)}, pages 189--201.
\newblock
Nashboro Press, Nashville, TN, 2006.

\bibitem{GL07}
K.~Guo and D.~Labate.
\newblock
Optimally sparse multidimensional representation using shearlets.
\newblock
{\em SIAM J. Math. Anal.}, 39(1):298--318, 2007.

\bibitem{Guo2012a}
K.~Guo and D.~Labate.
\newblock
The construction of smooth {P}arseval frames of shearlets.
\newblock
{\em Math. Model. Nat. Phenom.}, 8(1):82--105, 2013.

\bibitem{Teubner1996}
W.~Hackbusch, H.~R.~Schwarz, and E.~Zeidler.
\newblock
{\em Teubner-Taschenbuch der Mathematik}.
\newblock
B.~G.~Teubner Stuttgart, Leipzig, 1996.

\bibitem{Kei13}
S.~Keiper.
\newblock
A flexible shearlet transform -- sparse approximation and dictionary learning.
\newblock
Bachelor's thesis, TU Berlin, Germany, 2012.

\bibitem{Kittipoom2010}
P.~Kittipoom, G.~Kutyniok, and W.-Q~Lim.
\newblock
Construction of compactly supported shearlet frames.
\newblock
{\em Constr. Approx.}, 35(1):21--72, 2012.

\bibitem{K09}
J. Krommweh.
\newblock
Image approximation by adaptive tetrolet transform.
\newblock
In {\em International conference on sampling theory and applications}, Marseille, France, May 2009.

\bibitem{KuLaLiWe}
G.~Kutyniok, D.~Labate, W.-Q~Lim, and G.~Weiss.
\newblock
Sparse multidimensional representation using shearlets.
\newblock
In {\em Wavelets XI (San Diego, CA, 2005), SPIE Proc.}, volume 5914, pages 254--262.
\newblock
SPIE, Bellingham, WA, 2005.

\bibitem{Kutyniok2012correct}
G.~Kutyniok, J.~Lemvig, and W.-Q~Lim.
\newblock
Optimally sparse approximations of {3D} functions by compactly supported shearlet frames.
\newblock
{\em SIAM J. Math. Anal.}, 44(4):2962--3017, 2012.

\bibitem{Kutyniok2010}
G.~Kutyniok and W.-Q~Lim.
\newblock
Compactly supported shearlets are optimally sparse.
\newblock
{\em J. Approx. Theory}, 163(11):1564--1589, 2011.

\bibitem{PennecM}
E.~{Le Pennec} and S.~Mallat.
\newblock
Bandelet image approximation and compression.
\newblock
{\em Multiscale Model. Simul.}, 4(3):992--1039, 2005.

\bibitem{PM05}
E.~{Le Pennec} and S.~Mallat.
\newblock
Sparse geometric image representations with bandelets.
\newblock
{\em IEEE Trans. Image Process.}, 14(4):423--438, 2005.

\bibitem{LPS2016}
C.~Lessig, P.~Petersen, and M.~Schäfer.
\newblock
Bendlets: a second-order shearlet transform with bent elements.
\newblock
2016. submitted.
\newblock arXiv:1607.05520 [math.FA].

\bibitem{L11}
A.~Lisowska.
\newblock
Smoothlets -- multiscale functions for adaptive representation of images.
\newblock
{\em IEEE Trans. Image Process.}, 20(7):1777--1787, 2011.

\bibitem{L13}
A.~Lisowska.
\newblock
Multiwedgelets in image denoising.
\newblock
In J.~Park, J.~Ng, H.-Y.~Jeong, and B.~Waluyo, editors, {\em Multimedia and Ubiquitous Engineering: MUE 2013}, pages 3--11.
\newblock
Springer Netherlands, Dordrecht, 2013.

\bibitem{Mal08}
S.~Mallat.
\newblock
{\em A Wavelet Tour of Signal Processing: The Sparse Way}.
\newblock
Academic Press, 2nd edition, 2008.

\bibitem{M09}
S.~Mallat.
\newblock
Geometrical grouplets.
\newblock
{\em Appl. Comput. Harmon. Anal.}, 26(2):161--180, 2009.

\bibitem{WN03}
R.~M.~Willet and R.~D.~Nowak.
\newblock
Platelets: a multiscale approach for recovering edges and surfaces in photon-limited medical imaging.
\newblock
{\em IEEE Trans. Med. Imag.}, 22(3):332--350, 2003.

\end{thebibliography}
\end{document}